\documentclass[12pt]{amsart}
\usepackage{fullpage}
\usepackage{amsmath}
\usepackage{amssymb}
\usepackage{verbatim}
\usepackage[all]{xy}
\usepackage{color}
\usepackage{graphicx}
\usepackage{tikz} 
\usetikzlibrary{matrix,arrows} 

\newtheorem{Theorem}{Theorem}[section]

\newtheorem{Lemma}[Theorem]{Lemma}
\newtheorem{Proposition}[Theorem]{Proposition}
\newtheorem{Example}[Theorem]{Example}

\newtheorem{Corollary}[Theorem]{Corollary}

\newtheorem{Remark}{Remark}[section]
\theoremstyle{definition}

\theoremstyle{remark}

\numberwithin{equation}{section}
\newcommand{\oD}{ \!{\buildrel \circ 
\over D}}
\newcommand{\oDn}{ \!{\buildrel \circ 
\over D}^{_{_{_{\mbox{{\small $_{n}$}}}}}}}

\newcommand{\R}{\mathbb{R}}

\def\p{\partial}

\def\B{{\mathcal B}}

\def\P{{\mathcal P}}
\def\H{{\mathcal H}}
\def\T{{\mathcal T}}
\def\L{{\mathrm L}}
\def\R{{\mathrm R}}
\def\O{{\mathrm O}}
\def\Tr{{\mathcal T}}
\def\GTr{{\mathcal G}}
\def\In{{\mathrm I}{\mathrm n}}

\def\D{\mathcal{D}}

\def\H{{\mathcal H}}
\def\Riem{{\mathcal R}{\mathrm i}{\mathrm e}{\mathrm m}}

\def\closure{{\mathrm c}{\mathrm l}{\mathrm o}{\mathrm s}{\mathrm u}{\mathrm r}{\mathrm e}}

\def\NH{{\mathcal N}{\mathcal H}}

\def\head{\mathrm{h}\mathrm{e}\mathrm{a}\mathrm{d}}

\def\cyl{\mathrm{c}\mathrm{y}\mathrm{l}}

\def\lens{{\mathrm l}{\mathrm e}{\mathrm n}{\mathrm s}}

\def\torp{{\mathrm t}{\mathrm o}{\mathrm r}{\mathrm p}}

\def\Unc{\mathrm{U}\mathrm{n}\mathrm{c}}

\def\Cut{\mathrm{C}\mathrm{u}\mathrm{t}}
\def\Mov{\mathrm{M}\mathrm{o}\mathrm{v}}

\def\ann{\mathrm{A}\mathrm{n}\mathrm{n}}
\def\Fit{\mathrm{F}\mathrm{i}\mathrm{t}}

\def\ro{\mathrm{r}\mathrm{o}}
\def\bulb{\mathrm{b}\mathrm{u}\mathrm{l}\mathrm{b}}
\def\Bulb{\mathrm{B}\mathrm{u}\mathrm{l}\mathrm{b}}
\begin{document}

\title{H-Spaces, Loop Spaces and the Space of Positive Scalar Curvature Metrics on the Sphere}

\author{Mark Walsh}
\address{Department of Mathematics, Wichita State University}
\curraddr{Wichita, KS 67208}
\email{walsh@math.wichita.edu}
\thanks{}

\begin{abstract} For dimensions $n\geq 3$, we show that the space of metrics of positive scalar curvature on the sphere $S^{n}$, denoted $\Riem^{+}(S^{n})$, is homotopy equivalent to a subspace which takes the form of an $H$-space with a homotopy commutative, homotopy associative product operation. This product operation is based on the connected sum construction. We then exhibit an action on this subspace of the operad obtained by applying the bar construction to the little $n$-disks operad. Using results of Botvinnik, Boardman, Vogt and May we show that this implies, when $n=3$ or $n\geq 5$, that the space $\Riem^{+}(S^{n})$ is weakly homotopy equivalent to an $n$-fold loop space.
\end{abstract}

\maketitle
\section{Introduction}\label{intro}
This work is motivated by the problem of understanding the topology of the space of metrics of positive scalar curvature ({\em psc-metrics}) on the sphere $S^{n}$. This space is denoted $\Riem^{+}(S^{n})$ and is an open subspace of the space of all Riemannian metrics on $S^{n}$, $\Riem(S^{n})$, equipped with its standard smooth topology. It is known that when $n=2$, the space $\Riem^{+}(S^{n})$ is contractible; see \cite{RS}. When $n=3$, we know from a recent result of Marques that this space is path connected; see \cite{Marques}. In fact it is thought by experts that the space is contractible in this case also. When $n\geq 4$ however, the space $\Riem^+(S^n)$ is usually not path connected; see for example \cite{Carr}. Furthermore for $k\geq 1$, the groups $\pi_k(\Riem^+(S^n))$ are often non-trivial; see \cite{Hit}, \cite{Crowley-Schick} and \cite{Hanke-Schick-Steimle}. 
In this paper we make the following contribution.

\vspace{0.2cm}
\noindent {\bf Main Results.} 
{\em
\begin{enumerate}
\item[{\bf (i.)}] When $n\geq 3$, the space $\Riem^{+}(S^{n})$ is homotopy equivalent to a subspace which admits a homotopy product (i.e. is an $H$-space). Furthermore this product is homotopy commutative and homotopy associative.
\item[{\bf (ii.)}] When $n=3$ or $n\geq 5$, the space $\Riem^{+}(S^{n})$ is weakly homotopy equivalent to an $n$-fold loop space.
\end{enumerate}
}
\vspace{0.2cm}

\noindent Definitions of the terms $H$-space and Loop Space are given in section \ref{Hdefn}. We will not discuss the topological implications of such structure on $\Riem^{+}(S^{n})$ other than to point out that the condition that a topological space is an $H$-space, or especially an iterated loop space, imposes significant restrictions on its homotopy type. For more on this, see chapter 4 of \cite{Stash}. 

The main idea is as follows. We specify certain subspaces of $\Riem^{+}(S^{n})$ consisting of psc-metrics which take a ``standard form" near a fixed base point $p_0\in S^{n}$. It is known from results in \cite{Walsh3} that these subspaces are all homotopy equivalent to the space $\Riem^{+}(S^{n})$. We then construct products on these spaces, in the case where $n\geq 3$, based on the Gromov-Lawson connected sum construction in \cite{GL}. Roughly speaking, these products involve removing standard ``caps" around the point $p_0$ of the factor metrics and then taking a connected sum via some appropriate connecting metric. There is one important caveat. We need to ensure that the metric obtained by this product has a base point, something which the individual factors lose once we remove the standard caps. Hence we use, as an intermediary metric, a psc-metric on the sphere containing $3$ such standard caps, two of which will be removed for the attachments and the third which will be a base point cap. 
In all cases, we will show that this determines a homotopy product (i.e. makes the subspace an $H$-space) which is homotopy commutative and homotopy associative.

We then focus on one such subspace of $\Riem^{+}(S^{n})$: the space of psc-metrics which take the form of a round hemisphere of radius $1$ near a fixed base point $p_0$, denoted $\Riem_{D_{+}(1)}^{+}(S^{n})$. On this space, we show that the homotopy product generalises to an action of a certain operad. This operad is obtained via a process called the bar construction, from the {\em operad of little $n$-dimensional disks}; see below for a description of this object and the bar construction. It follows from results of Boardman, Vogt and May that a space $Z$ which admits such an action is weakly homotopy equivalent to an $n$-fold loop space, provided $Z$ is group-like (the induced multiplication on $\pi_{0}(X)$  gives it the structure of a group.) 

Thus, to demonstrate that $\Riem^{+}(S^{n})$ is weakly homotopic to an $n$-fold loop space, it remains to show that $\Riem_{D_{+}(1)}^{+}(S^{n})$ is group-like. Although our action is defined when $n\geq 3$, we show that $\Riem_{D_{+}(1)}^{+}(S^{n})$ is group-like only when $n=3$ or $n\geq 5$. In proving this, we make use of a recent theorem by Botvinnik, Theorem B. of \cite{Botvinnik}, concerning isotopy and concordance of psc-metrics. The hypothesis that $n=3$ or $n\geq 5$ stems from the fact that Botvinnik's theorem is false when $n=4$. Whether or not the same is true of our result is unknown.

\subsection {Organisation of the paper.} The paper is organised as follows. After recalling the defintions of an $H$-space and a loop space in section \ref{Hdefn}, we proceed in section \ref{disksphere} to describe two types of psc-metric on the disk which are well-behaved near the boundary. Roughly, these are metrics which are either cylindrical or sphere-like near the boundary and may be appropriately combined to obtain psc-metrics on the sphere. In later sections we make use of this in specifying product structures on certain subspaces of psc-metrics on the sphere which are standard near a fixed base point. In particular, in section \ref{torpconn}, we consider the most elementary of these subspaces: the space of psc-metrics on $S^{n}$ which have a ``torpedo cap" around some fixed base point $p_0\in S^{n}$. We denote this space $\Riem_{\torp(p_0)}^{+}(S^{n})$. After specifying a multiplication map on this space,  we review the Gromov-Lawson connected sum construction as generalised in \cite{Walsh1} for compact families of psc-metrics. In particular, we show in Lemma \ref{homequivspaces}, that the subspace of psc-metrics on $S^{n}$ which have torpedo caps at a fixed base point is actually homotopy equivalent, when $n\geq 3$, to the space of all psc-metrics on $S^{n}$. In section \ref{HSpaceThm3} we prove the first of our main results: the $H$-Space Theorem. This is Theorem \ref{HTheorem}, where we show that the multiplication map discussed above gives the space $\Riem_{\torp(p_0)}^{+}(S^{n})$ the structure of an  $H$-space. We also show that the product is both homotopy commutative and homotopy associative. Thus, $\Riem^{+}(S^{n})$ is homotopy equivalent to an $H$-space when $n\geq 3$. A minor consequence, Corollary \ref{AbFG}, is that the fundamental group of $\Riem^{+}(S^{n})$, with base point the standard round metric, is Abelian.

In section \ref{bulbsection}, we describe slightly more sophisticated subspaces of $\Riem^{+}(S^{n})$, consisting of psc-metrics which have the form of ``bulbs" and ``heads" near a base point $p_0\in S^{n}$. These spaces, denoted $\Riem_{\bulb(p_0)}^{+}(S^{n})$ and $\Riem_{\head(p_0)}^{+}(S^{n})$ are shown also to be homotopy equivalent to $\Riem^{+}(S^{n})$ in Lemma \ref{bulbdefret}. We then define multiplication maps, analogous to the one above, which give these spaces an $H$-space structure. In the case of $\Riem_{\head(p_0)}^{+}(S^{n})$, there is a deformation retract down to the subspace $\Riem_{D_{+}(1)}^{+}(S^{n})$, of psc-metrics which take the form of a round hemisphere of radius $1$ near $p_0$. On this space, we will show that the homotopy product generalises nicely to a certain operad action. Before doing this, we spend some time in section \ref{Operadsection} reviewing the various operads we will require. In particular, we consider the operad of little $n$-dimensional disks $\D_n$ (as well as a variant of this operad the round hemisphere) and the bar construction for operads. We recall relevant results of Boardman, Vogt and May: Theorems \ref{BVthm} and \ref{BVM}. These results allow us to conclude that the existence of an appropriate action of the operad $W\D_{n}$, obtained from $\D_{n}$ via the bar construction, on a group-like space $Z$ implies that $Z$ is weakly homotopy equivalent to an $n$-fold loop space. In section \ref{Dspace} we exhibit, for $n\geq 3$, precisely such an action on the space $\Riem_{D_{+}(1)}^{+}(S^{n})$. This is Lemma \ref{actionlemma}. Finally, in section \ref{grouplike}, we demonstrate that $\Riem_{D_{+}(1)}^{+}(S^{n})$ is indeed group-like in the case when $n=3$ or $n\geq5$; see Lemma \ref{pi0group}. Here we make great use of a recent theorem of Botvinnik from \cite{Botvinnik} concerning psc-concordance. This allows us to conclude the main result, Theorm \ref{LoopThm}.

\subsection{Acknowledgements.}I would like to thank Boris Botvinnik at the University of Oregon for suggesting this problem, Kirk Lancaster and Philip Parker at Wichita State University, David Wraith at NUI Maynooth, Ireland, and especially Victor Turchin at Kansas State University, for some helpful conversations. 

\section{$H$-Spaces and Loop spaces.}\label{Hdefn} A topological space $Z$ is a {\em $H$-space} if $Z$ is equipped with a continuous multiplication map $\mu:Z\times Z\rightarrow Z$ and an identity element $e\in Z$ so that the maps from $Z$ to $Z$ given by $x\mapsto \mu(x,e)$ and $x\mapsto \mu(e,x)$ are both homotopy equivalent to the identity map $x\mapsto x$. There are stronger versions of this definition where the above homotopies to the identity map are required to be homotopic through pointed maps $(Z,e)\rightarrow (Z,e)$ or where multiplication by the identity is the identity map. It is well known that in the case when $Z$ is homotopy equivalent to a CW complex, $Z$ admits a product which agrees with one of these definitions if and only if it admits products agreeing with the other two; see chapter 3.C of \cite{Hatcher}. Moreover, it follows from the work of Palais in \cite{Palais} that for any smooth compact manifold $X$, $\Riem^{+}(X)$ is homotopy equivalent to a CW complex. Thus, we feel justified in using the weaker definition.
An $H$-space $Z$ is said to be {homotopy commutative} if the maps $\mu$ and $\mu\circ\omega$, where $\omega:Z\times Z\rightarrow Z\times Z$ is the ``flip" map defined $\omega (x,y)=(y,x)$, are homotopy equivalent. Finally, $Z$ is a homotopy associative $H$-space if the maps from $Z\times Z\times Z$ to $Z$ given by $(x,y,z)\mapsto \mu(\mu(x,y),z)$ and $(x,y,z)\mapsto \mu(x, \mu(y,z))$ are homotopy equivalent. 

Given a topological space $Z$ with a prescribed base point $z_0\in Z$, we may consider the space of all loops based at $z_0$. This is the space of all continuous maps $\gamma: [0,1]\rightarrow Z$ so that $\gamma(0)=\gamma(1)=z_0$. This space is known as the {\em loop space of $Z$}, denoted $\Omega (Z, z_0)$, with base point the constant loop at $z_0$. Assuming the base point to be understood, we simply write $\Omega Z$. Repeated application of this construction yields the {k-th iterated loop space} $\Omega^{k}Z=\Omega(\Omega\cdots (\Omega Z)$ where at each stage the new base point is simply the constant loop at the old base point. We close by pointing out that a loop space is also an $H$-space with the multiplication determined by concatenation of loops. Whether or not a given $H$-space has the structure of a loop space is a more complicated problem concerning certain ``coherence" conditions on the homotopy associativity of the multiplication. It is a theorem of Stasheff that a space satisfies these conditions, is a so-called $A_{\infty}$-space, if and only if it is a loop space; see Theorem 4.18 in \cite{Stash}. The notion of an {\em operad} was constructed to more efficiently describe these coherence conditions, something we will return to in section \ref{Operadsection}.

\section{Metrics on the disk and sphere.}\label{disksphere}
For a smooth $n$-dimensional manifold $M$, possibly with non-empty boundary, we denote by $\Riem(M)$ the space of Riemannian metrics on $M$ equipped with its standard $C^{\infty}$-topology; see section 1.1 of \cite{Walsh1} for a description. Contained inside $\Riem(M)$ as an open subspace is the space of psc-metrics on $M$, denoted $\Riem^{+}(M)$. A path in this space is known as a {\em psc-isotopy} while metrics which lie in the same path component are said to be {\em psc-isotopic}.
In the case when $\p M\neq \emptyset$, it is common to consider only a subspace of $\Riem^{+}(M)$ of metrics which satisfy some constraint near the boundary. We will need such a constraint in this paper also and will return to this issue shortly.

We will mostly focus on the case when $M$ is either $D^{n}$ or $S^{n}$, the standard smooth disk or sphere of dimension $n$. Usually $n$ is assumed to be at least three. We denote by $ds_{n}^{2}$, the standard round metric of radius $1$ on $S^{n}$. As smooth topological objects, we model the disk $D^{n}=D^{n}(1)$ as the set of points $\{x\in\mathbb{R}^{n}:|x|\leq 1\}$ and the sphere $S^{n}=S^{n}(1)$ as the set $\{x\in\mathbb{R}^{n+1}:|x|=1\}$. In constructing metrics on these spaces, we will often work on an underlying disk or sphere $D^{n}(r)$ or $S^{n}(r)$ where the radius $r\neq 1$. The re-scaling function ${ x}\mapsto r{ x}$ gives a canonical way of pulling back metrics to the standard disk or sphere. Thus, we will often declare a metric which has been constructed on a general $D^{n}(r)$ or $S^{n}(r)$ to a be a metric on $D^{n}$ or $S^{n}$ assuming the metric to be pulled back in this way. Finally, we respectively denote by $D_{-}^{n}$ and $D_{+}^{n}$ the spaces $\{x\in\mathbb{R}^{n+1}:|x|=1, x_{n+1}\leq 0\}$ and $\{x\in\mathbb{R}^{n+1}:|x|=1, x_{n+1}\geq 0\}$, i.e. the southern and northern hemispheres of $S^{n}$. These hemispheres will be identified with the disk $D^{n}$ via the obvious map which sends geodesic rays emanating from the points $(0,0,\cdots, 0,\pm 1)$ to the corresponding ray on the flat disk $D^{n}(\frac{\pi}{2})$ followed by the above rescaling map. 

We now return to the question of boundary conditions on certain metrics. Our various constructions will involve attaching Riemannian disks along their boundaries to obtain new metrics on the sphere. On the smooth topological level, this involves gluing a pair disks together by identifying the boundary spheres via some diffeomorphism of $S^{n-1}$. In our case, the boundary sphere is canonically identified with the standard unit $(n-1)$-sphere in $\mathbb{R}^{n}$ and we always assume that we are gluing with the identity diffeomorphism. Of course, we need to ensure smooth attachment at the metric level. There are two ways of doing this which we will explore. The first is to work only with metrics which take the form of a round cylinder near or at least infinitesimally at, the boundary. The second is to consider metrics which, near their boundaries, agree with a geodesic ball from a round sphere (for example a hemisphere) near its boundary, at least infinitesimally; see Fig. \ref{stddiskmetrics} for a rough depiction. We will now describe these metrics in more detail.

\begin{figure}[!htbp]
\vspace{-0.5cm}
\hspace{3.0cm}
\begin{picture}(0,0)
\includegraphics{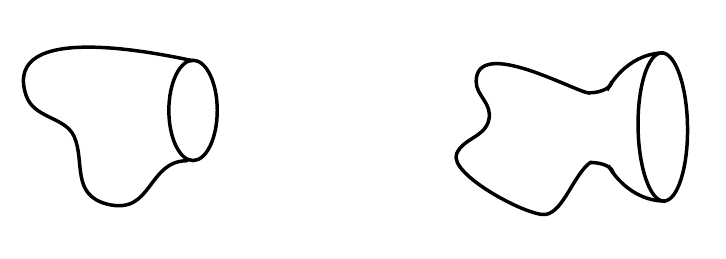}%
\end{picture}
\setlength{\unitlength}{3947sp}%
\begingroup\makeatletter\ifx\SetFigFont\undefined%
\gdef\SetFigFont#1#2#3#4#5{%
  \reset@font\fontsize{#1}{#2pt}%
  \fontfamily{#3}\fontseries{#4}\fontshape{#5}%
  \selectfont}%
\fi\endgroup%
\begin{picture}(5079,1559)(1902,-7227)
\end{picture}%
\caption{Metrics on the disk which are cylindrical (left) and spherical (right) near the boundary}
\label{stddiskmetrics}
\end{figure}    

\subsection{Metrics which are cylindrical near the boundary}\label{torprod} The usual method is to restrict ourselves to working with metrics which take the structure of a standard round cylinder near the boundary, at least infinitesimally. 
With this in mind we specify a subspace $\Riem_{\cyl(0)}^{+}(D^{n})\subset \Riem^{+}(D^{n})$ as follows. Beginning with the disk $D^{n}$, let $\epsilon>0$ and consider $D^{n}=D^{n}(1)$ as a submanifold of the disk of radius $1+\epsilon$, $D^{n}(1+\epsilon)$. Let $t$ denote the radial distance from the origin. Furthermore let $\ann(1, 1+\epsilon)$ denote the closure of the annulus $D^{n}(1+\epsilon)\setminus D^{n}(1)$. Let $\Riem_{{\cyl}(\epsilon)}^{+}(D^{n}(1+\epsilon))$ denote the space of psc-metrics on $D^{n}(1+\epsilon)$ defined as follows:
\begin{equation*}
\Riem_{{\cyl}(\epsilon)}^{+}(D^{n}(1+\epsilon)):=\{g\in \Riem^{+}(D^{n}(1+\epsilon)):g|_{\ann(1,1+\epsilon)}=dt^{2}+\delta^{2}ds_{n-1}^{2} {\text{ for some }} \delta>0\}.
\end{equation*}
We next consider the restriction map:
\begin{equation*}
\begin{split}
\Riem_{{\cyl}(\epsilon)}^{+}(D^{n}(1+\epsilon))&\longrightarrow \Riem^{+}(D^{n}(1)),\\
g&\longmapsto g|_{D^{n}(1)}.
\end{split}
\end{equation*}
We then define the space $\Riem_{{\cyl}(0)}^{+}(D^{n})$ to be the image of this restriction map. 

We now consider a pair of psc-metrics $g_0,g_1\in\Riem_{\cyl(0)}^{+}(D^{n})$. These metrics are very well behaved along the boundary. Indeed the only obstruction to simply gluing them together in the usual way is that the radii of their boundary spheres may not agree. There are two obvious ways we might proceed. The simplest is to simply rescale one or both of these metrics (multiplying the metric by an appropriate constant) so that the boundaries are compatible. Alternatively, one may wish to leave $g_0$ and $g_1$ unscathed and connect them via an appropriate warped round cylinder metric whose ends correspond to the respective boundaries. This second method seems a little cumbersome, but has the advantage that it does not require a global adjustment of $g_0$ or $g_1$. For our purposes however, the rescaling method will suffice.

We begin with an elementary fact. For any Riemannian metric $g$ on a smooth $n$-dimensional manifold $M$, the scalar curvature $R_{cg}$ of the metric $cg$ obtained by multiplying $g$ by a constant $c>0$ is given by the formula:
\begin{equation*}
R_{cg}=\frac{1}{c}R_{g}.
\end{equation*}
Thus, rescaling a psc-metric by a positive constant results in another psc-metric. We now define the function $\rho$, the {\em radius measuring map} as follows:
\begin{equation}\label{bdyrd}
\begin{split}
\rho:\Riem_{\cyl(0)}^{+}(D^{n})&\longrightarrow (0,\infty),\\
g& \longmapsto \rho(g)= {\text{Radius of sphere }} g|_{\p D^{n}}.
\end{split}
\end{equation}
Let $f$ be any function $f:(0,\infty)\times(0,\infty)\rightarrow (0,\infty)$. We now replace the metrics $g_0$ and $g_1$ respectively with the metrics: 
\begin{equation*}
\frac{f(\rho_0, \rho_1)^{2}}{{\rho_0}^{2}}g_0 \text{ and } \frac{f(\rho_0, \rho_1)^{2}}{{\rho_1}^{2}}g_1,
\end{equation*}
where $\rho_i=\rho(g_i)$ for $i=0,1$.
These replacement metrics are still elements of $\Riem_{\cyl(0)}^{+}(D^{n})$ but with boundary radii both equal to $f(\rho_0, \rho_1)$. We attach these metrics in the obvious way to obtain a new psc-metric on $S^{n}$ which we denote $g_0\cup_{f}g_1$; see Fig. \ref{rescaling}.
\begin{figure}[!htbp]
\vspace{-0.3cm}
\begin{picture}(0,0)
\includegraphics{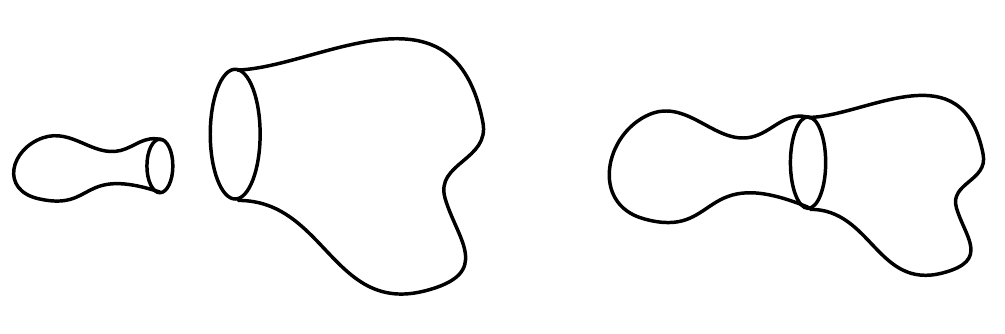}%
\end{picture}
\setlength{\unitlength}{3947sp}%
\begingroup\makeatletter\ifx\SetFigFont\undefined%
\gdef\SetFigFont#1#2#3#4#5{%
  \reset@font\fontsize{#1}{#2pt}%
  \fontfamily{#3}\fontseries{#4}\fontshape{#5}%
  \selectfont}%
\fi\endgroup%
\begin{picture}(5079,1559)(1902,-7227)
\put(2124,-6506){\makebox(0,0)[lb]{\smash{{\SetFigFont{10}{8}{\rmdefault}{\mddefault}{\updefault}{\color[rgb]{0,0,0}$g_0$}%
}}}}
\put(3384,-6506){\makebox(0,0)[lb]{\smash{{\SetFigFont{10}{8}{\rmdefault}{\mddefault}{\updefault}{\color[rgb]{0,0,0}$g_1$}%
}}}}
\put(5414,-6906){\makebox(0,0)[lb]{\smash{{\SetFigFont{10}{8}{\rmdefault}{\mddefault}{\updefault}{\color[rgb]{0,0,0}$g_0\cup_f g_1$}%
}}}}
\end{picture}%
\caption{The sphere metric $g_0\cup_f g_1$ (right) is formed by rescaling and gluing $g_0$ and $g_1$ (left)}
\label{rescaling}
\end{figure}
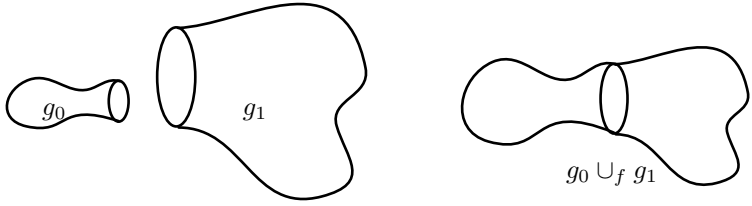   
\noindent Thus, for each $f$, the construction gives rise to a continuous {\em joining map}, $J^{\cyl(f)}$, defined:
\begin{equation}\label{joinmap}
\begin{split}
J^{\cyl(f)}:\Riem_{\cyl(0)}^{+}(D^{n})\times\Riem_{\cyl(0)}^{+}(D^{n})&\longrightarrow \Riem^{+}(S^{n})\\
(g_0, g_1)\longmapsto g_0\cup_f g_1.
\end{split}
\end{equation}
It is probably easiest to apply this construction when $f=\pi_\L$ or $\pi_\R$, the projection function onto the left or right factor. Here we write $J^{\L}$ or $J^{\R}$ to mean $J^{\cyl(\pi_\L)}$ or $J^{\cyl(\pi_\R)}$ respectively. Thus $J^{\L}$ fixes the size of the left input metric $g_0$ and rescales the right input $g_1$ while $J^{\R}$ fixes $g_1$ and rescales $g_0$. 

\subsection{Metrics which are sphere-like near the boundary}\label{lensection}
There is another approach to this problem, one which will be of great use to us later on. We begin with a round $n$-dimensional sphere of radius $\lambda$. We denote by $d_\lambda(a,b)$, the usual distance between points $a$ and $b$ on this sphere and by $B_\lambda(p,r)$, the closed geodesic ball of radius $r\in (0,\lambda\pi)$ about the $p\in S^{n}$. We identify $B_{\lambda}(p,r)$ with the northern hemisphere $D_{+}^{n}$ of the standard unit sphere in the following way. Move $p$, by the obvious rigid rotation of the $S^{n}$ along the great circle containing $p$ and the north pole, into the north pole position. We then rescale the sphere to make its radius $\lambda=1$. The ball $B_{\lambda}(p,r)$ is therefore replaced by the ball $B_1(p,\frac{r}{\lambda})$. Next, we identify $B_1(p,\frac{r}{\lambda})$ with the northern hemisphere $D_{+}^{n}$ by moving each point $x\in B_1(p,\frac{r}{\lambda})$, along the great circle through $p$ and $x$, to the point whose distance from $p$ is $\frac{\lambda\pi}{2r}d_1(p,x)$. All of this is depicted in Fig. \ref{lenscap}. 
Finally, we identify the northern hemisphere $D_{+}^{n}$ with the disk $D^{n}$ in the obvious way described at the beginning of this section.
By pulling back the restriction of the round metric of radius $\lambda$ to the ball $B_{\lambda}(p,r)$ via this composition of identifications, we obtain a metric on the disk $D^{n}$. This metric is known as the {\em $(\lambda,\epsilon)$-lens metric} on $D^{n}$ and denoted $g_{\lens}^{n}(\lambda, \epsilon)$. Note that, context permitting, we will sometimes refer to the ball $B_{\lambda}(p,r)$ as the {\em $(\lambda,r)$-lens} at $p$ also. Each lens $g_{\lens}^{n}(\lambda, r)$ has a {\em lens metric complement}, namely the metric $g_{\lens}^{n}(\lambda, \lambda\pi-r)$, which may attached to $g_{\lens}^{n}(\lambda, r)$ in the obvious way to reconstitute the round sphere metric of radius $\lambda$. 

\begin{figure}[!htbp]
\vspace{+1cm}
\hspace{-4.0cm}
\begin{picture}(0,0)
\includegraphics{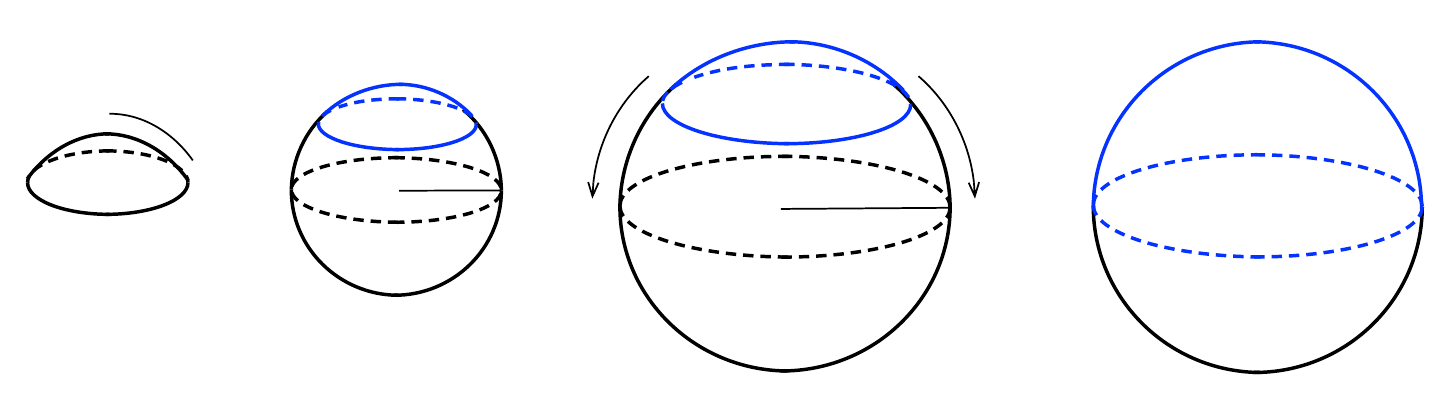}%
\end{picture}
\setlength{\unitlength}{3947sp}%
\begingroup\makeatletter\ifx\SetFigFont\undefined%
\gdef\SetFigFont#1#2#3#4#5{%
  \reset@font\fontsize{#1}{#2pt}%
  \fontfamily{#3}\fontseries{#4}\fontshape{#5}%
  \selectfont}%
\fi\endgroup%
\begin{picture}(5079,1559)(1902,-7227)
\put(2100,-6500){\makebox(0,0)[lb]{\smash{{\SetFigFont{10}{8}{\rmdefault}{\mddefault}{\updefault}{\color[rgb]{0,0,0}$g_{\lens}^{n}(\lambda,r)$}%
}}}}
\put(2600,-5800){\makebox(0,0)[lb]{\smash{{\SetFigFont{10}{8}{\rmdefault}{\mddefault}{\updefault}{\color[rgb]{0,0,0}$r$}%
}}}}
\put(3900,-6100){\makebox(0,0)[lb]{\smash{{\SetFigFont{10}{8}{\rmdefault}{\mddefault}{\updefault}{\color[rgb]{0,0,0}$\lambda$}%
}}}}
\put(3525,-5500){\makebox(0,0)[lb]{\smash{{\SetFigFont{10}{8}{\rmdefault}{\mddefault}{\updefault}{\color[rgb]{0,0,0}$B_{\lambda}(p,r)$}%
}}}}
\put(5250,-5300){\makebox(0,0)[lb]{\smash{{\SetFigFont{10}{8}{\rmdefault}{\mddefault}{\updefault}{\color[rgb]{0,0,0}$B_{1}(p,\frac{r}{\lambda})$}%
}}}}
\put(5900,-6180){\makebox(0,0)[lb]{\smash{{\SetFigFont{10}{8}{\rmdefault}{\mddefault}{\updefault}{\color[rgb]{0,0,0}$1$}%
}}}}
\put(7300,-5300){\makebox(0,0)[lb]{\smash{{\SetFigFont{10}{8}{\rmdefault}{\mddefault}{\updefault}{\color[rgb]{0,0,0}$D_{+}^{n}=B_{1}(p,\frac{\pi}{2})$}%
}}}}
\end{picture}%
\caption{The metric $g_{\lens}^{n}(\lambda, r)$ (left) and the rescaling to identify the ball $B_{\lambda}(p,r)$ with $D_{+}^{n}$ (right)}
\label{lenscap}
\end{figure}
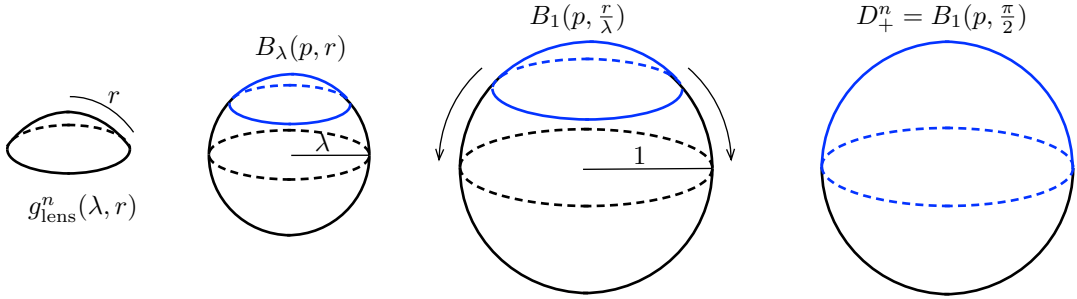   

In spherical coordinates on the disk $D^{n}(r)$, the metric $g_{\lens}^{n}(\lambda, r)$ takes the form:
\begin{equation*}
g_{\lens}^{n}(\lambda, r)=ds^{2}+\lambda^{2}\sin^{2}{\frac{s}{\lambda}}ds_{n-1}^{2},
\end{equation*}
where $s\in(0,r)$ denotes the radial distance coordinate. After pulling back to $D^{n}$ this metric then takes the form:
\begin{equation*}
g_{\lens}^{n}(\lambda, r)=r^{2}dt^{2}+\lambda^{2}\sin^{2}{\frac{rt}{\lambda}}ds_{n-1}^{2},
\end{equation*}
where $t\in(0,1]$ is the new radial distance coordinate. We consider the space of psc-metrics on $D^{n}(1+\epsilon)$ defined as follows. For each pair $\lambda>0, r\in(0,\lambda\pi)$ and each $\epsilon\in(0,\lambda{\pi}-r)$, let $\Riem_{{(\lambda,r)-\lens}(\epsilon)}^{+}(D^{n}(1+\epsilon))$ denote the space:
\begin{equation*}
\begin{split}
&\Riem_{{(\lambda,r)-\lens}(\epsilon)}^{+}(D^{n}(1+\epsilon)):=\\
&\{g\in \Riem^{+}(D^{n}(1+\epsilon)):g|_{\ann(1,1+\epsilon)}=r^2 dt^{2}+\lambda^{2}\sin^2{\frac{rt}{\lambda}}ds_{n-1}^{2}\},
\end{split}
\end{equation*}
where $t\in(0,1+\epsilon]$ here. As before, we consider the restriction map:
\begin{equation*}
\begin{split}
\Riem_{{(\lambda,r)-\lens}(\epsilon)}^{+}(D^{n}(1+\epsilon))&\longrightarrow \Riem^{+}(D^{n}(1)),\\
g&\longmapsto g|_{D^{n}(1)},
\end{split}
\end{equation*}
and define $\Riem_{{(\lambda,r)-\lens}(0)}^{+}(D^{n})$ to be the image of this map. Finally, we define $\Riem_{{\lens}(0)}^{+}(D^{n})$ to be the union, over all pairs $\lambda>0, r\in(0,\lambda\pi)$, of the spaces $\Riem_{{(\lambda,r)-\lens}(0)}^{+}(D^{n})$. Recall that our original motivation was in gluing disk metrics together to obtain metrics on the sphere. In this case, it is clear that elements of $\Riem_{{\lens}(0)}^{+}(D^{n})$ may be smoothly attached to other elements of $\Riem_{{\lens}(0)}^{+}(D^{n})$ provided their boundaries correspond to complementary lenses; see Fig. \ref{lenscapglue}. 
\begin{figure}[!htbp]
\vspace{2cm}
\hspace{-5.0cm}
\begin{picture}(0,0)
\includegraphics{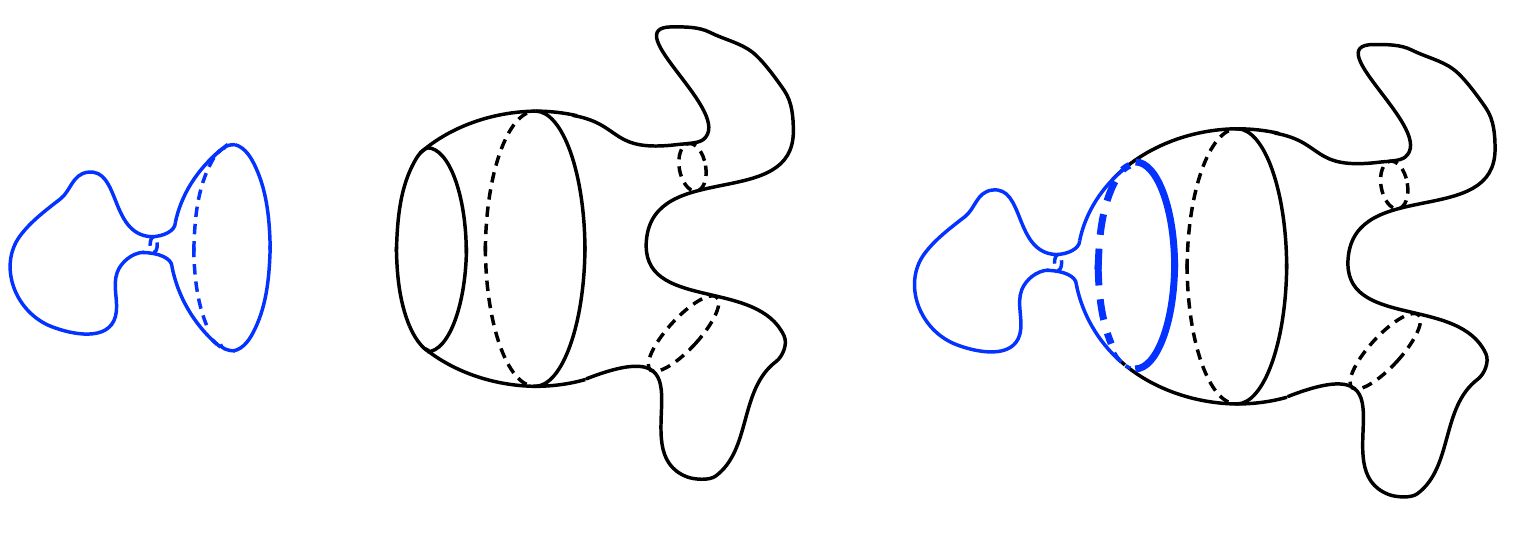}%
\end{picture}
\setlength{\unitlength}{3947sp}%
\begingroup\makeatletter\ifx\SetFigFont\undefined%
\gdef\SetFigFont#1#2#3#4#5{%
  \reset@font\fontsize{#1}{#2pt}%
  \fontfamily{#3}\fontseries{#4}\fontshape{#5}%
  \selectfont}%
\fi\endgroup%
\begin{picture}(5079,1559)(1902,-7227)

\end{picture}%
\caption{Combining metrics in $\Riem_{{\lens}(0)}^{+}(D^{n})$ which have complementary boundaries (left) gives rise to a psc-metric on $S^{n}$ (right)}
\label{lenscapglue}
\end{figure}   
\noindent It is possible to specify a map from $\Riem_{{\lens}(0)}^{+}(D^{n})\times\Riem_{{\lens}(0)}^{+}(D^{n})$ to $\Riem^{+}(S^{n})$, which is analogous to the map defined in \ref{joinmap}. Indeed, such a map will be very important for us. The construction in this case is more complicated and so we will postpone it until section \ref{bulbsection}.

\section{Torpedo Metrics and the Gromov-Lawson Construction}\label{torpconn}
We now turn our attention to the problem of combining pairs of psc-metrics on the sphere $S^{n}$ in order to obtain new psc-metrics, also on $S^{n}$. The best known example of this is the connected sum construction of Gromov and Lawson; see \cite{GL}. This is at the heart of our work. Essentially, for any metrics $g_0$ and $g_1$ on $S^{n}$ and provided $n\geq 3$, one may use this construction to obtain a new psc-metric on $S^{n}$ which is obtained by taking a geometric connected sum $g_0\# g_1$. We will shortly revisit this construction and so we will postpone the details until then. However, we should point out that the Gromov-Lawson construction requires we make a number of choices. Thus, it does not give rise to a well-defined binary operation on the space $\Riem^{+}(S^{n})$. It does in fact give rise to a mulitplication on certain quotient spaces of $\Riem^{+}(S^{n})$, such as the spaces of psc-concordance classes or psc-isotopy classes of metrics; this is something we will return to later on. There are ways, however, of achieving such an operation without taking a quotient, provide we restrict to certain subspaces of $\Riem^{+}(S^{n})$. Over the next two sections, we will spend time defining some of these subspaces as well as recalling the Gromov-Lawson construction. 

\subsection{Warped product metrics on the disk.} We begin by constructing a particular family of rotationally symmetric warped product metrics on the disk $D^{n}$. For our purposes, this is a metric on the disk which takes the form: 
\begin{equation}\label{warpmetric}
g^{\eta}=dt^{2}+{\eta}^{2}(t)ds_{n-1}^{2}
\end{equation}
where $t$ denotes the radial distance coordinate and where for some $b>0$, $\eta:[0,b)\rightarrow [0,\infty)$ is a smooth function satisfying:
\begin{enumerate}
\item[{\bf (i)}] $\eta(0)=\delta\sin(\frac{t}{\delta})$ for some $\delta>0$, when $t$ is near $0$,
\item[{\bf (ii)}] $\eta(t)>0$ when $t>0$. 
\end{enumerate} 
Although technically $dt^{2}+{\eta}^{2}(t)ds_{n-1}^{2}$ degenerates at $t=0$, the radius of the sphere factor closes in at the end $\{0\}\times S^{n-1}$ in such a way as to uniquely determine a smooth Riemannian metric on the disk $D^{n}$ which is rotationally symmetric with respect to the obvious action of orthogonal group $\O(n)$. This follows from the results of Chapter 1, Section 3.4 of \cite{P}. 
\begin{Remark}
We will intermittently regard $g^{\eta}$ as a metric on both $(0,\delta\frac{\pi}{2}]\times S^{n-1}$ and on $D^{n}$ depending on our circumstances. 
\end{Remark}
\noindent Finally, the scalar curvature $R$ of the metric $g^{\eta}$ at the point $(t, \theta)\in (0,\frac{\pi}{2}]\times S^{n-1}$ is given by the formula:
\begin{equation}\label{Rcurv}
R(t, \theta)=-2(n-1)\frac{{\eta}''(t)}{{\eta}(t)}+(n-1)(n-2)\frac{1-({\eta}'(t))^{2}}{{\eta}(t)^{2}}.
\end{equation}

\subsection{Infinitesimal torpedo metrics on the disk.}\label{inftorpmetric}
Recall that a torpedo metric on the disk $D^{n}$ is an $\O(n)$-symmetric metric which is round near the centre of the disk but transitions to a standard round Riemannian cylinder (neck) $[0,\epsilon]\times S^{n-1}$ near the boundary. For a detailed discussion of these metrics and their variants, see chapter 1 of \cite{Walsh1}. Roughly speaking, an infinitesimal torpedo metric takes this product structure only infinitesimally at the boundary of $D^{n}$. Importantly however, it smoothly attaches along the boundary to an end of a round cylinder $[0,b]\times S^{n-1}$. With this in mind, we fix a smooth function  $\eta_1:[0,\frac{\pi}{2}]\rightarrow[0,\infty)$ which satisfies the following requirements:
\begin{enumerate}
\item[{\bf (i)}] $\eta_1(t)=\sin{t}$, when $t$ is near $0$,
\item[{\bf (ii)}] $\eta_1(\frac{\pi}{2})=1$,
\item[{\bf (iii)}] ${\eta_1}''(t)<0$, when $0\leq t<\frac{\pi}{2}$,
\item[{\bf (iv)}] ${\eta_1}_{-}^{(k)}(\frac{\pi}{2})=0$, for all $k\geq 1$,
\end{enumerate} 
where ${\eta_1}_{-}^{(k)}$ represents the left sided $k$-th derivative of $\eta_1$. The graph of $\eta_1$ is depicted in Fig. \ref{inftorp} below. 
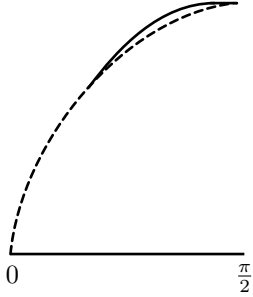
\begin{figure}[!htbp]
\vspace{1cm}
\hspace{5.5cm}
\begin{picture}(0,0)
\includegraphics{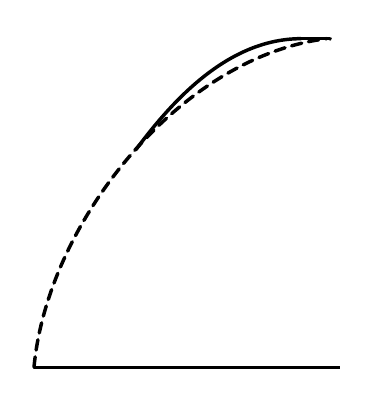}%
\end{picture}
\setlength{\unitlength}{3947sp}%
\begingroup\makeatletter\ifx\SetFigFont\undefined%
\gdef\SetFigFont#1#2#3#4#5{%
  \reset@font\fontsize{#1}{#2pt}%
  \fontfamily{#3}\fontseries{#4}\fontshape{#5}%
  \selectfont}%
\fi\endgroup%
\begin{picture}(5079,1559)(1902,-7227)
\put(1974,-7166){\makebox(0,0)[lb]{\smash{{\SetFigFont{10}{8}{\rmdefault}{\mddefault}{\updefault}{\color[rgb]{0,0,0}$0$}%
}}}}
\put(3414,-7166){\makebox(0,0)[lb]{\smash{{\SetFigFont{10}{8}{\rmdefault}{\mddefault}{\updefault}{\color[rgb]{0,0,0}$\frac{\pi}{2}$}%
}}}}ik
\end{picture}%
\caption{Comparing the graph of $\eta_1$ with the graph of the standard $\sin$ function represented by the dashed curve}
\label{inftorp}
\end{figure}   
Essentially, we want functions which behave like $\sin$ for the most part but end with all zero derivatives, as illustrated in Fig. \ref{inftorp} below. More generally, we obtain a family of functions $\{ \eta_\delta \}_{\delta>0}$ defined as follows:
\begin{equation*}
\begin{split}
\eta_\delta:[0,\delta\frac{\pi}{2}]&\longrightarrow [0,\infty)\\
t&\longmapsto \delta\eta_1(\frac{t}{\delta}).
\end{split}
\end{equation*}

Given any such function $\eta_\delta$ we obtain a metric $g^{\eta_\delta}=dt^{2}+{\eta_\delta}^{2}(t)ds_{n-1}^{2}$ as described above.
It is clear from formula \ref{Rcurv} and the conditions on the second derivative of $\eta_\delta$ that the scalar curvature of $g^{\eta_\delta}$ is always positive. (Recall we assume $n\geq 3$. When $n=2$ the best we can say is that $R\geq 0$.) We then obtain a space of metrics $\T_{\Riem}^{+}$, the subspace of $\Riem_{\cyl(0)}^{+}(D^{n})$ defined:
\begin{equation}\label{inftorpspace}
\T_{\Riem}^{+}:=\{g^{\eta_\delta}\in\Riem_{\cyl(0)}^{+}(D^{n}):\delta>0\},
\end{equation} 
where $g^{\eta_\delta}$ is given by formula \ref{warpmetric} above on $(0,\delta\frac{\pi}{2}]\times S^{n-1}$ but of course extends uniquely onto $D^{n}$. We make one final elementary observation concerning the fact that the restriction of the radius measuring map $\rho:\T_{\Riem}^{+}\rightarrow (0,\infty)$ is a bijection. 
\begin{Proposition}\label{welldefret}
For any constant $c>0$ and any $g^{\eta_\delta}\in \T_{\Riem}^{+}$, the metric $c^{2}g^{\eta_\delta}$ is exactly the element $g^{\eta_{c\delta}}\in \T_{\Riem}^{+}$.
\end{Proposition}
\begin{proof}
The element $g=dt^{2}+\eta_\delta(t)^{2}$ on $(0,\delta\frac{\pi}{2}]\times S^{n-1}$. Replacing $t$ with $\frac{s}{c}$ we see that $c^{2}g=ds^{2}+\eta_{{c}\delta}(s)^{2}ds_{n-1}^{2}$ on $(0,{c}\delta\frac{\pi}{2}]\times S^{n-1}$.
\end{proof}

\subsection{ A space of psc-metrics with torpedos} We now describe a subspace of $\Riem^{+}(S^{n})$ on which the construction described in section \ref{torprod} yields a product. When studying spaces of metrics on a manifold one often fixes a particular metric, called a reference metric, to be used to unambiguously specify coordinate balls or an exponential map. Although in principle the choice of reference metric does not matter, it is convenient in our case to use the standard round metric of radius $1$, $ds_n^{2}$, as a reference metric on $S^{n}$. Let $p$ be a point in $S^{n}$. A choice of orthonormal (with respect to $ds_{n}^{2}$) basis for the tangent space $T_{p}{S^{n}}$ to $S^{n}$ at $p$ gives rise to an isomorphism from $\mathbb{R}^{n}$ to $T_{p}{S^{n}}$. Composing this with the exponential map gives rise to a smooth pointed map $(\mathbb{R}^{n}, 0)\rightarrow (S^{n}, p)$ which restricts to an embedding on small disks around $0\in \mathbb{R}^{n}$. By precomposing with an appropriate rescaling we obtain an embedding of the standard unit disk $D^{n}$ into $S^{n}$. We will call this embedding $\phi_{p}$ and let $D_{p}$ denote its image in $S^{n}$. Finally, let $p'$ denote the antipodal point to $p$ on $S^{n}$, let $D_{p}'=\closure (S^{n}\setminus {D_{p}})$ and let $\phi_{p}':D^{n}\rightarrow D_{p}'$ denote the corresponding complementary embedding obtained from an appropriate restriction of the exponential map at $p'$. We now define a subspace of $\Riem^{+}(S^{n})$, which we denote $\Riem^{+}_{{\rm{torp}}(p_0)}(S^{n})$, as follows:
\begin{equation}\label{spacetorp}
\Riem^{+}_{{\rm{torp}}(p)}(S^{n})=\{g\in \Riem^{+}(S^{n}):{\phi_{p}^{*}}(g|_{D_{p}})\in\T_{\Riem}^{+}\},
\end{equation}
where, recall, $\T_{\Riem}^{+}$ is the space of infinitesimal torpedo metrics on $D^{n}$ defined in \ref{inftorpspace}.
This is known as the space of {\em psc-metrics with a torpedo at $p$}. Thus, each element of $\Riem^{+}_{{\rm{torp}}(p)}(S^{n})$ is a metric which has an infinitesimal torpedo-like ``cap" at the point $p$.
Furthermore, there is an ``uncapping" map $\Unc_{p}$, which removes this torpedo cap around $p$ by restricting such metrics to the complementary disk $D_{p}'$ (and then pulling back to the standard $D^{n}$). This map is defined as follows:
\begin{equation}\label{uncap}
\begin{split}
\Unc_{p}:\Riem^{+}_{{\rm{torp}}(p)}(S^{n})&\longrightarrow \Riem_{{\cyl}(0)}^{+}(D^{n})\\
g&\longmapsto (\phi_{p}')^{*}(g|_{D_{p}'}).
\end{split}
\end{equation}
Thus, $\Unc_{p}$ sends certain psc-metrics on $S^{n}$ to psc-metrics on $D^{n}$ with (infinitesimal) cylindrical boundaries.

Naively, in constructing a product on $\Riem^{+}_{{\rm{torp}}(p_0)}(S^{n})$, one might consider taking two psc-metrics with caps, removing the caps and then gluing them together after some appropriate rescaling. Using the joining map $J^{\cyl(f)}$ from \ref{joinmap} for some rescaling function $f:(0,\infty)\times (0,\infty)\rightarrow (0,\infty)$, and the uncapping map $\Unc_{p}$ defined above in \ref{uncap}, the composition map $J^{\cyl(f)}\circ(\Unc_{p}\oplus\Unc_{p})$ does precisely this. Unfortunately this produces a metric on a sphere with no base point (and thus no torpedo cap) and so a slightly more intricate multiplication is required. Before describing the more intricate construction, it is worth describing a version of this naive construction as it gets us most of the way there. 

We will begin with slight generalisation of the idea of a psc-metric with a torpedo cap. Suppose ${\bf p}=\{p_0,p_1,\cdots, p_k\}\subset S^{n}$ is a finite collection of points on $S^{n}$. We may specify around each of these points, closed disjoint normal coordinate neighbourhoods $D_{p_0}, \cdots D_{p_k}$ of the type described above, with corresponding diffeomorphisms $\phi_{p_i}:D^{n}\rightarrow D_{p_i}$ and complementary diffeomorphisms $\phi_{p_i}':D^{n}\rightarrow D_{p_i}'$ for each $i\in\{0,1,\cdots, k\}$. We now define the space $\Riem^{+}_{{\rm{torp}}({\bf p})}(S^{n})$ as follows:
\begin{equation*}
\Riem^{+}_{{\rm{torp}}({\bf p})}(S^{n})=\bigcap_{i=0}^{k}\Riem^{+}_{{\rm{torp}}(p_i)}(S^{n}).
\end{equation*}  
In Fig. \ref{cappysphere} below, we represent an element of $\Riem^{+}_{{\rm{torp}}({\bf p})}(S^{n})$ where ${\bf p}=\{p_0, p_1,p_2, p_3\}$ is a set of four distinct points on $S^{n}$.
\begin{figure}[!htbp]
\vspace{1.0cm}
\hspace{5.5cm}
\begin{picture}(0,0)
\includegraphics{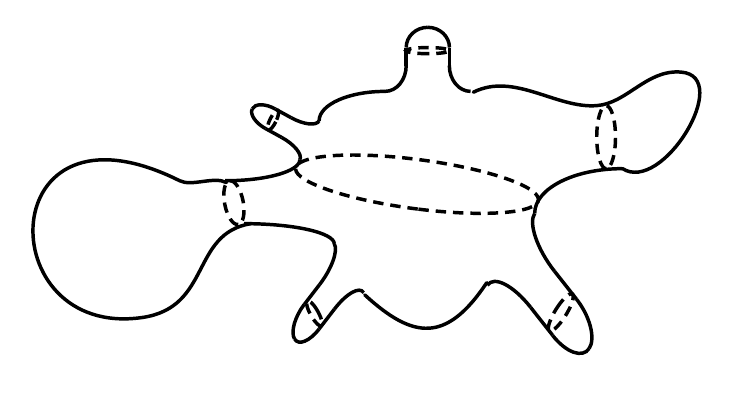}%
\end{picture}
\setlength{\unitlength}{3947sp}%
\begingroup\makeatletter\ifx\SetFigFont\undefined%
\gdef\SetFigFont#1#2#3#4#5{%
  \reset@font\fontsize{#1}{#2pt}%
  \fontfamily{#3}\fontseries{#4}\fontshape{#5}%
  \selectfont}%
\fi\endgroup%
\begin{picture}(5079,1559)(1902,-7227)
\put(2844,-5806){\makebox(0,0)[lb]{\smash{{\SetFigFont{10}{8}{\rmdefault}{\mddefault}{\updefault}{\color[rgb]{0,0,0}$p_3$}%
}}}}
\put(3824,-5306){\makebox(0,0)[lb]{\smash{{\SetFigFont{10}{8}{\rmdefault}{\mddefault}{\updefault}{\color[rgb]{0,0,0}$p_2$}%
}}}}
\put(3114,-7006){\makebox(0,0)[lb]{\smash{{\SetFigFont{10}{8}{\rmdefault}{\mddefault}{\updefault}{\color[rgb]{0,0,0}$p_0$}%
}}}}
\put(4614,-7106){\makebox(0,0)[lb]{\smash{{\SetFigFont{10}{8}{\rmdefault}{\mddefault}{\updefault}{\color[rgb]{0,0,0}$p_1$}%
}}}}
\end{picture}%
\caption{A sphere with four torpedo caps}
\label{cappysphere}
\end{figure}
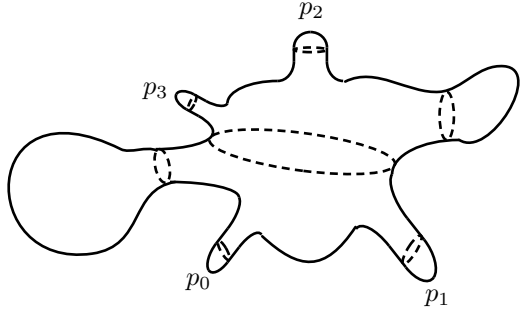   
We will now make a couple of technical observations about the space $\Riem^{+}_{{\rm{torp}}({\bf p})}(S^{n})$. Firstly, the choice of orthonormal basis for each tangent space is unimportant.
\begin{Lemma}\label{caponb}
For $n\geq 3$, the subspace $\Riem^{+}_{{\rm{torp}}({\bf p})}(S^{n})\subset\Riem^{+}(S^{n})$ remains fixed if we vary the choice of orthonormal basis for $T_{p_{i}}S^{n}$ for each $p_i\in{\bf p}$.
\end{Lemma}
\begin{proof}
This essentially follows from the rotational symmetry of the caps. For a detailed proof, see \cite{Walsh2}.
\end{proof}
\noindent For a single point $p$, the topology of $\Riem^{+}_{{\rm{torp}}(p)}(S^{n})$ is also unaffected by the choice of $p$. In particular we have the following lemma.
\begin{Lemma}\label{caprot}
For $n\geq 3$ and for any $p,q\in S^{n}$, the spaces $\Riem^{+}_{{\rm{torp}}(p)}(S^{n})$ and $\Riem^{+}_{{\rm{torp}}(q)}(S^{n})$ are homeomorphic.
\end{Lemma}
\begin{proof}
Let $\ro_{(p,q)}$ denote the rotation from $p$ to $q$ along the great circle (with respect to the standard round metric) containing $p$ and $q$. Then the map which sends a metric $g$ on $S^{n}$ to the pull back metric $\ro_{(p,q)}^{*}g$ defines a homeomorphism from $\Riem^{+}_{{\rm{torp}}(q)}(S^{n})$ to $\Riem^{+}_{{\rm{torp}}(p)}(S^{n})$.
\end{proof}
Remaining a little longer with the case when ${\bf p}$ is a single point $p$ in $S^{n}$, it is worth pointing out that this space may be simplified somewhat by considering only torpedo caps of a fixed radius. For any $\delta>0$, let $\Riem^{+}_{{\rm{torp}}(p,\delta)}(S^{n})$ be the subspace of $\Riem^{+}_{{\rm{torp}}({p})}(S^{n})$ which is defined as follows:
\begin{equation*}
\Riem^{+}_{{\rm{torp}}(p,\delta)}(S^{n}):=\{g\in \Riem^{+}_{{\rm{torp}}(p)}(S^{n})\}:\rho(g|_{D_{p}})=\delta\}.
\end{equation*} 
Equivalently, this is the space of psc-metrics on $S^{n}$ so that $\phi_{p}^{*}(g|_{D_{p}})=g^{\eta_{\delta}}$, or more simply the space of psc-metrics on $S^{n}$ with a torpedo cap of radius $\delta$ about the point $p$. As the following lemma shows, up to homotopy, this space is no different from $\Riem^{+}_{{\rm{\torp}}({p})}(S^{n})$.
\begin{Lemma}\label{spheredefret}
For $n\geq 3$ and any $\delta>0$,  there is a deformation retract from the space $\Riem^{+}_{{\rm{torp}}({p})}(S^{n})$ onto its subspace $Riem^{+}_{{\rm{torp}}(p,\delta)}(S^{n})$.
\end{Lemma}
\begin{proof}
Let $i:\Riem^{+}_{{\rm{torp}}(p,\delta)}(S^{n})\hookrightarrow \Riem^{+}_{{\rm{torp}}(p)}(S^{n})$ denote the inclusion map.  
For each $s\in [0,1]$, let $r_s$ be the map defined:
\begin{equation*}
\begin{split}
r_s:\Riem^{+}_{{\rm{torp}}(p)}(S^{n})&\longrightarrow \Riem^{+}_{{\rm{torp}}(p)}(S^{n})\\
g&\longmapsto \frac{{\delta}^{2}g}{((1-s)\rho(g|_{D_p})+s{\delta})^{2}}.
\end{split}
\end{equation*}
It follows from Proposition \ref{welldefret} that for each $s\in [0,1]$, this map is well defined. Furthermore, it is immediate that $r_1$ is the identity map on $\Riem^{+}_{{\rm{torp}}(p)}(S^{n})$, $r_0$ maps into $\Riem^{+}_{{\rm{torp}}(p,\delta)}(S^{n})$ and that the composition $r_0\circ i$ is the identity map on 
$\Riem^{+}_{{\rm{torp}}(p,\delta)}(S^{n})$.
\end{proof}
\begin{Remark}
With somewhat more sophisticated tools, such as those utilised in the Gromov-Lawson construction as described in \cite{Walsh1}, one could prove a more general version of this homotopy equivalence for psc-metrics with multiple torpedo caps.  
\end{Remark}
\noindent A little later we will revisit the fact, proved in \cite{Walsh3}, that when $n\geq 3$, the subspace $\Riem^{+}_{\rm{torp}({\bf p})}(S^{n})$ is homotopy equivalent to the space $\Riem^{+}(S^{n})$. Now however, we will return to the problem of defining a product on the space $\Riem^{+}_{\rm{torp}({p})}(S^{n})$ for the case when $\bf p$ is a single point $p$. To do this we will need to spend a little more time on the more general case where $\bf p$ contains several points.

Let ${\bf p}=\{p_0,p_1,\cdots p_k\}$ and ${\bf q}=\{q_0, q_1, \cdots, q_l\}$ be two finite sets of points on $S^{n}$. We will assume that $p_i\neq p_j$ when $i\neq j$ and that $q_i\neq q_j$ when $i\neq j$ but make no assumptions about whether or not $p_i=q_j$. We now consider the corresponding spaces $\Riem_{{\rm torp}({\bf p})}^{+}(S^{n})$ and $\Riem_{{\rm torp}({\bf q})}^{+}(S^{n})$. For each pair of integers $(i,j)$ with $0\leq i\leq k$ and $0\leq j\leq l$ we define the {\em $ij$-uncapping map} $Unc_{ij}$ as follows:
\begin{equation}\label{ijuncap}
\begin{split}
\Unc_{ij}:\Riem_{{\rm torp}({\bf p})}^{+}(S^{n})\times \Riem_{{\rm torp}({\bf q})}^{+}(S^{n})&\longrightarrow \Riem_{\cyl(0)}^{+}(D^{n})\times \Riem_{\cyl(0)}^{+}(D^{n}), \\
(g, h)&\longmapsto  (\Unc_{p_i}(g), \Unc_{q_j}(h)).
\end{split}
\end{equation}
In simple terms, the map $\Unc_{ij}$ removes the torpedo caps at $p_i$ and $q_j$ on the metrics $g$ and $h$ respectively. After an appropriate rescaling, the resulting disks can be glued together along their boundaries. Thus, for each map $f:(0,\infty)\times(0,\infty)\rightarrow (0,\infty)$, we obtain the {\em $ij$-joining map}, $J_{ij}^{\cyl(f)}$, defined as follows:

\begin{equation}\label{Jij}
\begin{split}
J_{ij}^{\cyl(f)}:\Riem_{{\rm torp}({\bf p})}^{+}(S^{n})\times \Riem_{{\rm torp}({\bf q})}^{+}(S^{n})&\longrightarrow \Riem_{{\rm torp}({\{{\bf p}\setminus\{p_{i}\}\}}\cup{\{{\bf q}\setminus \{q_j\}})\}}^{+}(S^{n}),\\
(g,h)&\longmapsto J^{\cyl(f)}(\Unc_{ij}(g,h)),
\end{split}
\end{equation}
where $J^{\cyl(f)}$ is the map defined in \ref{joinmap}. Henceforth, we will usually suppress the function $f$ and simply write $J$ for $J^{\cyl(f)}$ and $J_{ij}$ for $J_{ij}^{\cyl(f)}$, knowing that some $f$ is fixed in the background. This construction is illustrated schematically (for some unspecified $f$) in Fig. \ref{cappyspherejoin} below where ${\bf p}=\{p_0, p_1, p_2, p_3\}$, ${\bf q}=\{q_0, q_1, q_2\}$, $i=1$ and $j=2$.
\begin{figure}[!htbp]
\vspace{2.0cm}
\hspace{-3.0cm}
\begin{picture}(0,0)
\includegraphics{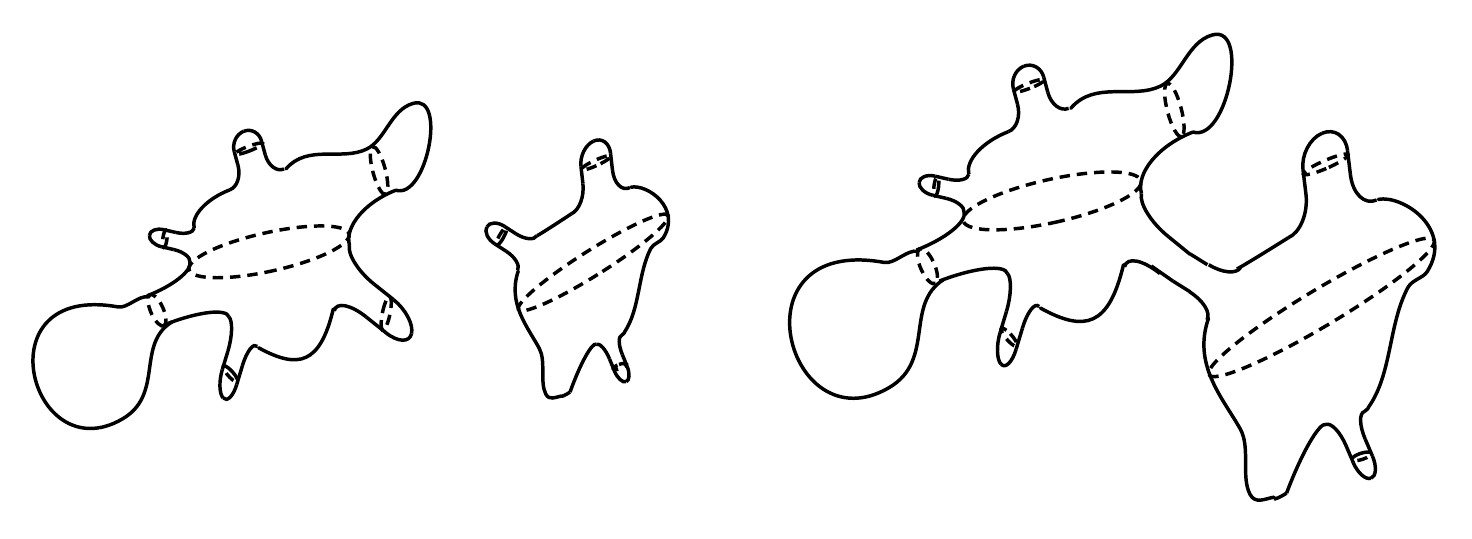}%
\end{picture}
\setlength{\unitlength}{3947sp}%
\begingroup\makeatletter\ifx\SetFigFont\undefined%
\gdef\SetFigFont#1#2#3#4#5{%
  \reset@font\fontsize{#1}{#2pt}%
  \fontfamily{#3}\fontseries{#4}\fontshape{#5}%
  \selectfont}%
\fi\endgroup%
\begin{picture}(5079,1559)(1902,-7227)
\put(2404,-5706){\makebox(0,0)[lb]{\smash{{\SetFigFont{10}{8}{\rmdefault}{\mddefault}{\updefault}{\color[rgb]{0,0,0}$p_3$}%
}}}}
\put(2824,-5206){\makebox(0,0)[lb]{\smash{{\SetFigFont{10}{8}{\rmdefault}{\mddefault}{\updefault}{\color[rgb]{0,0,0}$p_2$}%
}}}}
\put(2764,-6646){\makebox(0,0)[lb]{\smash{{\SetFigFont{10}{8}{\rmdefault}{\mddefault}{\updefault}{\color[rgb]{0,0,0}$p_0$}%
}}}}
\put(3714,-6406){\makebox(0,0)[lb]{\smash{{\SetFigFont{10}{8}{\rmdefault}{\mddefault}{\updefault}{\color[rgb]{0,0,0}$p_1$}%
}}}}
\put(4774,-6596){\makebox(0,0)[lb]{\smash{{\SetFigFont{10}{8}{\rmdefault}{\mddefault}{\updefault}{\color[rgb]{0,0,0}$q_0$}%
}}}}
\put(4600,-5206){\makebox(0,0)[lb]{\smash{{\SetFigFont{10}{8}{\rmdefault}{\mddefault}{\updefault}{\color[rgb]{0,0,0}$q_1$}%
}}}}
\put(4000,-5806){\makebox(0,0)[lb]{\smash{{\SetFigFont{10}{8}{\rmdefault}{\mddefault}{\updefault}{\color[rgb]{0,0,0}$q_2$}%
}}}}

\end{picture}%
\caption{The metric $J_{12}(g,h)$ (right) formed by combining $g$ and $h$ (left)}
\label{cappyspherejoin}
\end{figure}
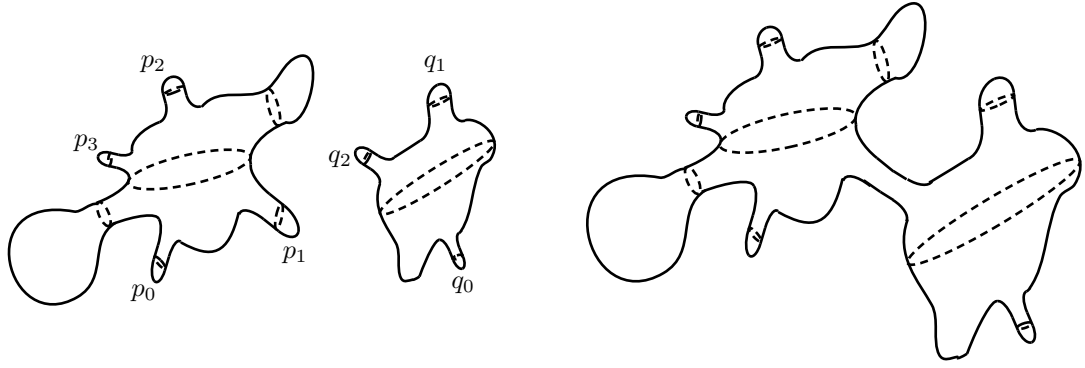   

In the case when ${\bf p}={\bf q}=\{p_0\}$, the image of the map $J_{00}$ does not lie in $\Riem_{{\rm torp}({ p_0})}^{+}(S^{n})$. Thus, to define any sort of product we need to do some additional work. Henceforth we fix a {\em base point} $p_0\in S^{n}$. Let $p_1$ and $p_2$ be two distinct points on $S^{n}\setminus\{p_0\}$ and let ${\bf p}=\{p_0, p_1, p_2\}$. We now define a product on $\Riem_{{\rm torp}({ p_0})}^{+}(S^{n})$ as follows. Consider for each $j=1,2$, the map:
\begin{equation*}
J_{0j}:\Riem_{{\mathrm{torp}}({p_0 })}^{+}(S^{n})\times \Riem_{{\mathrm{torp}}({\bf p})}^{+}(S^{n})\longrightarrow \Riem_{{\mathrm{torp}}({\bf p}\setminus\{p_j\})}^{+}(S^{n}),
\end{equation*}
defined as in formula \ref{Jij}. Suppose we fix the second input metric as some $g_3\in\Riem_{{\rm torp}({\bf p})}^{+}(S^{n})$. Then for each of $j=1,2$, we obtain maps: 
\begin{equation}
\begin{split}
J_{0j}^{3}:\Riem_{{\mathrm{torp}}({p_0 })}^{+}(S^{n})&\longrightarrow \Riem_{{\mathrm{torp}}({\bf p}\setminus\{p_j\})}^{+}(S^{n}),\\
g&\longmapsto J_{0j}(g, g_3).
\end{split}
\end{equation}
Finally, we define a product on $\Riem_{{\mathrm{torp}}({p_0 })}^{+}(S^{n})$ by means of the following continuous map:
\begin{equation}\label{muprod1}
\begin{split}
\mu^{\torp}:\Riem_{{\mathrm{torp}}({p_0 })}^{+}(S^{n})\times \Riem_{{\mathrm{torp}}({p_0 })}^{+}(S^{n})&\longrightarrow \Riem_{{\mathrm{torp}}({p_0 })}^{+}(S^{n}),\\
(g,h)&\longmapsto J_{02}(h,J_{01}(g, g_3)).
\end{split}
\end{equation}

\begin{Example}\label{g3eg}
{\rm 
As an illustration, consider the metric $\bar{g}_{\lambda}$ obtained by gluing two copies the infinitesimal torpedo metric on the disk $D^{n}$, $g^{\eta_{\lambda}}$, together along the boundary in the obvious way. This metric metric is represented in Fig. \ref{Egg} below.
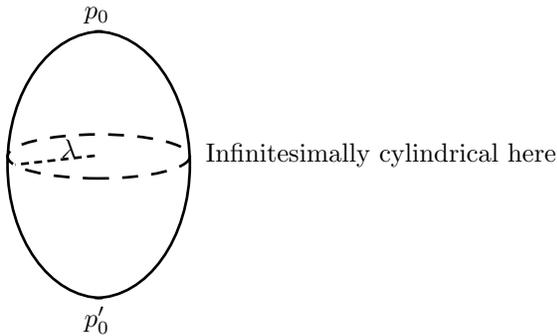
\begin{figure}[!htbp]
\vspace{1.5cm}
\hspace{5.0cm}
\begin{picture}(0,0)
\includegraphics{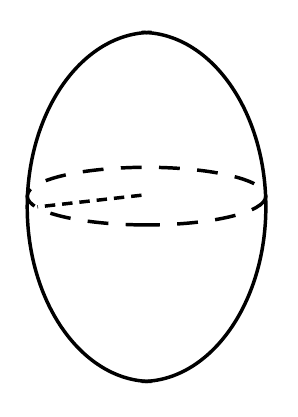}%
\end{picture}
\setlength{\unitlength}{3947sp}%
\begingroup\makeatletter\ifx\SetFigFont\undefined%
\gdef\SetFigFont#1#2#3#4#5{%
  \reset@font\fontsize{#1}{#2pt}%
  \fontfamily{#3}\fontseries{#4}\fontshape{#5}%
  \selectfont}%
\fi\endgroup%
\begin{picture}(5079,1559)(1902,-7227)
\put(2304,-6176){\makebox(0,0)[lb]{\smash{{\SetFigFont{10}{8}{\rmdefault}{\mddefault}{\updefault}{\color[rgb]{0,0,0}$\lambda$}%
}}}}
\put(2450,-5306){\makebox(0,0)[lb]{\smash{{\SetFigFont{10}{8}{\rmdefault}{\mddefault}{\updefault}{\color[rgb]{0,0,0}$p_0$}%
}}}}
\put(2450,-7240){\makebox(0,0)[lb]{\smash{{\SetFigFont{10}{8}{\rmdefault}{\mddefault}{\updefault}{\color[rgb]{0,0,0}$p_0'$}%
}}}}
\put(3214,-6206){\makebox(0,0)[lb]{\smash{{\SetFigFont{10}{8}{\rmdefault}{\mddefault}{\updefault}{\color[rgb]{0,0,0}{\rm Infinitesimally cylindrical here}}%
}}}}
\end{picture}%
\caption{The metric $\bar{g}_\lambda$}
\label{Egg}
\end{figure}   
Strictly speaking this is a metric with two torpedo caps, one at the north pole which we denote by $p_0$ and one at the south pole $p_0'$. For simplicity, we will just assume that $\lambda =1$. Now choose two points $p_1$ and $p_2$ in the interior of the southern hemisphere. For simplicity, we may as well choose $p_1$ and $p_2$ so that $p_0, p_1$ and $p_2$ are equidistant along a great circle. As the dimension of the underlying sphere $n$ is at least three, it is possible to ``push out" two torpedo caps of radius $\delta>0$ (for some sufficiently small $\delta>0$) at each of the points $p_1$ and $p_2$ to construct the psc-metric $g_3^{m}$ illustrated in Fig. \ref{g2} below. This follows from the work of Gromov and Lawson in \cite{GL} and is something we will discuss in more detail in the next section. Finally, in Fig. \ref{minemetricjoin2}, we depict the result of multiplying a pair of metrics $g,h\in\Riem_{\torp(p_0)}$ via the multiplication $\mu^{\torp}$ in \ref{muprod1}, where $g_3=g_{3}^{m}$, the metric constructed above. 
\begin{figure}[!htbp]
\vspace{1.5cm}
\hspace{5.5cm}
\begin{picture}(0,0)
\includegraphics{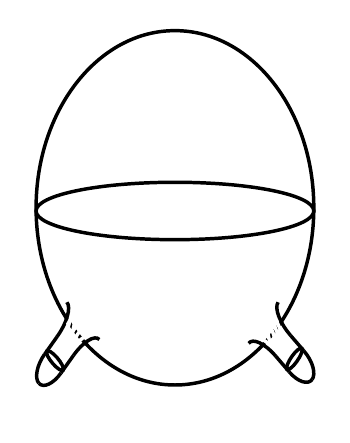}%
\end{picture}
\setlength{\unitlength}{3947sp}%
\begingroup\makeatletter\ifx\SetFigFont\undefined%
\gdef\SetFigFont#1#2#3#4#5{%
  \reset@font\fontsize{#1}{#2pt}%
  \fontfamily{#3}\fontseries{#4}\fontshape{#5}%
  \selectfont}%
\fi\endgroup%
\begin{picture}(5079,1559)(1902,-7227)
\put(2550,-5150){\makebox(0,0)[lb]{\smash{{\SetFigFont{10}{8}{\rmdefault}{\mddefault}{\updefault}{\color[rgb]{0,0,0}$p_0$}%
}}}}
\put(2000,-7161){\makebox(0,0)[lb]{\smash{{\SetFigFont{10}{8}{\rmdefault}{\mddefault}{\updefault}{\color[rgb]{0,0,0}$p_1$}%
}}}}
\put(3300,-7161){\makebox(0,0)[lb]{\smash{{\SetFigFont{10}{8}{\rmdefault}{\mddefault}{\updefault}{\color[rgb]{0,0,0}$p_2$}%
}}}}
\end{picture}%
\caption{The metric $g_3^{m}$ with three torpedo caps}
\label{g2}
\end{figure}
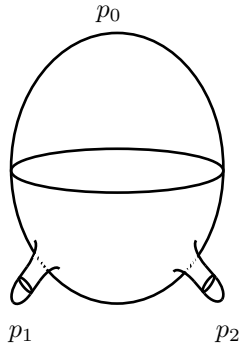   
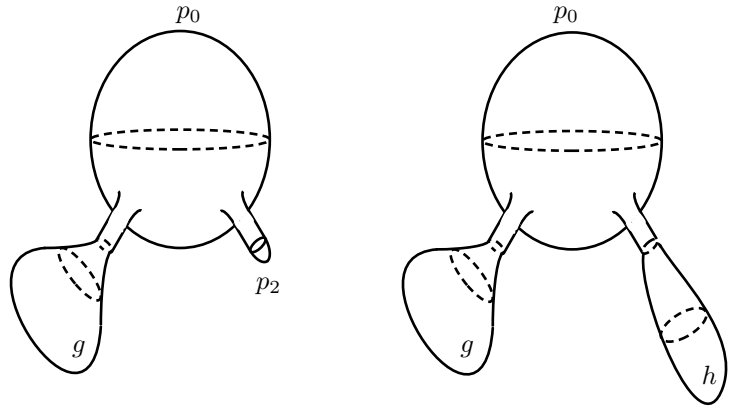
\begin{figure}[!htbp]
\vspace{3cm}
\hspace{1.0cm}
\begin{picture}(0,0)
\includegraphics{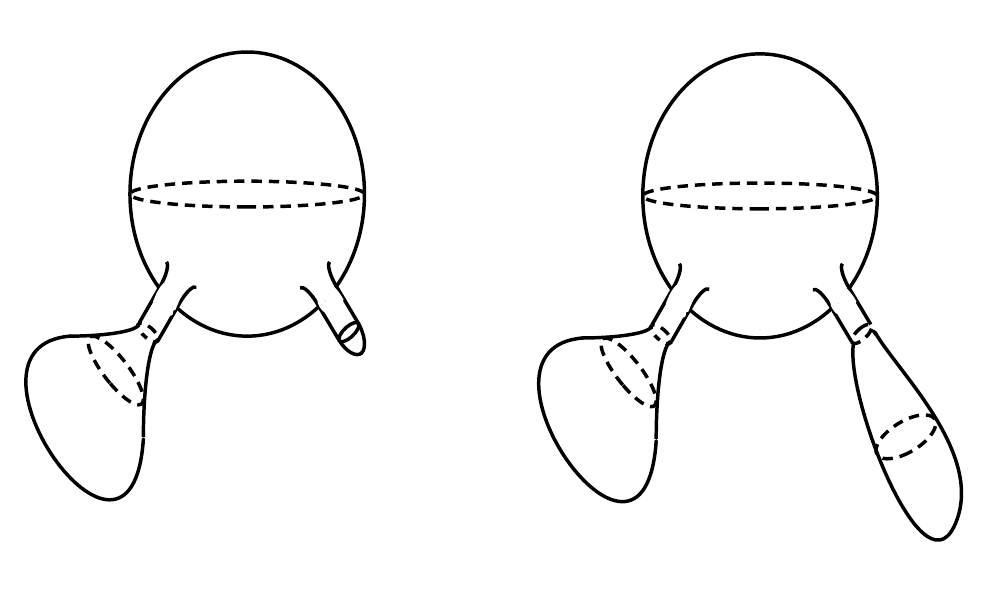}%
\end{picture}
\setlength{\unitlength}{3947sp}%
\begingroup\makeatletter\ifx\SetFigFont\undefined%
\gdef\SetFigFont#1#2#3#4#5{%
  \reset@font\fontsize{#1}{#2pt}%
  \fontfamily{#3}\fontseries{#4}\fontshape{#5}%
  \selectfont}%
\fi\endgroup%
\begin{picture}(5079,1559)(1902,-7227)
\put(2344,-6600){\makebox(0,0)[lb]{\smash{{\SetFigFont{10}{8}{\rmdefault}{\mddefault}{\updefault}{\color[rgb]{0,0,0}$g$}%
}}}}
\put(4784,-6600){\makebox(0,0)[lb]{\smash{{\SetFigFont{10}{8}{\rmdefault}{\mddefault}{\updefault}{\color[rgb]{0,0,0}$g$}%
}}}}
\put(6300,-6800){\makebox(0,0)[lb]{\smash{{\SetFigFont{10}{8}{\rmdefault}{\mddefault}{\updefault}{\color[rgb]{0,0,0}$h$}%
}}}}
\put(3000,-4500){\makebox(0,0)[lb]{\smash{{\SetFigFont{10}{8}{\rmdefault}{\mddefault}{\updefault}{\color[rgb]{0,0,0}$p_0$}%
}}}}
\put(3500,-6206){\makebox(0,0)[lb]{\smash{{\SetFigFont{10}{8}{\rmdefault}{\mddefault}{\updefault}{\color[rgb]{0,0,0}$p_2$}%
}}}}
\put(5370,-4500){\makebox(0,0)[lb]{\smash{{\SetFigFont{10}{8}{\rmdefault}{\mddefault}{\updefault}{\color[rgb]{0,0,0}$p_0$}%
}}}}
\end{picture}%
\caption{The metrics $J_{01}(g,g_3)$ (left) and $\mu^{\torp}(g,h)=J_{02}(h, J_{01}(g,g_3))$ (right)}
\label{minemetricjoin2}
\end{figure}   
}
\end{Example}
Of course there are several choices to be made in specifying such a map. For a start, there is the choice of rescaling function $f$ whose notation we have suppressed.  Furthermore, there is the choice of distinct points $p_1, p_2\in S^{n}\setminus \{p_0\}$, and the choice of metric $g_3\in \Riem_{{\rm torp}({\bf p})}^{+}(S^{n})$, where ${\bf p}=\{p_0, p_1, p_2\}$. We are assuming that $p_0$ is fixed throughout and so the map $\mu^{\torp}$ can be thought of as dependent on the choices of $f, p_1, p_2$ and $g_3$. As one may suspect, up to homotopy many of these choices have no impact.  We will shortly state a lemma, Lemma \ref{homotprod} below, which helps clarify the situation. Before doing this, it is time to revisit a construction which is at the heart of this paper: the Gromov-Lawson connected sum construction.

\subsection{The Gromov-Lawson connected sum construction.}
 We now consider a special case of the well-known Surgery Theorem due to Gromov and Lawson {\cite{GL}} and proved independently by Schoen and Yau {\cite{SY}}. In this paper we will concern ourselves only with the simplest type of surgery on smooth manifolds: the connected sum. Let $M^{n}$ be a smooth manifold with $n\geq 3$, $g$ a psc-metric on $M$, $p\in M$ a point and $B_g(p,\epsilon)$ a closed geodesic ball about $p$ with respect to the metric $g$. By specifying a curve $\gamma=\gamma_{g,p,\epsilon}$ of the type shown on the left of Fig. \ref{GLcurve}, it is possible to adjust the metric $g$ inside $B_g(p,\epsilon)$ by pushing out geodesic spheres in the space $B_g(p,\epsilon)\times [0,\infty)$ in a way that is determined by $\gamma$. More precisely, a geodesic sphere of radius $r$ is pushed out to lie in the slice $B_g(p,\epsilon)\times \{t\}$ where $(t,r)\in \gamma$. The induced metric on the resulting hypersurface $M_{\gamma}$, shown in the right image of Fig. \ref{GLcurve}, extends smoothly onto the rest of $M$ as a new metric $g'$. Essentially, the shape of $\gamma$ means that the resulting metric, $g'$, is very close to an infinitesimal torpedo metric with radius $\delta>0$ (which may be very small) on a neighbourhood of $p$. This is indicated by the shaded region of Fig. \ref{GLcurve}. 

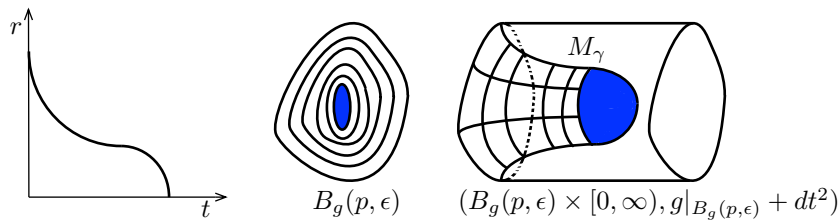
\begin{figure}[!htbp]
\vspace{0.0cm}
\hspace{2cm}
\begin{picture}(0,0)
\includegraphics{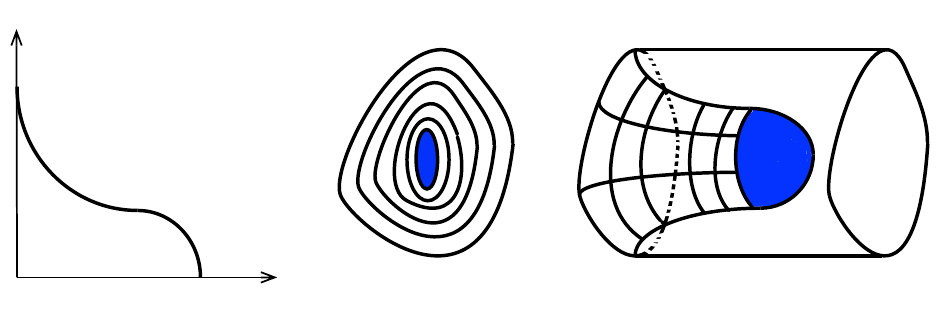}%
\end{picture}
\setlength{\unitlength}{3947sp}%
\begingroup\makeatletter\ifx\SetFigFont\undefined%
\gdef\SetFigFont#1#2#3#4#5{%
  \reset@font\fontsize{#1}{#2pt}%
  \fontfamily{#3}\fontseries{#4}\fontshape{#5}%
  \selectfont}%
\fi\endgroup%
\begin{picture}(5079,1559)(1902,-7227)
\put(1800,-6000){\makebox(0,0)[lb]{\smash{{\SetFigFont{10}{8}{\rmdefault}{\mddefault}{\updefault}{\color[rgb]{0,0,0}$r$}%
}}}}
\put(3700,-7100){\makebox(0,0)[lb]{\smash{{\SetFigFont{10}{8}{\rmdefault}{\mddefault}{\updefault}{\color[rgb]{0,0,0}$B_g(p,\epsilon)$}%
}}}}
\put(3000,-7150){\makebox(0,0)[lb]{\smash{{\SetFigFont{10}{8}{\rmdefault}{\mddefault}{\updefault}{\color[rgb]{0,0,0}$t$}%
}}}}
\put(4614,-7100){\makebox(0,0)[lb]{\smash{{\SetFigFont{10}{8}{\rmdefault}{\mddefault}{\updefault}{\color[rgb]{0,0,0}$(B_g(p,\epsilon)\times [0,\infty), g|_{B_g(p,\epsilon)}+dt^{2})$}%
}}}}
\put(5300,-6130){\makebox(0,0)[lb]{\smash{{\SetFigFont{10}{8}{\rmdefault}{\mddefault}{\updefault}{\color[rgb]{0,0,0}$M_{\gamma}$}%
}}}}
\end{picture}%
\caption{The curve $\gamma$ (left), geodesic ball $B_{g}(p,\epsilon)$ (middle) and the hypersurface obtained by pushing out geodesic spheres with respect to $\gamma$}
\label{GLcurve}
\end{figure}   

One of the main challenges of this construction was in demonstrating the fact that $\gamma$ can always be chosen so that positivity of the scalar curvature is maintained. As the space of psc-metrics is open and as the psc-metric induced on $M_{\gamma}$ is close to being ``standard" near $p$, a psc-isotopy (obtained from a linear homotopy through metrics) is then used to adjust $g'$ so that, near $p$, it is precisely an infinitesimal torpedo metric of radius $\delta$. Of course, if the geodesic spheres around $p$ are already standard spheres (as in the case when the original metric $g$ is a round metric) then no such final adjustment is necessary. This will very often be the case for us in this paper. Note that if $\delta>0$ is sufficiently small for the construction to work, then it works also for all $\delta'$ so that $0<\delta'<\delta$. Thus, given a pair of Riemannian manifolds $(M^{n}, g)$ and $(N^{n},h)$ where $g$ and $h$ are psc-metrics and $n\geq 3$, and a pair of points $p\in M$ and $q\in N$, a sufficiently small $\delta$ may be found for the construction of both $g'$ and $h'$. Removing the infinitesimal torpedo caps from these metrics gives rise to psc-metrics which may be identified along their respective boundaries to obtain a psc-metric on the connected sum $M\# N$.   

In \cite{Walsh1}, we further show that a psc-isotopy may be obtained between the metrics $g$ and $g'$. The main work is in contracting the curve $\gamma$ back to the vertical axis. Note that we show in Theorem 2.13 of \cite{Walsh1} that the whole construction (including this psc-isotopy) actually goes through for a compact family of psc-metrics. We restate the relevant aspects of this theorem as Lemma \ref{GLfamily} below.
\begin{Lemma}\label{GLfamily}
Let $M^{n}$ be a compact smooth manifold of dimension $n\geq 3$. Suppose $g_t, t\in K$ is a compact family of psc-metrics on $M$, indexed by a compact space $K$. Suppose also that $p_s, s\in I$ is a path in $M$. Then there is a continuous family of neighbourhoods $U_s$ of $p_s$ in $M$ as $s$ varies along $I$, a constant $\delta>0$ and a compact family of psc-metrics $g_{st}', (s,t)\in I\times K$, which satisfies the following:
\begin{enumerate}
\item{} For each $(s,t)\in I\times  K$, $g_{st}'=g_t$ outside of $U_s$. 
\item{} For $s\in I$, there is a continuously varying family of neighbourhoods $D_{s}$ of $p$, with $D_{s}\subset U_{s}$ so that for each $(s,t)\in I\times K$, $g_{st}'$ takes the form of an infinitesimal torpedo metric of radius $\delta$ on $D_s$.  
\item{} For each $s_0\in I$ there is a continuous homotopy through families of psc-metrics on $S^{n}$ which deforms the family $\{g_t:t\in K\}$ into $\{g_{{s_0}t}':t\in K\}$.
\end{enumerate} 
\end{Lemma}
\begin{proof}
This is a special case of Theorem 2.13 from \cite{Walsh1}.
\end{proof}
\noindent We have already seen an application of Lemma \ref{GLfamily} in the construction of the metric $g_3^{m}$ in Example \ref{g3eg}. Another  important application for us is the following lemma. This is actually a special case of a more general result proved in Lemma 3.3 of \cite{Walsh3}.
\begin{Lemma}\label{homequivspaces}
When $n\geq 3$ and ${\bf p}$ is a finite collection of points on the sphere $S^{n}$, the spaces $\Riem_{{\torp}^{+}({\bf p})}(S^{n})$ and $\Riem^{+}(S^{n})$ are homotopy equivalent.
\end{Lemma}
\begin{proof}
 The idea is to use the construction in Lemma \ref{GLfamily} to show that elements of the relative homotopy groups $\pi_{k}(\Riem^{+}(S^{n}), \Riem^{+}_{{\torp}({\bf p})}(S^{n}))$ are trivial. This, coupled with a result of Palais which indicates that these spaces of psc-metrics are dominated by CW-complexes, allows us to conclude the result via a well-known theorem of Whitehead. Details can be found in Lemma 3.3 of \cite{Walsh3}.
\end{proof}

We close this section by clarifying an earlier comment concerning the multiplication map $\mu^{\torp}$ in \ref{muprod1}. Recall that this map depended on several choices: a rescaling function $f:(0, \infty)\times(0,\infty)\rightarrow (0,\infty)$, a pair $p_1, p_2 \in S^{n}\setminus\{p_0\}$ determining a triple of distinct points ${\bf p}=\{p_0, p_1, p_2\}$ (assuming $p_0$ is already fixed) and a psc-metric $g_3\in \Riem_{{\torp}({\bf p})}^{+}(S^{n})$. We now state a theorem, which makes use of Lemma \ref{GLfamily} above.

\begin{Lemma}\label{homotprod}
Let $f_1, f_2$ be continuous maps $(0, \infty)\times(0,\infty)\rightarrow (0,\infty)$ and let $p_1, p_2, p_1', p_2'\in S^{n}\setminus \{p_0\}$. Suppose also that $g_3\in \Riem_{{\rm torp}({\bf p})}^{+}(S^{n})$ and  $g_3'\in \Riem_{{\rm torp}({\bf p'})}^{+}(S^{n})$ where ${\bf p}=\{p_0, p_1, p_2\}$ and ${\bf p'}=\{p_0, p_1', p_2'\}$. Then the corresponding maps $\mu^{\torp}=\mu_{(f_1, p_1, p_2, g_3)}^{\torp}$ and $\mu^{\torp'}=\mu_{(f_2, p_1',p_2', g_3' )}^{\torp}$ are homotopic if and only if the metrics $g_3$ and $g_3'$ are psc-isotopic.
\end{Lemma}
\begin{proof} Convexity of the space of maps $(0, \infty)\times(0,\infty)\rightarrow (0,\infty)$ means that the choice of rescaling map has no effect up to homotopy. Furthermore, it is an immediate consequence of Lemma \ref{GLfamily} that the process of pushing out a cap while maintaining positive scalar curvature can be done ``on the move", i.e. while the point is moving along a continuous path. Thus the choice of points, whether  $\{ p_1, p_2 \}$ or $\{p_1', p_2'\}$, has no effect up to homotopy either.

Now suppose $g_3$ and $g_3'$ are psc-isotopic. It is a straightforward consequence of Lemma \ref{GLfamily} that $g_3$ and $g_3'$ can be connected by a continuous path in $\Riem_{{\rm torp}({\bf p})}^{+}(S^{n})$. Homotopy equivalence of the maps $\mu^{\torp}$ and $\mu^{\torp'}$ follows easily.
Now suppose $\mu^{\torp}$ and $\mu^{\torp'}$ are homotopic. Let $\mu_t^{\torp}, t\in I$, be a continuous family of maps:
\begin{equation*}
\mu_t^{\torp}:\Riem_{{\mathrm{torp}}({p_0 })}^{+}(S^{n})\times \Riem_{{\mathrm{torp}}({p_0 })}^{+}(S^{n})\longrightarrow \Riem_{{\mathrm{torp}}({p_0 })}^{+}(S^{n}),
\end{equation*}
where $\mu_0^{\torp}=\mu^{\torp}$ and $\mu_1^{\torp}=\mu^{\torp'}$. Consider the path in $\Riem^{+}(S^{n})$ given by $t\mapsto \mu_{t}^{\torp}(\bar{g}_1, \bar{g}_1)$, where $\bar{g}_1$ is the metric we constructed before stating this lemma. This path is now easily deformed, using Lemma \ref{GLfamily}, into a path which connects the metrics $g_3$ and $g_3'$.
\end{proof}
\noindent As the space $\Riem^{+}(S^{n})$ is often not path connected, it is evident that there are many non-homotopic possibilities for the map $\mu^{\torp}$. In the next section, we will see that in order to make $\Riem_{{\mathrm{torp}}({p_0 })}^{+}(S^{n})$ into an $H$-space, we will have to restrict our choice of $g_3$ to metrics which lie in the same path component of the standard round metric.

\section{The $H$-Space Theorem.}\label{HSpaceThm3} We are now able to state and prove the first of our main results. Let $\mu^{\torp}$ be the map defined in \ref{muprod1} above with respect to some $4$-tuple $(f, p_1, p_2, g_3)$ where as before $f:(0,\infty)\times (0,\infty)\rightarrow (0,\infty)$ is a continuous map, $p_1, p_2\in S^{n}\setminus \{p_0\}$ are distinct points and $g_3\in \Riem_{{\rm torp}({\bf p})}^{+}(S^{n})$. We obtain the following theorem.
\begin{Theorem}\label{HTheorem}
Let $n\geq 3$ and let $\mu^{\torp}=\mu_{(f, p_1, p_1, g_3)}^{\torp}$ be the multiplication map given by formula \ref{muprod1}. In the case when the metric $g_3$ is psc-isotopic to the round metric $ds_{n}^{2}$, $\mu^{\torp}$ defines defines a homotopy product on $\Riem_{{\mathrm{torp}}({p_0 })}^{+}(S^{n})$ with homotopy identity $\bar{g}_1$, giving it the structure of an $H$-space. Furthermore, this product is both homotopy commutative and homotopy associative.
\end{Theorem}
\begin{proof}
For convenience, we take $g_3$ to be the $3$-cap metric $g_3^{m}$ constructed in the previous section. This is reasonable given Lemma \ref{homotprod} and our hypothesis that all choices of $g_3$ are psc-isotopic to the standard round metric $ds_{n}^{2}$.
We begin by showing that the metric $\bar{g}_1$ constructed earlier plays the role of homotopy identity. We will show that the map $g\mapsto \mu^{\torp}(g,\bar{g}_1)$ is homotopic to the identity map $g\mapsto g$. The case of $g\mapsto \mu^{\torp}(\bar{g}_1, g)$ is completely analogous. For the most part this will involve a psc-isotopy of the metric $g_3$. To help see this we represent, for an arbitrary $g\in \Riem_{{\mathrm{torp}}({p_0 })}^{+}(S^{n})$, the element $\mu^{\torp}(g, \bar{g}_1)$ in Fig. \ref{homidentity} below. 
\begin{figure}[!htbp]
\vspace{3.0cm}
\hspace{-1.0cm}
\begin{picture}(0,0)
\includegraphics{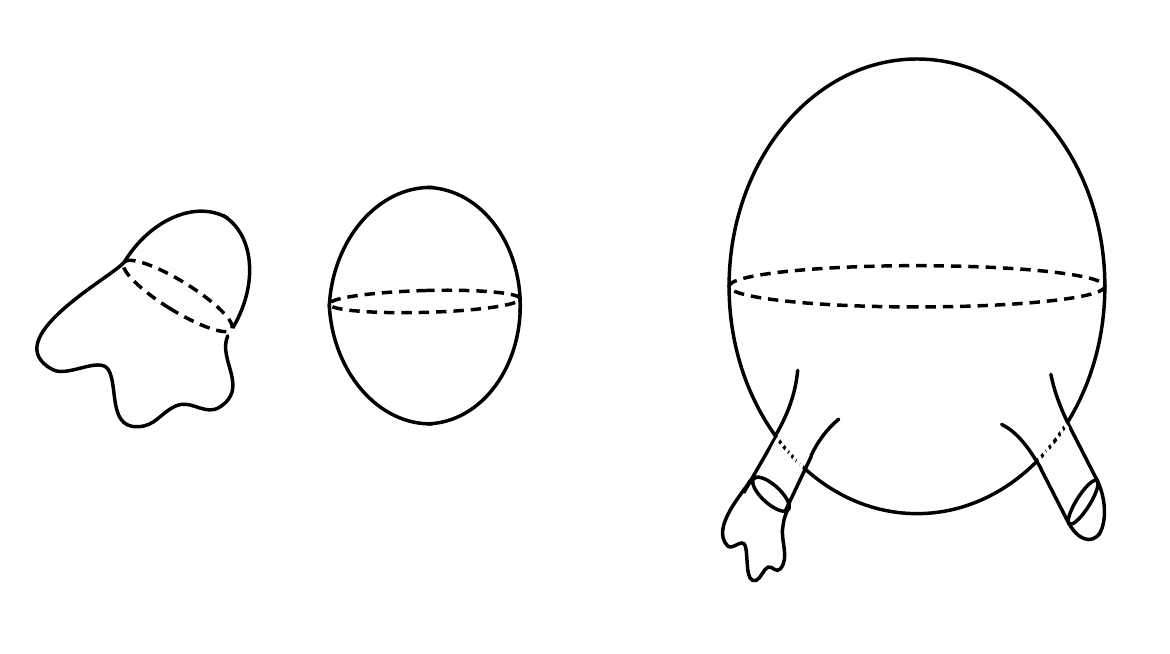}%
\end{picture}
\setlength{\unitlength}{3947sp}%
\begingroup\makeatletter\ifx\SetFigFont\undefined%
\gdef\SetFigFont#1#2#3#4#5{%
  \reset@font\fontsize{#1}{#2pt}%
  \fontfamily{#3}\fontseries{#4}\fontshape{#5}%
  \selectfont}%
\fi\endgroup%
\begin{picture}(5079,1559)(1902,-7227)
\put(2414,-6266){\makebox(0,0)[lb]{\smash{{\SetFigFont{10}{8}{\rmdefault}{\mddefault}{\updefault}{\color[rgb]{0,0,0}$g$}%
}}}}
\put(2820,-5050){\makebox(0,0)[lb]{\smash{{\SetFigFont{10}{8}{\rmdefault}{\mddefault}{\updefault}{\color[rgb]{0,0,0}$p_0$}%
}}}}
\put(3844,-6300){\makebox(0,0)[lb]{\smash{{\SetFigFont{10}{8}{\rmdefault}{\mddefault}{\updefault}{\color[rgb]{0,0,0}$\bar{g}_1$}%
}}}}
\put(3800,-4950){\makebox(0,0)[lb]{\smash{{\SetFigFont{10}{8}{\rmdefault}{\mddefault}{\updefault}{\color[rgb]{0,0,0}$p_0$}%
}}}}
\put(5940,-6000){\makebox(0,0)[lb]{\smash{{\SetFigFont{10}{8}{\rmdefault}{\mddefault}{\updefault}{\color[rgb]{0,0,0}$\mu^{\torp}(g, \bar{g}_1)$}%
}}}}
\end{picture}%
\caption{The metrics $g, \bar{g}_1$ and $\mu^{\torp}(g, \bar{g}_1)$}
\label{homidentity}
\end{figure}
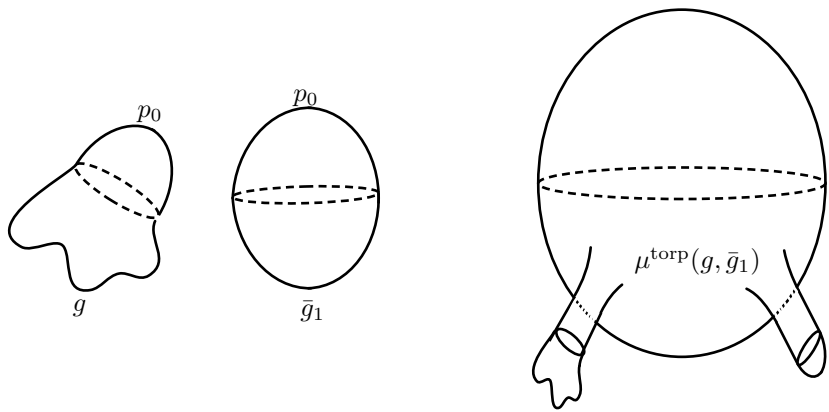   
We will now construct a deformation of the map $g\mapsto \mu^{\torp}(g, \bar{g}_1)= J_{01}(g, J_{02}(\bar{g}_1, g_3))$ to the identity map. Recall that the map $J_{0j}$ is really $J_{0j}^{\cyl(f)}$ where $f:(0,\infty)\times (0,\infty)\rightarrow(0,\infty)$ is the rescaling map defined earlier. The first step is to replace $J_{01}^{\cyl(f)}$ with $J_{0j}^{\cyl((1-t)f+t\pi_\L)}$, where $t\in [0,1]$. Recall that $\pi_\L:(0,\infty)\times (0,\infty)\rightarrow(0,\infty)$ is projection onto the left factor. This induces a homotopy of the map $g\mapsto \mu^{\torp}(g, \bar{g}_1)$ to one which fixes the size of the left input metric $g$ during attachment to the right metric $ J_{02}(\bar{g}_1, g_3)$. We once again supress the scaling function in our notation. The main step now involves constructing a psc-isotopy of the metric $J_{02}(\bar{g}_1, g_3)$ which will induce  a psc-isotopy on $J_{01}(g, J_{02}(\bar{g}_1, g_3))$ turning it into $g$. Importantly, this construction uses no data arising from the metric $g$ and so easily goes through for all choices of $g$.  As a result, this psc-isotopy induces a homotopy of the map $g\mapsto \mu^{\torp}(g, \bar{g}_1)$ to the identity map $g\mapsto g$. 

We will now describe the psc-isotopy from $J_{01}(g, J_{02}(\bar{g}_1, g_3))$ to $g$. To aid the reader, a step by step description of this psc-isotopy is depicted in Fig. \ref{homidentity2}. The first image in Fig. \ref{homidentity2} is the metric $J_{01}(g, J_{02}(\bar{g}_1, g_3))$. As a first step, we use lemma \ref{GLfamily} to contract the cap at $p_2$ and obtain a psc-isotopy (by a simple rotation) of the remaining metric to one where the cap at $p_0$ is antipodal to the connecting metric with $g$. This is depicted in the second and third images in Fig. \ref{homidentity2}. The resulting metric then easily contracts, via the results of chapter 1 of \cite{Walsh1}, to the one shown in the fourth image in Fig. \ref{homidentity2} which takes the form of a standard torpedo metric on $S^{n}\setminus D_{p_1}$ and is connected to $g|_{D_{p_1}'}$ along the boundary. Finally, the neck of this torpedo is contracted down to yield precisely the metric $g$. 
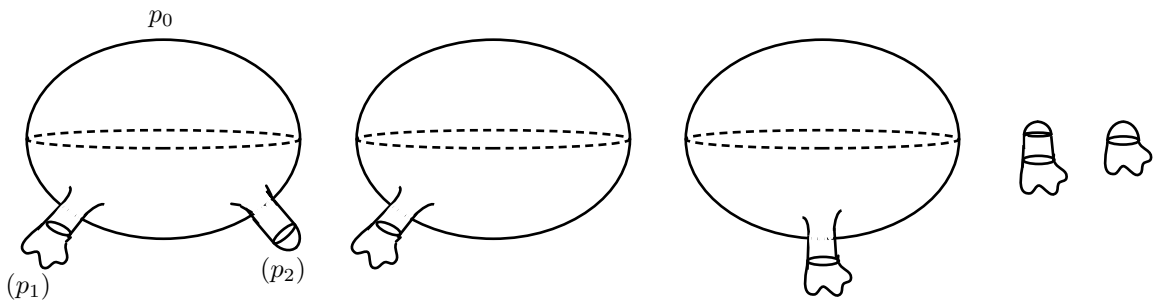
\begin{figure}[!htbp]
\vspace{1.0cm}
\hspace{-4.5cm}
\begin{picture}(0,0)
\includegraphics{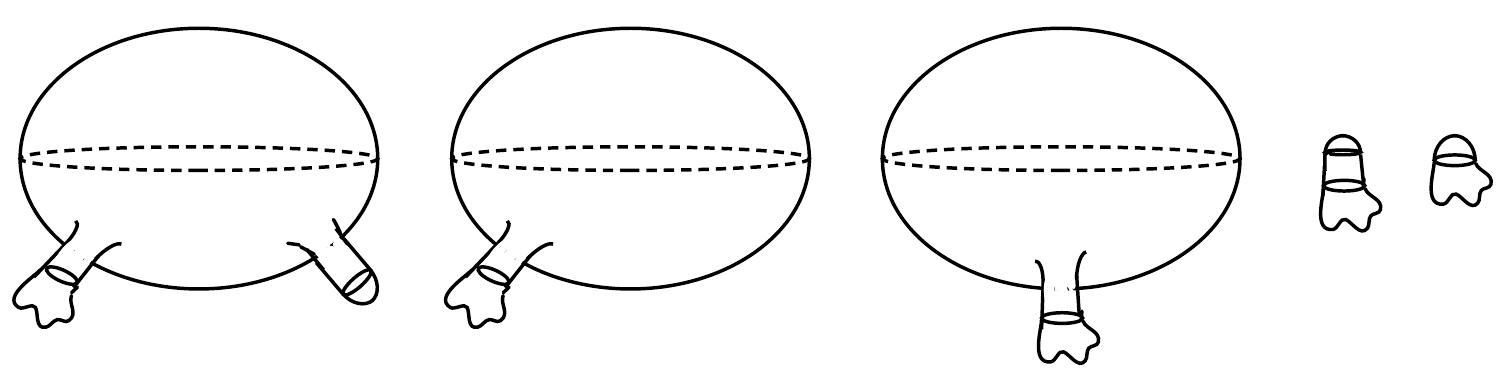}%
\end{picture}
\setlength{\unitlength}{3947sp}%
\begingroup\makeatletter\ifx\SetFigFont\undefined%
\gdef\SetFigFont#1#2#3#4#5{%
  \reset@font\fontsize{#1}{#2pt}%
  \fontfamily{#3}\fontseries{#4}\fontshape{#5}%
  \selectfont}%
\fi\endgroup%
\begin{picture}(5079,1559)(1902,-7227)
\put(1800,-7100){\makebox(0,0)[lb]{\smash{{\SetFigFont{10}{8}{\rmdefault}{\mddefault}{\updefault}{\color[rgb]{0,0,0}$(p_1)$}%
}}}}
\put(2700,-5400){\makebox(0,0)[lb]{\smash{{\SetFigFont{10}{8}{\rmdefault}{\mddefault}{\updefault}{\color[rgb]{0,0,0}$p_0$}%
}}}}
\put(3400,-7000){\makebox(0,0)[lb]{\smash{{\SetFigFont{10}{8}{\rmdefault}{\mddefault}{\updefault}{\color[rgb]{0,0,0}$(p_2)$}%
}}}}

\end{picture}%
\caption{An isotopy of the metric $J_{02}(\bar{g}_1, g_3)$ back to $g$. The brackets indicate the places where $p_1$ and $p_2$ lay before attachments} 
\label{homidentity2}
\end{figure}
   
Homotopy commutativity follows immediately from Lemma \ref{homotprod}. For a general $g_3$, just choose choose $p_1'=p_2$ and $p_2'=p_1$ in the statement of that lemma. In the case when $g_3=g_3^{m}$, this is most easily induced by continuous rotation of the sphere which swaps $p_1$ with $p_2$ as shown in Fig. \ref{rotate} below.

\begin{figure}[!htbp]
\vspace{1.5cm}
\hspace{1.0cm}
\begin{picture}(0,0)
\includegraphics{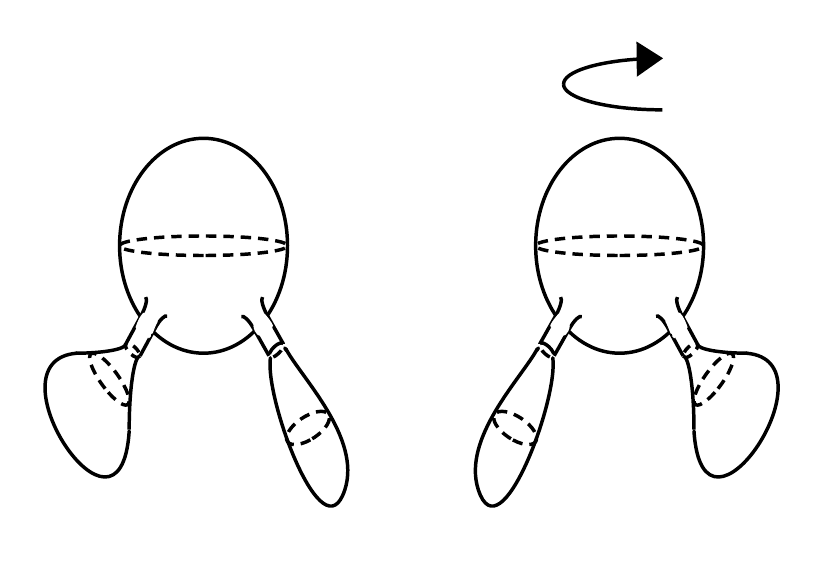}%
\end{picture}
\setlength{\unitlength}{3947sp}%
\begingroup\makeatletter\ifx\SetFigFont\undefined%
\gdef\SetFigFont#1#2#3#4#5{%
  \reset@font\fontsize{#1}{#2pt}%
  \fontfamily{#3}\fontseries{#4}\fontshape{#5}%
  \selectfont}%
\fi\endgroup%
\begin{picture}(5079,1559)(1902,-7227)
\put(2200,-6866){\makebox(0,0)[lb]{\smash{{\SetFigFont{10}{8}{\rmdefault}{\mddefault}{\updefault}{\color[rgb]{0,0,0}$g$}%
}}}}
\put(5500,-6766){\makebox(0,0)[lb]{\smash{{\SetFigFont{10}{8}{\rmdefault}{\mddefault}{\updefault}{\color[rgb]{0,0,0}$g$}%
}}}}
\put(4100,-7020){\makebox(0,0)[lb]{\smash{{\SetFigFont{10}{8}{\rmdefault}{\mddefault}{\updefault}{\color[rgb]{0,0,0}$h$}%
}}}}
\put(3500,-7020){\makebox(0,0)[lb]{\smash{{\SetFigFont{10}{8}{\rmdefault}{\mddefault}{\updefault}{\color[rgb]{0,0,0}$h$}%
}}}}
\end{picture}%
\caption{The metrics $\mu^{\torp}(g,h)$ and $\mu^{\torp}(h,g)$}
\label{rotate}
\end{figure}
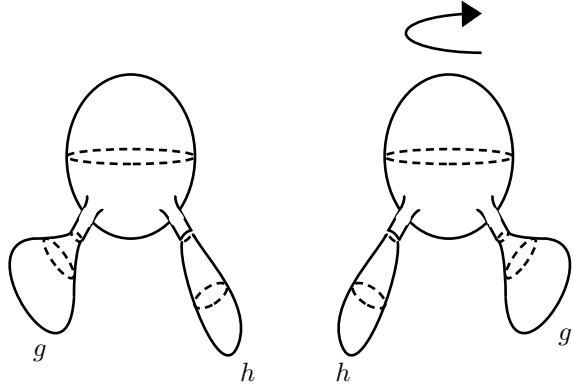   

It remains to show homotopy associativity. This is similar to the proof of homotopy commutativity in that it involves moving the metrics which are attached to $g_3$, by their attaching tubes, along a closed bounded arc while adjusting the radius if necessary. Again we utilise Lemma \ref{GLfamily} to move the metrics continuously and maintain positivity of the scalar curvature. In particular, this means that for any $g,h,h'\in\Riem_{{\mathrm{torp}}({p_0 })}^{+}(S^{n})$ the metric $\mu^{\torp}(\mu^{\torp}(g,h), h')$ is psc-isotopic to $\mu^{\torp}(g,\mu^{\torp}(h,h'))$ as shown in Fig. \ref{assoc} below. 
\begin{figure}[!htbp]
\vspace{3.0cm}
\hspace{1.0cm}
\begin{picture}(0,0)
\includegraphics{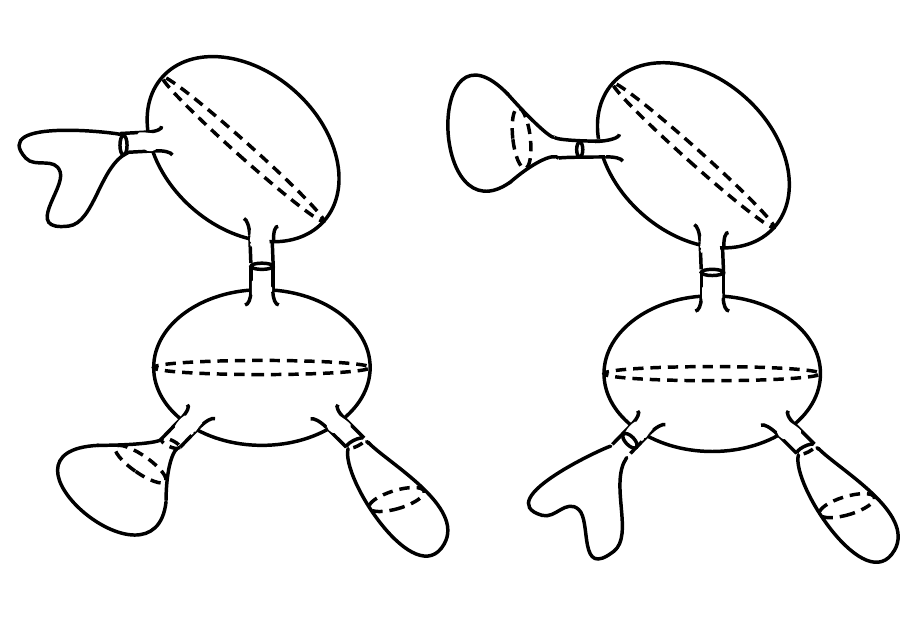}%
\end{picture}
\setlength{\unitlength}{3947sp}%
\begingroup\makeatletter\ifx\SetFigFont\undefined%
\gdef\SetFigFont#1#2#3#4#5{%
  \reset@font\fontsize{#1}{#2pt}%
  \fontfamily{#3}\fontseries{#4}\fontshape{#5}%
  \selectfont}%
\fi\endgroup%
\begin{picture}(5079,1559)(1902,-7227)
\put(2250,-6866){\makebox(0,0)[lb]{\smash{{\SetFigFont{10}{8}{\rmdefault}{\mddefault}{\updefault}{\color[rgb]{0,0,0}$g$}%
}}}}
\put(3850,-4866){\makebox(0,0)[lb]{\smash{{\SetFigFont{10}{8}{\rmdefault}{\mddefault}{\updefault}{\color[rgb]{0,0,0}$g$}%
}}}}
\put(2200,-4800){\makebox(0,0)[lb]{\smash{{\SetFigFont{10}{8}{\rmdefault}{\mddefault}{\updefault}{\color[rgb]{0,0,0}$h'$}%
}}}}
\put(5644,-6766){\makebox(0,0)[lb]{\smash{{\SetFigFont{10}{8}{\rmdefault}{\mddefault}{\updefault}{\color[rgb]{0,0,0}$h$}%
}}}}
\put(4400,-6866){\makebox(0,0)[lb]{\smash{{\SetFigFont{10}{8}{\rmdefault}{\mddefault}{\updefault}{\color[rgb]{0,0,0}$h'$}%
}}}}

\put(3500,-6766){\makebox(0,0)[lb]{\smash{{\SetFigFont{10}{8}{\rmdefault}{\mddefault}{\updefault}{\color[rgb]{0,0,0}$h$}%
}}}}
\end{picture}%
\caption{The metrics $\mu^{\torp}(\mu^{\torp}(g,h),h')$ and $\mu^{\torp}(g,\mu^{\torp}(h.h'))$}
\label{assoc}
\end{figure}
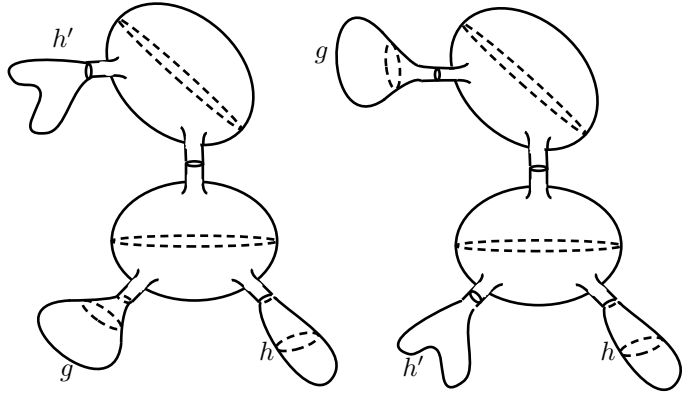   
With this is mind, let $g_3$ and $g_3'$ be two copies of the same intermediary metric with cap points $\{p_0, p_1, p_2\}$ and $\{ p_0', p_1', p_2'\}$ respectively. Obviously $p_0'$ is identified with the base point $p_0\in S^{n}$. We now fix a psc-isotopy $g(t)$ which moves $g(0)=J_{01}(g_3, g_3')$ to the metric $g(1)$ obtained by swapping the $p_1$ cap of $g_3$ with the $p_2'$ cap of $g_3'$. Now consider the map $(g,h,h')\mapsto \mu_0^{\torp}(g,h,h')=\mu^{\torp}(\mu^{\torp}(g,h),h'))$. Viewing the metric $J_{01}(g_3, g_3)$ as the metric $J_{01}(g_3, g_3')$ and then replacing it by $g(t)$ induces a homotopy of maps $\mu_t^{\torp}$ between $\mu_0^{\torp}$ which is defined above and $\mu_1^{\torp}$ defined by $\mu_1^{\torp}(g,h,h')=\mu^{\torp}(g,\mu^{\torp}(h,h'))$. 
\end{proof}

\begin{Corollary}\label{AbFG}
For $n\geq 3$, the group $\pi_1(\Riem^{+}(S^{n}), ds_{n}^{2})$ is Abelian.
\end{Corollary}
\begin{proof}
This is a well-known fact about $H$-spaces and so we will be terse. Suppose $Z$ is an $H$-space with multiplication $\mu$ and homotopy identity $e\in Z$. Let $\alpha, \beta:[0,1]\rightarrow Z$ be loops based at $e$ representing classes of $\pi_{1}(Z, e)$. Now consider the map $F:[0,1]\times [0,1]\rightarrow Z$ which is defined $F(s,t)=\mu(\alpha(s), \beta(t))$. Depending on one's choice of parameterisation, the restriction of $F$ firstly to the left and bottom sides of the square and secondly to the top and right sides of the square gives rise to maps which are respectively homotopy equivalent to the loop concatenations $\alpha\circ \beta$ and $\beta\circ \alpha$. Constructing the appropriate homotopy is then a straightforward exercise. 
\end{proof}

\section{Bulb Metrics}\label{bulbsection} 
In section \ref{disksphere}, we discussed two types of psc-metric on the disk $D^{n}$ which we could use in appropriate gluing constructions to obtain psc-metrics on the sphere $S^{n}$. The first of these, metrics which are cylindrical (at least infinitesimally) near the boundary, motivated the construction of the space of psc-metrics on $S^{n}$ with torpedo caps, and ultimately a product on this space. We will now carry out an analogous project for the second of our families of disk metrics with well-behaved boundaries: metrics which are sphere-like near the boundary. The construction here is more complicated. However, it will allow us to expose extra structure, beyond the $H$-space structure, on the space of psc-metrics on $S^{n}$.

\subsection{The space of psc-metrics with bulbs.}
Suppose we have a psc-metric $g$ on a smooth manifold $M^{n}$, a point $x\in M$ and a geodesic ball of radius $\epsilon>0$ about $x$, $B_g(x,\epsilon)$. Recall that when $n\geq3$, we may construct a psc-isotopy starting at the metric $g$, which fixes $g$ outside of $B_g(x,\epsilon)$, and which pushes out an infinitesimal torpedo metric of radius $\delta>0$ (dependent on $g,x$ and $\epsilon$), on a disk around $x$ and inside $B_g(x,\epsilon)$. This follows immediately from Lemma \ref{GLfamily}. We have called this process {\em pushing out a torpedo cap} around $x$. In the case when the starting metric $g$ is the round metric of radius $\lambda>0$ on the sphere $S^{n}$, we can be very explicit about the construction. With this in mind, we will describe a family of psc-metrics which we call ``bulb" metrics. These are obtained by pushing out a torpedo cap of radius $\delta$ inside a geodesic ball of radius $r$ on a round sphere of radius $\lambda$. We make this precise in the following lemma.
\begin{Lemma}\label{bulblemma}
Let $n\geq 3$, $\lambda>0$ and $r<(0,\lambda\frac{\pi}{2}]$. Let $p$ and $p'$ be antipodal points in $S^{n}$. Then there is a psc-metric $g_{bulb}(p,\lambda, r)$ on the sphere $S^{n}$ which satisfies the following conditions:
\begin{enumerate}
\item{} The metric $g_{bulb}(p,\lambda, r)$ is rotationally symmetric about the line in $\mathbb{R}^{n+1}$ which connects $p$ to its antipodal point $p'$.
\item{} There are continuous parameters $r'=r'(\lambda, r)\in(0,r]$ and $\delta=\delta(\lambda, r)\in(0,r]$ so that on the ball $B_\lambda (p',r')$, the metric $g_{bulb}(p,\lambda, r)$ restricts as an infinitesimal torpedo metric of radius $\delta$. 
\item{} On the annular region $\ann_{\lambda}(p';r',r)$ (taken with respect to the round metric of radius $\lambda$), the metric $g_{bulb}(p,\lambda, r)$ takes the form of the connecting tube from the Gromov-Lawson construction.
\item{} Outside of the ball $B_\lambda (p',r)$, the metric $g_{bulb}(p,\lambda, r)$ is precisely the lens metric $g_{\lens}(\lambda, \lambda\pi-r)$, the lens complement obtained by removing the geodesic ball $B_{\lambda}(p', r)$ from the original round sphere of radius $\lambda$.
\item{} The metric $g_{\bulb}(p,\lambda, \lambda\frac{\pi}{2})$ is the original round metric of radius $\lambda$.
\end{enumerate}
\end{Lemma}
\begin{proof}
This is an easy consequence of Lemma \ref{GLfamily}.
\end{proof}
The metric $g_{\bulb}(p,\lambda, r)$ obtained by this construction will be known as a {\em $(\lambda,r)$-bulb metric}, or simply a {\em bulb}. Needless to say, all such metrics are psc-isotopic to the original round metric via Lemma \ref{GLfamily}. We point out that this metric decomposes the sphere naturally into regions:
\begin{equation}
S^{n}=B_{\lambda}(p,\lambda\pi-r)\cup\ann_{\lambda}(p';r',r)\cup B_{\lambda}(p', r),
\end{equation}
on which the metric takes the forms indicated in Fig. \ref{bulb} below. We will refer to these pieces respectively as the {\em head}, {\em neck} and {\em cap} of the bulb metric $g_{\bulb}(p,\lambda, r)$. Furthermore the quantities $\lambda, r$ and $\delta=\delta(\lambda,r)$ will be known as the {\em head radius}, {\em head angle} and {\em cap radius} respectively. The head metric, which is of course the metric $\lambda^{s}ds_{n}^{2}|_{B_\lambda(p,\lambda\pi-r)}$, will be denoted $g_{\head}(p,\lambda,r)$. Note that in the case where $r=\lambda\frac{\pi}{2}$, the head and cap are the respective hemispheres of radius $\lambda$ about $p$ and $p'$ while the neck is simply the $(n-1)$-dimensional equator.
\begin{figure}[!htbp]
\vspace{2cm}
\hspace{-1cm}
\begin{picture}(0,0)
\includegraphics{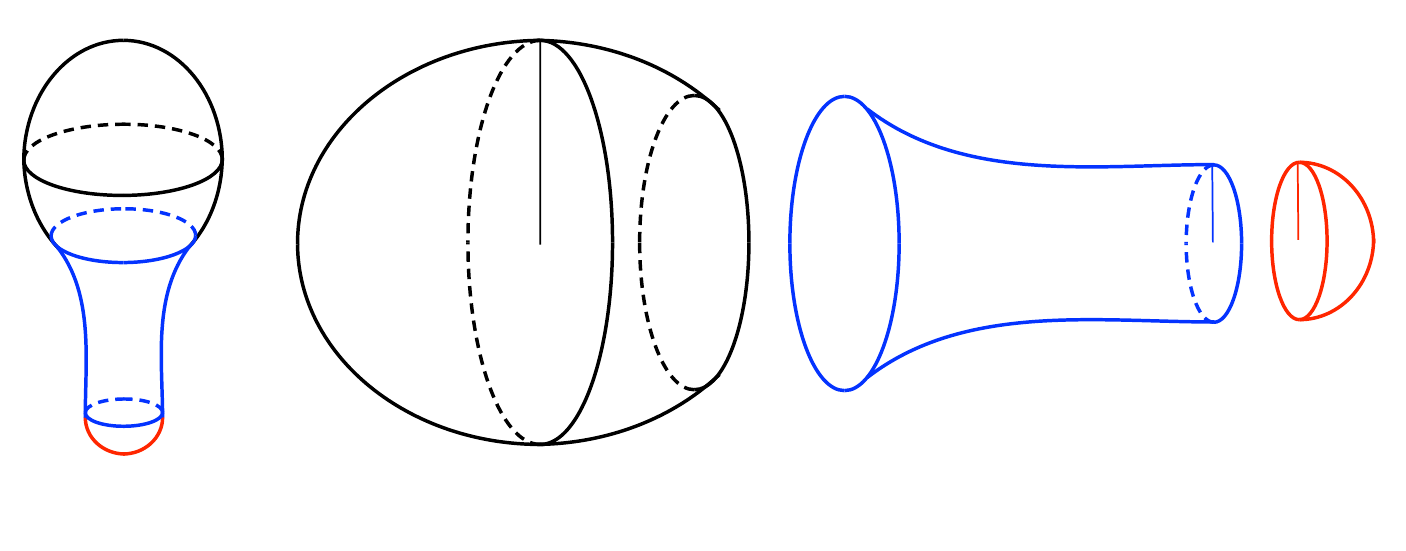}%
\end{picture}
\setlength{\unitlength}{3947sp}%
\begingroup\makeatletter\ifx\SetFigFont\undefined%
\gdef\SetFigFont#1#2#3#4#5{%
  \reset@font\fontsize{#1}{#2pt}%
  \fontfamily{#3}\fontseries{#4}\fontshape{#5}%
  \selectfont}%
\fi\endgroup%
\begin{picture}(5079,1559)(1902,-7227)
\put(4320,-5300){\makebox(0,0)[lb]{\smash{{\SetFigFont{10}{8}{\rmdefault}{\mddefault}{\updefault}{\color[rgb]{0,0,0}$\lambda$}%
}}}}
\put(7980,-5600){\makebox(0,0)[lb]{\smash{{\SetFigFont{10}{8}{\rmdefault}{\mddefault}{\updefault}{\color[rgb]{0,0,0}$\delta$}%
}}}}
\put(2400,-6900){\makebox(0,0)[lb]{\smash{{\SetFigFont{10}{8}{\rmdefault}{\mddefault}{\updefault}{\color[rgb]{0,0,0}$p'$}%
}}}}
\put(1500,-6700){\makebox(0,0)[lb]{\smash{{\SetFigFont{10}{8}{\rmdefault}{\mddefault}{\updefault}{\color[rgb]{0,0,0}$B_{\lambda}(p', r')$}%
}}}}
\put(1150,-6300){\makebox(0,0)[lb]{\smash{{\SetFigFont{10}{8}{\rmdefault}{\mddefault}{\updefault}{\color[rgb]{0,0,0}$\ann_{\lambda}(p'; r',r)$}%
}}}}
\put(1000,-4900){\makebox(0,0)[lb]{\smash{{\SetFigFont{10}{8}{\rmdefault}{\mddefault}{\updefault}{\color[rgb]{0,0,0}$B_{\lambda}(p, \lambda\pi-r)$}%
}}}}
\put(2400,-4700){\makebox(0,0)[lb]{\smash{{\SetFigFont{10}{8}{\rmdefault}{\mddefault}{\updefault}{\color[rgb]{0,0,0}$p$}%
}}}}
\put(4100,-6900){\makebox(0,0)[lb]{\smash{{\SetFigFont{10}{8}{\rmdefault}{\mddefault}{\updefault}{\color[rgb]{0,0,0}$g_{\lens}(\lambda,\lambda\pi-r)$}%
}}}}
\put(8100,-6300){\makebox(0,0)[lb]{\smash{{\SetFigFont{10}{8}{\rmdefault}{\mddefault}{\updefault}{\color[rgb]{0,0,0}$g^{\eta_{\delta}}$}%
}}}}
\end{picture}%
\caption{The metric $g_{bulb}(p,\lambda, r)$ (left) and its decomposition into head, neck and cap (right)}
\label{bulb}
\end{figure}
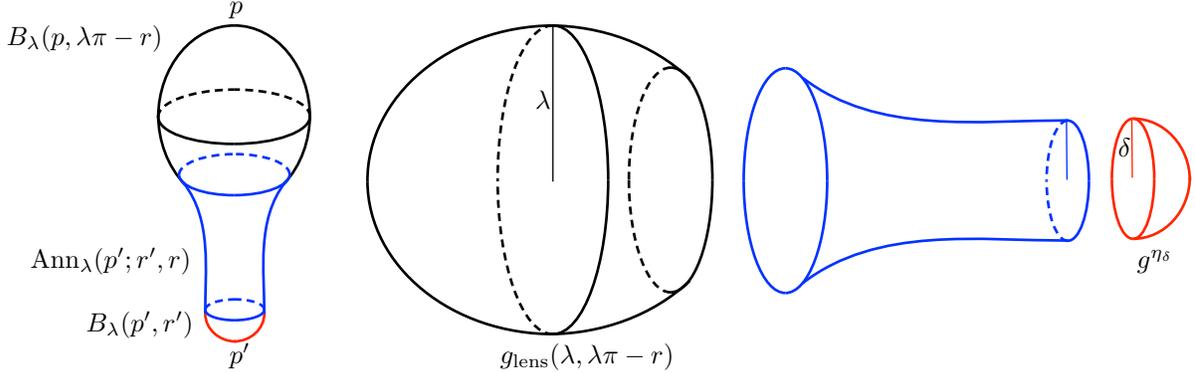   

It will be useful to think of the construction in Lemma \ref{bulblemma} as the following continuous map. 
Consider the subspace $\B$ of $(0,\infty)\times (0,\infty)$ consisting of pairs $(\lambda, r)$ which satisfy $0<r<\lambda\frac{\pi}{2}$. We define the map $\Bulb_{p}$ as follows:
\begin{equation}\label{bulbmap}
\begin{split}
\Bulb_{p}:S^{n}\times\B&\longrightarrow\Riem_{\torp(p')}^{+}(S^{n})\\
(p, (\lambda,r))&\longmapsto g_{\bulb}(p,\lambda,r).
\end{split}
\end{equation}
The reader should note that the image consists of metrics which are torpedo at $p'$, not $p$. In most cases the choice of $p$ is unimportant (unless of course we start pushing out extra torpedo caps). Thus, we will declare the ``standard case" to be the one where  $p$ is the north pole, $p'$ the south pole and write $\bulb_{p}$ simply as $\bulb$ and $g_{\bulb}(p,\lambda,r)$ simply as $g_{\bulb}(\lambda,r)$ in this case.

We denote by $\B{\rm ulb}_{p}(S^{n})$, the subspace in $\Riem_{\torp(p')}^{+}(S^{n})$ which is the image of the map $\Bulb_{p}$. Recall, we  defined an ``uncapping" map $\Unc_{p'}:\Riem_{\torp(p')}^{+}(S^{n})\rightarrow \Riem_{\cyl(0)}^{+}(D^{n})$, in \ref{uncap}, which removes the torpedo cap about $p'$ and pulls the remaining metric back to the standard disk $D^{n}$ in the usual way (sending $p$ to the origin $0$). We thus denote by $\NH_{\Riem}^{+}$, the image of the composition $\Unc_{p'}\circ \Bulb_{p}$, which uncaps the bulb around $p'$, leaving only the neck and head, and pulls the metric back to the standard disk $D^{n}$. The corresponding space of all head metrics, obtained by restricting each $g_{\bulb}(p,\lambda,r)$ to the ball $B_{\lambda}(p,\lambda\pi-r)$ and pulling back to $D^{n}$ in the usual way, is denoted $\H_{\Riem}^{+}$.
Thus, an element of $\NH_{\Riem}^{+}$ is the union of the head and neck of the bulb metric on $D^{n}$ and an element of $\H_{\Riem}^{+}$ is simply a head metric on $D^{n}$.

The reader should consider these space to be the analogues, with respect to bulbs, of the space of infinitesimal torpedo metrics $\T_{\Riem}^{+}$. This allows us define two important subspaces of the space of psc-metrics on $S^{n}$, in a similar vein to the definition of the space $\Riem_{\torp(p)}^{+}(S^{n})$ in \ref{spacetorp}.  Letting $D_{p}$ denote a normal coordinate ball around $p$ and  $\phi_{p}:D^{n}\rightarrow D_p$, $\phi_{p}':D^{n}\rightarrow D_{p}'$ the corresponding diffeomorphisms, as defined immediately prior to the definition in \ref{spacetorp}, we define the spaces of {\em psc-metrics on $S^{n}$ with bulbs at $p$} and {\em psc-metrics with heads at $p$}, $\Riem_{\bulb(p)}^{+}(S^{n})$ and $\Riem_{\head(p)}^{+}(S^{n})$, as follows:
\begin{equation}\label{genbulb}
\begin{split}
\Riem^{+}_{{\rm{bulb}}({p})}(S^{n})&=\{g\in \Riem^{+}(S^{n}):{\phi_{p}^{*}}(g|_{D_{p}})\in\NH_{\Riem}^{+}\},\\
\Riem^{+}_{{\rm{head}}({p})}(S^{n})&=\{g\in \Riem^{+}(S^{n}):{\phi_{p}^{*}}(g|_{D_{p}})\in\H_{\Riem}^{+}\}.
\end{split}
\end{equation}
Moreover, for a finite collection of points on $S^{n}$, ${\bf p}=\{p_0, p_1, \cdots, p_k\}$, with corresponding coordinate diffeomorphisms $\phi_{i}:D^{n}\rightarrow D_{p_i}$ , we obtain the more general spaces of psc-metrics with bulbs or heads around ${\bf p}$:
\begin{equation}\label{moregenbulb}
\begin{split}
\Riem^{+}_{{\rm{bulb}}({\bf p})}(S^{n})&=\bigcap_{i=0}^{k}\Riem^{+}_{{\rm{bulb}}({p_i})}(S^{n}),\\
\Riem^{+}_{{\rm{head}}({\bf p})}(S^{n})&=\bigcap_{i=0}^{k}\Riem^{+}_{{\rm{head}}({p_i})}(S^{n}).
\end{split}
\end{equation}

Ultimately, we will be more interested in the space $\Riem^{+}_{{\rm{head}}({\bf p})}(S^{n})$. However, both spaces will have their uses for us. As much of what we are going to say applies equally to both spaces, we let $\Riem^{+}_{{\rm{b/h}}({\bf p})}(S^{n})$ denote either of these spaces. 
Notice that each element $g\in\Riem^{+}_{{\rm{b/h}}({\bf p})}(S^{n})$ restricts near $p_i$ to a metric which is precisely the head of a $(\lambda_i,r_i)$-bulb metric for some $\lambda_i>0$ and $r_i\in(0,\lambda_i\frac{\pi}{2}]$. In particular, this means that if we cut off the head on such a metric, the resulting metric is, infinitesimally at least, sphere-like at the boundary. Recall that the head of the bulb at such a point $p_i$ is the part of the metric defined on the region $B_{g}(p_i, r_i)$ ($=B_{\lambda}(p_i, r_i)$, as $g$ and $\lambda^{2} ds_{n}^{2}$ agree on this region.) We now define, for any $\rho> 0$, the map $\Cut_{p_i, \rho}$ as follows:
\begin{equation}\label{cutmap}
\begin{split}
\Cut_{p_i, \rho}:\Riem^{+}_{{\rm{b/h}}({\bf p})}(S^{n})&\longrightarrow\Riem_{{\lens}(0)}^{+}(D^{n})\\
g&\longmapsto \begin{cases}
g|_{S^{n}\setminus B_{\lambda_i(g)}(p_i,\rho)}, & \text{if }0<\rho\leq\lambda_i(g)\pi-r_i(g),\\
g|_{S^{n}\setminus B_{\lambda_i(g)}(p_i,\lambda_i(g)\pi-r_i(g))}, & \text{if }\rho>\lambda_i(g)\pi-r_i(g),
\end{cases}
\end{split}
\end{equation}
where $\lambda_i(g)$ and $r_i(g)$ are the head radius and head angle of $g$ around $p_i$. In the case when $\rho=\lambda_i(g)\pi-r_i(g)$, we simply remove the actual $i^{th}$ head from $g$. To aid the reader we depict the resulting metric $\Cut_{p_i, \lambda_i(g)\pi-r_i(g)}(g)$, where $g\in \Riem^{+}_{{\rm{bulb}}({\bf p})}(S^{n})$ in Fig. \ref{metricwithbulbs} below. By allowing $\rho$ to vary, we give ourselves the option of cutting through the head along various geodesic spheres about $p_i$. For example, choosing $\rho=\lambda_{i}(g)\frac{\pi}{2}$ would cut the round hemisphere of radius $\lambda_i(g)$ from the head at $p_i$. This will be important later on.
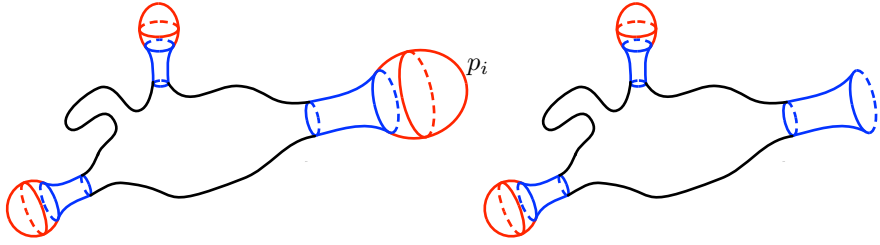
\begin{figure}[!htbp]
\vspace{0.5cm}
\hspace{-1cm}
\begin{picture}(0,0)
\includegraphics{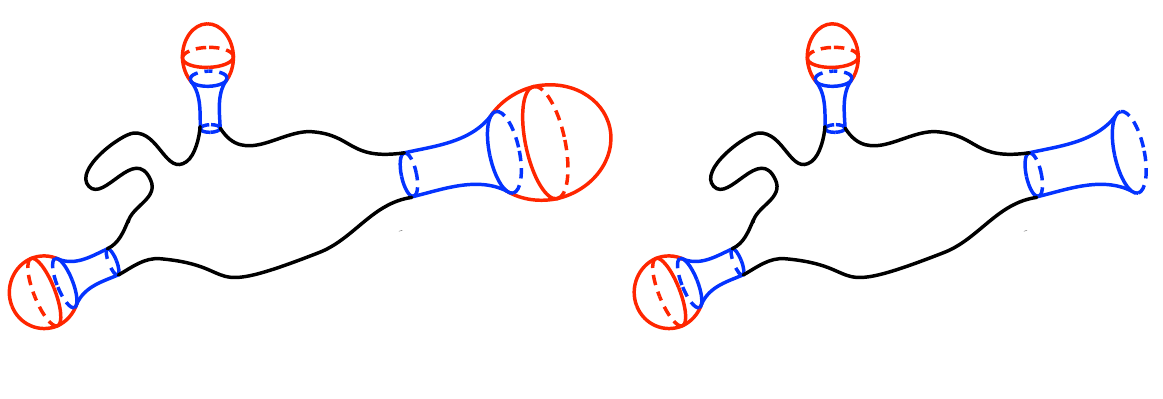}%
\end{picture}
\setlength{\unitlength}{3947sp}%
\begingroup\makeatletter\ifx\SetFigFont\undefined%
\gdef\SetFigFont#1#2#3#4#5{%
  \reset@font\fontsize{#1}{#2pt}%
  \fontfamily{#3}\fontseries{#4}\fontshape{#5}%
  \selectfont}%
\fi\endgroup%
\begin{picture}(5079,1559)(1902,-7227)
\put(4770,-5850){\makebox(0,0)[lb]{\smash{{\SetFigFont{10}{8}{\rmdefault}{\mddefault}{\updefault}{\color[rgb]{0,0,0}$p_i$}%
}}}}
\end{picture}%
\caption{A metric $g\in \Riem^{+}_{{\rm{bulb}}({\bf p})}(S^{n})$ (left) and the metric $\Cut_{p_i,\lambda_i(g)\pi-r_i(g)}(g)$ on $D^{n}$ obtained by cutting the $i^{th}$ head from $g$ (right)}
\label{metricwithbulbs}
\end{figure}   

Before discsussing the spaces $\Riem^{+}_{{\rm{b/h}}({\bf p})}(S^{n})$ any further, there is another construction we must attend to. We return once more to the general case of a smooth Riemannian manifold $(M, g)$, a point $x\in M$ and closed geodesic ball $B_g(x,\epsilon)$. Let $g'$ denote the metric obtained by pushing out an infinitesimal torpedo metric of radius $\delta$, for some $\delta>0$, on a disk around $x$ inside $B_g(x,\epsilon)$. Using Lemma \ref{GLfamily}, we will perform an adjustment to the metric $g'$ near $x$. Essentially, we construct a psc-isotopy, starting at $g'$, which is trivial outside of $B_g(x,\epsilon)$ and which results in a metric which is the result of a Gromov-Lawson connected sum between $(M, g)$ and the standard sphere. Importantly, it is the metric obtained by removing the torpedo cap of radius $\delta$ about $x$, removing a corresponding torpedo cap of radius $\delta$ from a bulb metric and gluing them together in the obvious way. This is made more precise in Lemma \ref{bulblemma2} below. In Fig. \ref{pushcapbulb} we depict the restrictions of these metrics to $B_g(x,\epsilon)$. 

\begin{figure}[!htbp]
\vspace{8cm}
\hspace{-5.0cm}
\begin{picture}(0,0)
\includegraphics{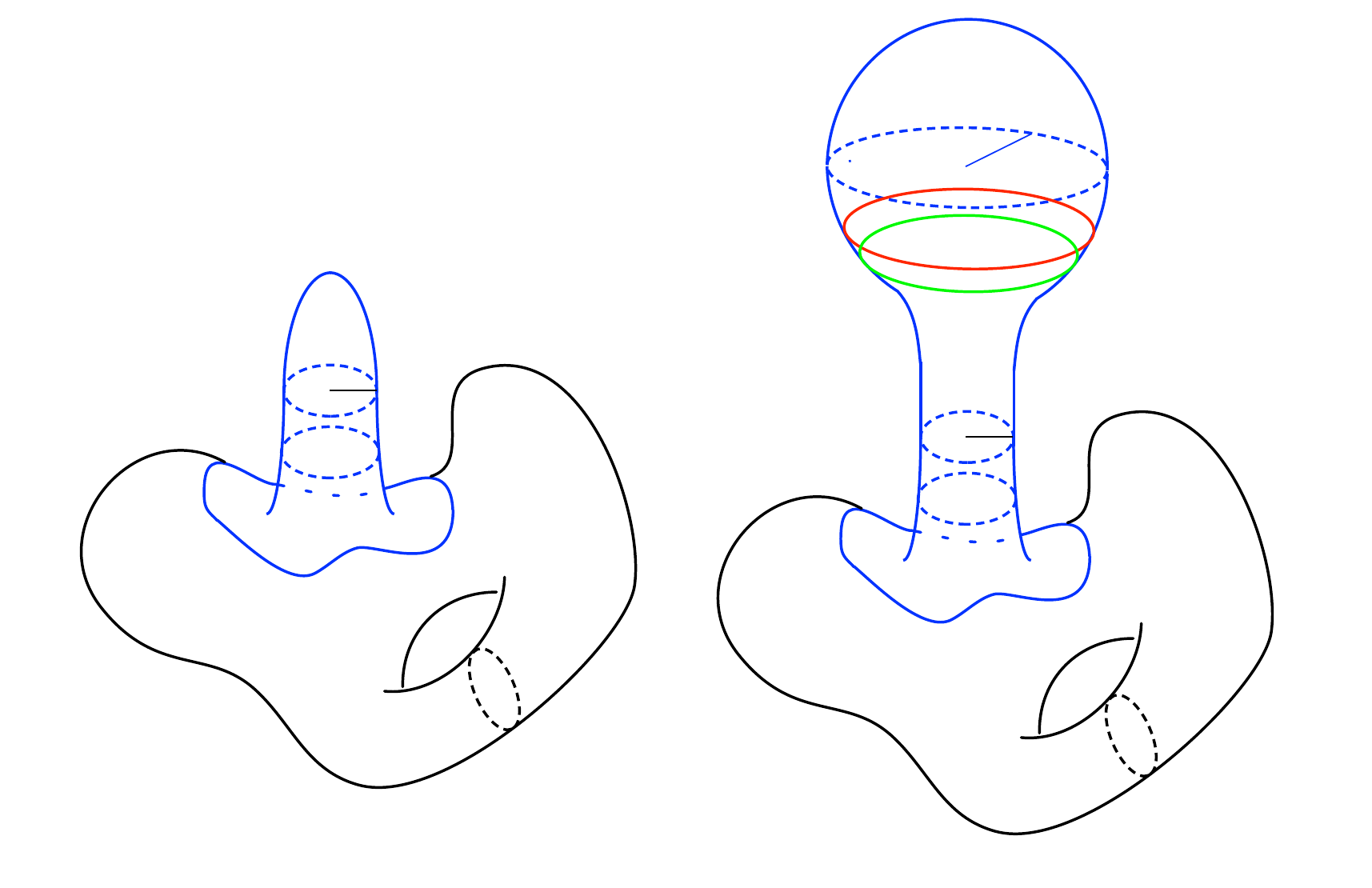}%
\end{picture}
\setlength{\unitlength}{3947sp}%
\begingroup\makeatletter\ifx\SetFigFont\undefined%
\gdef\SetFigFont#1#2#3#4#5{%
  \reset@font\fontsize{#1}{#2pt}%
  \fontfamily{#3}\fontseries{#4}\fontshape{#5}%
  \selectfont}%
\fi\endgroup%
\begin{picture}(5079,1559)(1902,-7227)
\put(2550,-5200){\makebox(0,0)[lb]{\smash{{\SetFigFont{10}{8}{\rmdefault}{\mddefault}{\updefault}{\color[rgb]{0,0,0}$B_g(x,\epsilon)$}%
}}}}
\put(2300,-5900){\makebox(0,0)[lb]{\smash{{\SetFigFont{10}{8}{\rmdefault}{\mddefault}{\updefault}{\color[rgb]{0,0,0}$M$}%
}}}}
\put(3870,-4220){\makebox(0,0)[lb]{\smash{{\SetFigFont{10}{8}{\rmdefault}{\mddefault}{\updefault}{\color[rgb]{0,0,0}$\delta$}%
}}}}
\put(3800,-3450){\makebox(0,0)[lb]{\smash{{\SetFigFont{10}{8}{\rmdefault}{\mddefault}{\updefault}{\color[rgb]{0,0,0}$p$}%
}}}}
\put(7650,-4500){\makebox(0,0)[lb]{\smash{{\SetFigFont{10}{8}{\rmdefault}{\mddefault}{\updefault}{\color[rgb]{0,0,0}$\delta$}%
}}}}
\put(7670,-2820){\makebox(0,0)[lb]{\smash{{\SetFigFont{10}{8}{\rmdefault}{\mddefault}{\updefault}{\color[rgb]{0,0,0}$\lambda$}%
}}}}
\put(7650,-1900){\makebox(0,0)[lb]{\smash{{\SetFigFont{10}{8}{\rmdefault}{\mddefault}{\updefault}{\color[rgb]{0,0,0}$p$}%
}}}}
\put(5700,-3320){\makebox(0,0)[lb]{\smash{{\SetFigFont{10}{8}{\rmdefault}{\mddefault}{\updefault}{\color[rgb]{0,0,0}$\p B_{\lambda}(p,\lambda\pi-r)$}%
}}}}
\put(5800,-3620){\makebox(0,0)[lb]{\smash{{\SetFigFont{10}{8}{\rmdefault}{\mddefault}{\updefault}{\color[rgb]{0,0,0}$\p B_{\lambda}(p,\lambda\pi-r^{*})$}%
}}}}

\end{picture}%
\caption{Pushing out a bulb (right) around $x$ in $B_g(x,\epsilon)$}
\label{pushcapbulb}
\end{figure}
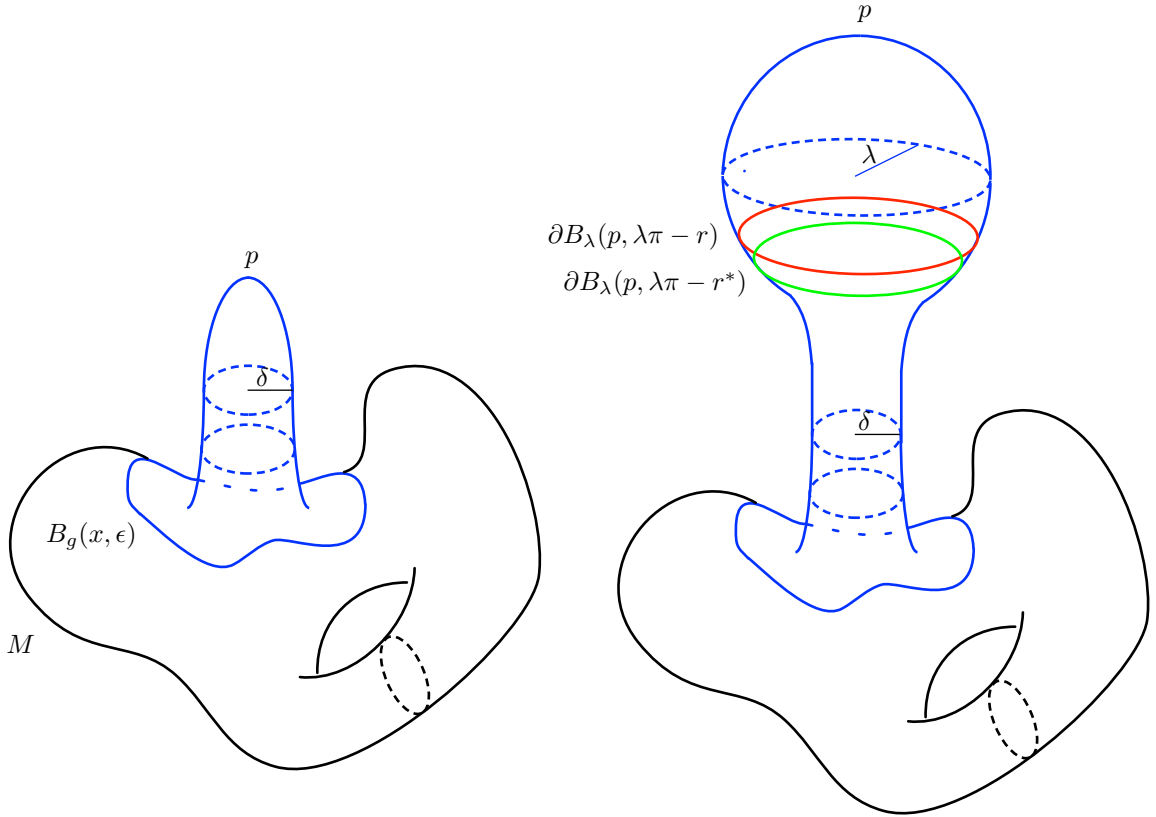   
\begin{Lemma}\label{bulblemma2}
Let $M^{n}$ be a smooth compact manifold with dimension $n\geq 3$ and $g$ a psc-metric on $M$. Suppose $x\in M$ and $B_g(x,\epsilon)$ is a geodesic ball around $p$ for some $\epsilon>0$. For any $\lambda>0$ and any $r\in (0,\lambda\frac{\pi}{2}]$, there is a psc-isotopy $g_{\bulb}(t), t\in I$ so that the following conditions hold:
\begin{enumerate}
\item{} The psc-isotopy $g_{\bulb}(t), t\in I$ varies continuously with respect to $g, p, \epsilon, r$ and $\lambda$.
\item{} Outside of $B_g(p,\epsilon)$, $g_{\bulb}(t)=g$ for all $t\in I$.
\item{} There are continuous parameters ${\epsilon^{*}}$ satisfying $0<\epsilon^{*}<\epsilon$ and ${r^{*}}$ satisfying $0<r^{*}\leq r$, which depend continuously on $g, p, \epsilon, r$ and $\lambda$, so that the restriction metric $g_{\bulb}(1)|_{B_g(p,\epsilon_{0})}$ is the bulb metric $g_{\bulb}(\lambda, r^{*})$.
\end{enumerate}
\end{Lemma}
\begin{proof}
Most of the work here lies in continuously pushing out torpedo caps. This is done in original Gromov-Lawson construction as described in Theorem 2.13 of \cite{Walsh1}. As we already discussed, these torpedo caps have a nice standard structure that is easily manipulated to obtain the desired bulb structure.
\end{proof}
\begin{Remark}
The reader may wonder why we require the parameter $r^{*}$ with the property that $0<r^{*}\leq r$ and why we only end up by pushing out a $(\lambda,r^{*})$-bulb metric, in the lemma above. This is a consequence of the Gromov-Lawson construction. The radius of the connecting tube, $\delta$, may have to be very small. Thus, we cannot guarantee that it connects up correctly with the head and neck of the $(\lambda,r)$-bulb. However, for any $\lambda$ there is always a small enough head angle $r^{*}$ which works. 
\end{Remark}

We are now ready to make some observations about the spaces $\Riem_{\bulb(\bf p)}(S^{n})$ and $\Riem_{\head(\bf p)}(S^{n})$ defined above. Firstly, the obvious analogues of lemmas \ref{caponb} and \ref{caprot} hold here also.
\begin{Proposition}
Lemma \ref{caponb} and Lemma \ref{caprot} hold if we replace the space $\Riem_{\torp({\bf p})}(S^{n})$ with the spaces $\Riem_{\bulb(\bf p)}(S^{n})$ or $\Riem_{\head(\bf p)}(S^{n})$ for relevant collections of points $\bf p$.
\end{Proposition}
\noindent More importantly, we have the following lemma on the homotopy type of  $\Riem_{\bulb(\bf p)}(S^{n})$.
\begin{Lemma}
For $n\geq 3$, the spaces $\Riem^{+}_{{\rm{torp}}({\bf p})}(S^{n})$, $\Riem^{+}_{{\rm{bulb}}(\bf p)}(S^{n})$ and $\Riem_{\head(\bf p)}(S^{n})$ are homotopy equivalent.
\end{Lemma}
\begin{proof}
The construction in Lemma \ref{bulblemma2} gives rise to an obvious map from $\Riem^{+}_{{\torp}({\bf p})}(S^{n})$ to $\Riem^{+}_{{\rm{bulb}}({\bf p})}(S^{n})$. Conversely one can easily construct a map in the opposite direction, which involves pushing out a cap on the bulb head of each element in $\Riem^{+}_{{\rm{bulb}}({\bf p})}(S^{n})$. Showing that compositions of these maps is homotopy equivalent to the appropriate identity map is then straightforward given the psc-isotopies constructed in  Lemma \ref{bulblemma2}.  Moreover, a similar argument using the construction done in Lemma \ref{bulblemma2} may be used to add necks to the heads of psc-metrics in $\Riem_{\head(\bf p)}(S^{n})$ in order to demonstrate a homotopy equivalence between $\Riem^{+}_{{\rm{bulb}}({\bf p})}(S^{n})$ and $\Riem^{+}_{{\rm{head}}({\bf p})}(S^{n})$. 
\end{proof}
\begin{Corollary}
For $n\geq 3$, the spaces $\Riem_{\bulb(\bf p)}(S^{n})$ and $\Riem_{\head(\bf p)}(S^{n})$ are homotopy equivalent to $\Riem^{+}(S^{n})$. 
\end{Corollary}
\begin{proof}
This is immediate by Lemma \ref{homequivspaces}.
\end{proof}
\noindent In fact, we can say a little more. To simplify notation, we will work in the case when ${\bf p}$ is just a single point $p_0$, although there is an obvious generalisation to the case when ${\bf p}$ has many points. As in the case of psc-metrics with torpedo caps, we may also define, for each pair $\lambda_0, r_0>0$ with $r_0\in(0,\lambda_0\frac{\pi}{2}]$, the following subspaces of psc-metrics with a fixed bulb or fixed head:  
\begin{equation*}\label{fixbulb}
\begin{split}
&\Riem^{+}_{{\rm{bulb}}({p_0}, \lambda_0, r_0)}(S^{n}):=\\
&\{g\in \Riem^{+}_{{\rm{bulb}}(p_0)}(S^{n}):{\phi_{p_0}^{*}}(g|_{D_{p_0}})=\Unc_{p_0'}\circ \Bulb_{p_0}(\lambda_0, r_0)\},
\end{split}
\end{equation*}
and 
\begin{equation*}\label{fixhead}
\begin{split}
\Riem^{+}_{{\rm{head}}({p_0}, \lambda_0, r_0)}(S^{n}):=\{g\in \Riem^{+}_{{\rm{head}}(p_0)}(S^{n}):{\phi_{p_0}^{*}}(g|_{D_{p_0}})=g_{\lens}(\lambda_0, r_0)\}.
\end{split}
\end{equation*}  
There are also intermediary spaces where the head radius is fixed at $\lambda_0$ but the head angle $r$ is allowed to vary in the interval $(0,\lambda_0\frac{\pi}{2}]$. We denote these spaces $ \Riem^{+}_{{\rm{b/h}}(p_0,\lambda_0)}(S^{n})$, simply dropping the head radius coordinate. To clarify, we have the following sequence of inclusions:
\begin{equation*}\label{fixheadvary}
\Riem^{+}_{{\rm{b/h}}({p_0}, \lambda_0, r_0)}(S^{n})\subset \Riem^{+}_{{\rm{b/h}}({p_0}, \lambda_0)}(S^{n})\subset \Riem^{+}_{{\rm{b/h}}({p_0})}(S^{n}).
\end{equation*}
Given a metric a metric $g\in \Riem^{+}_{{\rm{b/h}}(p_0)}(S^{n})$, there is a canonical way of moving it, in the space $\Riem^{+}_{{\rm{b/h}}(p_0)}(S^{n})$ to a metric which lies in $\Riem^{+}_{{\rm{b/h}}({p_0}, \lambda_0, r_0)}(S^{n})$.
This method takes the form of a map, which we denote $\Mov^{\rm b/h}=\Mov_{p_0, \lambda_0, r_0}^{\rm b/h}$, and define as follows:
\begin{equation}\label{movmap}
\Mov^{\rm b/h}:\Riem^{+}_{{\rm{b/h}}(p_0)}(S^{n})\longrightarrow\Riem^{+}_{{\rm{bulb}}({p_0}, \lambda_0, r_0)}(S^{n}),
\end{equation}
where for each input metric $g$ whose bulb about $p_0$ takes the form $g_{\bulb}(p_0,\lambda_1, r_1)$, the output $\Mov^{\rm b/h}(g)$ is obtained by construction described below. We denote by $\delta_1$, the neck radius of this metric.
\begin{enumerate}
\item{} Replace the entire metric $g$ with the metric $(\frac{\lambda_0}{\lambda_1})^{2}g$ to obtain a new metric, whose bulb head has radius $\lambda_0.$
\item{} The next part is more delicate as the cases of $\Mov^{\bulb}$ and $\Mov^{\head}$ are slightly different. We begin with the map $\Mov^{\bulb}$. By Lemma \ref{bulblemma}, we may replace the bulb component of the metric newly scaled metric $(\frac{\lambda_0}{\lambda_1})^{2}g$ with the unique one with head radius $\lambda_0$ and head angle $r_0$. This metric has a neck whose radius we denote $\delta_0>0$. The problem is that this may not agree with the neck radius, $\delta_1$, of the connecting tube which connects the bulb with the rest of the metric $g$. Thus, to compensate, we replace the restriction metric $g|_{S^{n}\setminus D_{p}}$ with $(\frac{\delta_0}{\delta_1})^{2}g|_{S^{n}\setminus D_{p}}$. This latter restriction metric attaches smoothly to the head and neck of the new bulb component, resulting in an element of $\Riem^{+}_{{\rm{bulb}}({p_0}, \lambda_0, r_0)}$. In the case of $\Mov^{\head}$ we first use the techniques of Lemma \ref{bulblemma2} to push out a neck on the metric $(\frac{\lambda_0}{\lambda_1})^{2}g$ before proceeding as above. We emphasise that the pushing out of such a neck can be done in a canonical way inside the standard head, relying only input data arising from the head angle and radius and so does not require any non-standard metric data from $g$.
\end{enumerate}
This gives rise to the following lemma, an analogue of Lemma \ref{spheredefret}.
\begin{Lemma}\label{bulbdefret}
For $n\geq 3, \lambda_0>0$ and $r_0\in(0,\lambda\frac{\pi}{2}]$, there is a deformation retract from the space $\Riem^{+}_{{\rm{b/h}}(p_0)}(S^{n})$ onto its corresponding subspace $\Riem^{+}_{{\rm{b/h}}(p_0, \lambda_0)}(S^{n})$ and then a further deformation retract to the corresponding subspace $\Riem^{+}_{{\rm{b/h}}({p_0}, \lambda_0, r_0)}(S^{n})$.
\end{Lemma}
\begin{proof}
This is similar to the proof of Lemma \ref{spheredefret}. The map $\Mov^{\rm b/h}$ fixes metrics which lie in $\Riem^{+}_{{\rm{b/h}}({p_0}, \lambda_0, r_0)}(S^{n})$ and its composition with the inclusion of $\Riem^{+}_{{\rm{b/h}}({p_0}, \lambda_0, r_0)}(S^{n})$ into $\Riem^{+}_{{\rm{b/h}}(p_0)}(S^{n})$ is easily shown to be homotopic to the identity.
\end{proof}
\subsection{A product on the spaces $\Riem^{+}_{{\rm{b/h}}(p_0)}(S^{n})$} We now return to a construction alluded to at the end of section \ref{disksphere}. Once again, we let ${\bf p}$ and ${\bf q}$ denote finite collections of points $\{ p_0, p_1, \cdots p_k\}$ and $\{ q_0, q_1, \cdots q_l\}$ on $S^{n}$. For some $i\in\{0,1,\cdots ,k\}$, we begin by composing the maps $\Cut_{p_i, \rho}$ and $\Mov_{p_i, \lambda, r}^{\rm b/h}$, defined in \ref{cutmap} and \ref{movmap}, for some $\rho\geq 0$ and $\lambda, r>0$ with $r\in(0,\lambda\frac{\pi}{2}]$. Note that the values $\lambda$ and $r$ are not apriori connected with $p$. Any bulb around $p$ will be made to ``fit" those values. With this in mind we define a new map $\Fit_{p_i, \lambda, r_, \rho}^{\rm b/h}=\Cut_{p_i, \rho}\circ \Mov_{p_i, \lambda, r}^{\rm b/h}$. Just to clarify, the map $\Fit_{p_i, \lambda, r, \rho}^{\rm b/h}$ takes the form:
\begin{equation}\label{fitmap}
\begin{split}
\Fit_{p_i, \lambda, r, \rho}^{\rm b/h}:\Riem^{+}_{{\rm{b/h}}({\bf p})}(S^{n})&\longrightarrow\Riem_{{\lens}(0)}^{+}(D^{n})\\
g&\longmapsto \begin{cases}
\Mov_{p_i, \lambda, r}^{\rm b/h}(g)|_{S^{n}\setminus B_{\lambda}(p_i,\rho)}, & \text{if }0<\rho\leq\lambda\pi-r,\\
\Mov_{p_i, \lambda, r}^{\rm b/h}(g)|_{S^{n}\setminus B_{\lambda}(p_i,\lambda\pi-r)}, & \text{if }\rho>\lambda\pi-r.
\end{cases}
\end{split}
\end{equation}
For example, the metric $\Fit_{p_i, \lambda, r, \lambda\frac{\pi}{2}}^{\rm b/h}(g)$ is obtained by removing the round northern hemisphere of radius $\lambda$ about $p_i$, from $\Mov_{p_i,\lambda,r}^{\rm b/h}(g)$. On the other hand, the output metric $\Fit_{p_i, \lambda, r, \lambda\pi-r}^{\rm b/h}(g)$ is obtained by removing the head, but not the neck, of the metric $\Mov_{p_i,\lambda,r}(g)^{\rm b/h}$ about $p_i$. In particular, we see that the boundary of the output metric $\Fit_{p_i, \lambda, r, \lambda\pi-r }^{\rm b/h}(g)$, smoothly attaches to the boundary metric $g_{\lens}(\lambda, \lambda\pi-r)$.  We therefore obtain, for each pair $(i,j)$ with $i\in \{ p_0, p_1, \cdots p_k\}$ and $j\in \{ q_0, q_1, \cdots q_l\}$, the following ``joining" map for psc-metrics on $S^{n}$ with bulbs or heads:

\begin{equation}\label{bulbjoin}
\begin{split}
J_{ij}^{\rm b/h({\lambda, r})}:\Riem^{+}_{{\rm{b/h}}({\bf p})}(S^{n})\times\Riem^{+}_{{\rm{b/h}}({\bf q})}(S^{n})&\longrightarrow\Riem^{+}_{{\rm{b/h}}(\{{\bf p}\setminus\{p_i\}\}\cup\{ {\bf q}\setminus\{ q_j\}\})}(S^{n})\\
(g,h)&\longmapsto \Fit_{p_i, \lambda,r,\lambda\pi-r}^{\rm b/h}(g)\cup\Fit_{q_j, \lambda, \lambda \frac{r_j(h)}{\lambda_j(h)}, r}^{\rm b/h}(h),
\end{split}
\end{equation}
where $r_j(h)$ is the head angle and  $\lambda_{j}(h)$ is the head radius at $q_j$ of the metric $h$. Note that the output metric is obtained by gluing the metrics $\Fit_{p_i, \lambda,r,\lambda\pi-r}^{\rm b/h}(g)$ and $\Fit_{q_j, \lambda, \lambda \frac{r_j(h)}{\lambda_j(h)}, r}^{\rm b/h}(h)$ together in the obvious way. To aid the read we depict an example in the case of psc-metrics with bulbs in Fig. \ref{bulbyspherejoin} below.
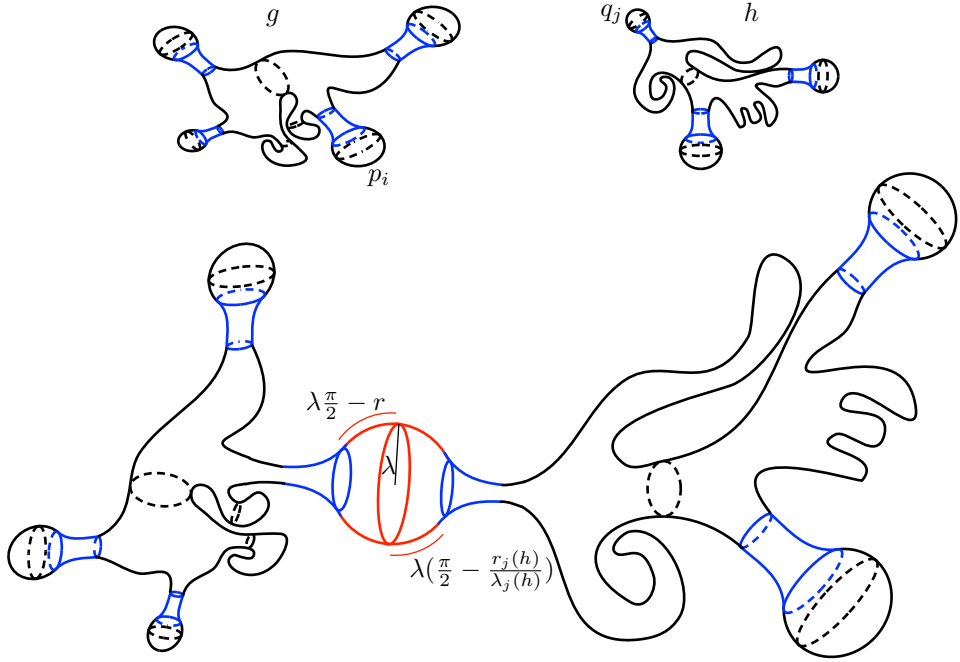
\begin{figure}[!htbp]
\vspace{6cm}
\hspace{-2.0cm}
\begin{picture}(0,0)
\includegraphics{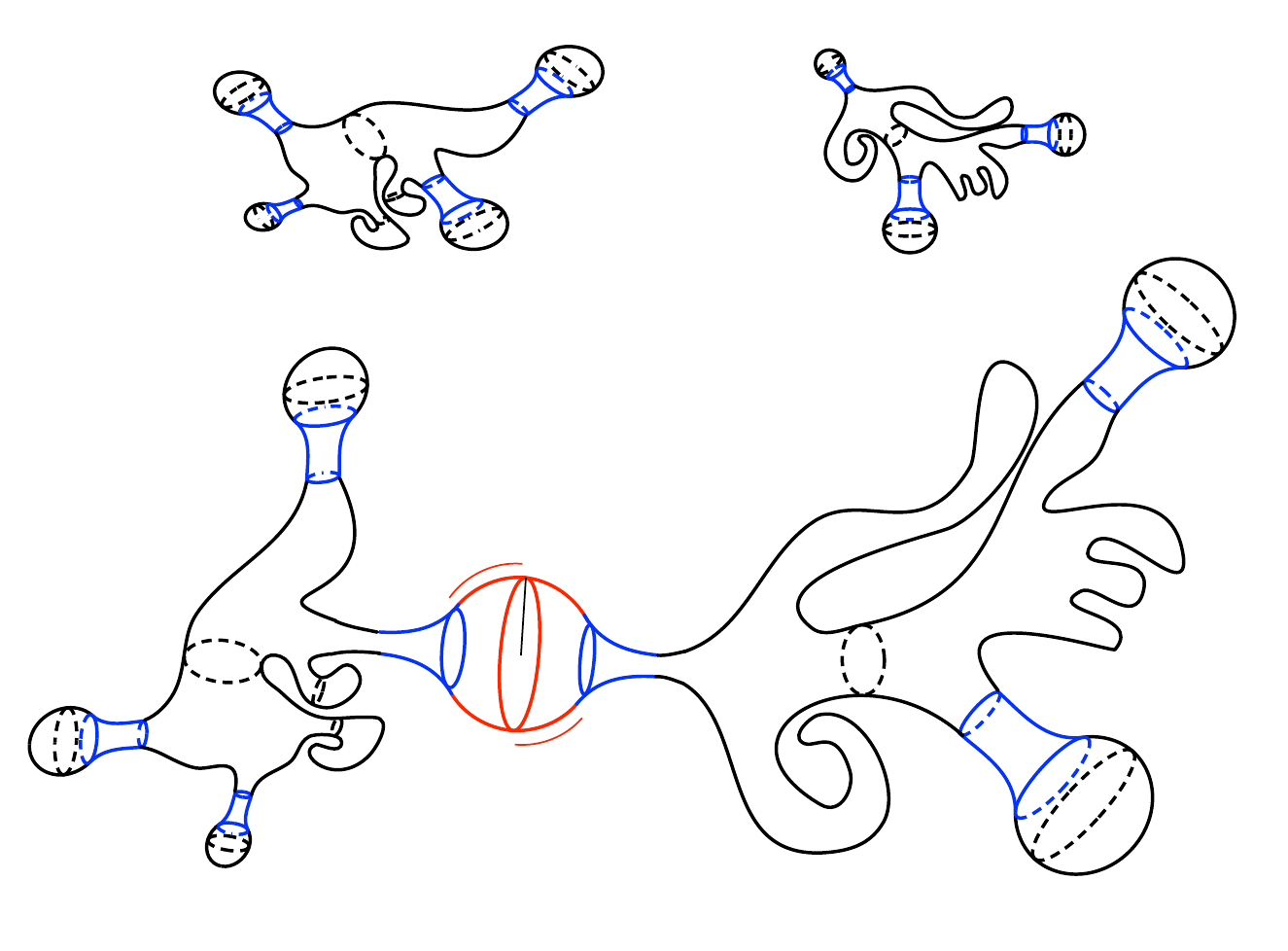}%
\end{picture}
\setlength{\unitlength}{3947sp}%
\begingroup\makeatletter\ifx\SetFigFont\undefined%
\gdef\SetFigFont#1#2#3#4#5{%
  \reset@font\fontsize{#1}{#2pt}%
  \fontfamily{#3}\fontseries{#4}\fontshape{#5}%
  \selectfont}%
\fi\endgroup%
\begin{picture}(5079,1559)(1902,-7227)
\put(3850,-5300){\makebox(0,0)[lb]{\smash{{\SetFigFont{10}{8}{\rmdefault}{\mddefault}{\updefault}{\color[rgb]{0,0,0}$\lambda\frac{\pi}{2}-r$}%
}}}}
\put(4500,-6350){\makebox(0,0)[lb]{\smash{{\SetFigFont{10}{8}{\rmdefault}{\mddefault}{\updefault}{\color[rgb]{0,0,0}$\lambda(\frac{\pi}{2}-\frac{r_{j}(h)}{\lambda_{j}(h)})$}%
}}}}
\put(4320,-5700){\makebox(0,0)[lb]{\smash{{\SetFigFont{10}{8}{\rmdefault}{\mddefault}{\updefault}{\color[rgb]{0,0,0}$\lambda$}%
}}}}
\put(4240,-3860){\makebox(0,0)[lb]{\smash{{\SetFigFont{10}{8}{\rmdefault}{\mddefault}{\updefault}{\color[rgb]{0,0,0}$p_i$}%
}}}}
\put(3600,-2850){\makebox(0,0)[lb]{\smash{{\SetFigFont{10}{8}{\rmdefault}{\mddefault}{\updefault}{\color[rgb]{0,0,0}$g$}%
}}}}
\put(6600,-2850){\makebox(0,0)[lb]{\smash{{\SetFigFont{10}{8}{\rmdefault}{\mddefault}{\updefault}{\color[rgb]{0,0,0}$h$}%
}}}}
\put(5700,-2820){\makebox(0,0)[lb]{\smash{{\SetFigFont{10}{8}{\rmdefault}{\mddefault}{\updefault}{\color[rgb]{0,0,0}$q_j$}%
}}}}

\end{picture}%
\caption{The metric $J_{ij}^{\bulb({\lambda, r})}(g,h)$  (bottom) obtained from $g$ (left) and $h$ (right)}
\label{bulbyspherejoin}
\end{figure}   

We now return to the case when ${\bf p}={\bf q}=\{p_0\}$. As in the case of psc-metrics with torpedo caps, the maps $J_{00}^{\rm b/h({\lambda,r})}$ for $\lambda>0, r\in(0,\lambda\frac{\pi}{2}]$, do not quite give us the product we need. We solve this problem as we did in the torpedo case, when defining the multiplication $\mu^{\torp}$ in \ref{muprod1}. Let $p_1$ and $p_2$ be two distinct points on $S^{n}\setminus \{p_0\}$ and redefine ${\bf p}=\{p_0, p_1, p_2\}$. 
We now define products on the spaces $\Riem_{{\rm b/h}({ p_0})}^{+}(S^{n})$ as follows. Consider for each $j=1,2$, the map:
\begin{equation*}
J_{0j}^{{\rm b/h}(\lambda,r)}:\Riem_{{\mathrm{b/h}}({p_0 })}^{+}(S^{n})\times \Riem_{{\mathrm{b,h}}({\bf p})}^{+}(S^{n})\longrightarrow \Riem_{{\mathrm{b/h}}({\bf p}\setminus\{p_j\})}^{+}(S^{n}),
\end{equation*}
defined as in formula \ref{bulbjoin}. We now fix a psc-metric $g_{3}\in \Riem_{{\rm b/h}({\bf p})}^{+}(S^{n})$ as the second input. Then for each of $j=1,2$, we obtain maps: 
\begin{equation}
\begin{split}
J_{0j}^{{\rm b/h}(g_3)}:\Riem_{{\mathrm{b/h}}({p_0 })}^{+}(S^{n})&\longrightarrow \Riem_{{\mathrm{b/h}}({\bf p}\setminus\{p_j\})}^{+}(S^{n}),\\
g&\longmapsto J_{0j}^{{\rm b/h}(\lambda_{j}(g_3), r_{j}(g_3))}(g, g_3).
\end{split}
\end{equation}
Finally, we define products $\mu^{\bulb}$ and $\mu^{\head}$ (which we notationally combine as $\mu^{\rm b/h}$) on the spaces $\Riem_{{\mathrm{bulb}}({p_0 })}^{+}(S^{n})$ and $\Riem_{{\mathrm{head}}({p_0 })}^{+}(S^{n})$ by means of the following continuous maps:
\begin{equation}\label{muprod2}
\begin{split}
\mu^{\rm b/h}:\Riem_{{\mathrm{b/h}}({p_0 })}^{+}(S^{n})\times \Riem_{{\mathrm{b/h}}({p_0 })}^{+}(S^{n})&\longrightarrow \Riem_{{\mathrm{b/h}}({p_0 })}^{+}(S^{n}),\\
(g,h)&\longmapsto J_{02}^{g_3}(h,J_{01}^{g_3}(g, g_3)).
\end{split}
\end{equation}
There are obvious analogues of Lemma \ref{homotprod} which clarify the role played by the various choices in determining these maps up to homotopy type but we will not state them here. We close by pointing out that for certain choices of $g_3$, the maps $\mu^{\rm b/h}$ determine an $H$-space structure on the respective spaces $\Riem_{{\mathrm{b/h}}({p_0 })}^{+}(S^{n})$.
\begin{Theorem}\label{HTheorem2}
Let $n\geq 3$ and let $\mu^{\rm b/h}$ be the multiplication map given by formula \ref{muprod2} with respect to a psc-metric $g_3^{\rm b/h}\in\Riem_{{\mathrm{b/h}}({p_0 })}^{+}(S^{n})$. In the case when the metric $g_3^{\rm b/h}$ is psc-isotopic to the round metric $ds_{n}^{2}$, $\mu^{\rm b/h}$ defines a homotopy product on $\Riem_{{\mathrm{b/h}}({p_0 })}^{+}(S^{n})$ with homotopy identity $g_{\rm id}^{\rm b/h}=ds_n^{2}$, the round metric of radius $1$. This gives $\Riem_{{\mathrm{b/h}}({p_0 })}^{+}(S^{n})$ the structure of an $H$-space. Furthermore, this product is both homotopy commutative and homotopy associative.
\end{Theorem}
\begin{Remark}
The reader should note that the standard round metric, of any radius, is an element of both spaces $\Riem_{{\mathrm{b/h}}({p_0})}^{+}(S^{n})$, where $p_0$ is the north pole and $D_{p_0}$ is the northern hemisphere.
\end{Remark}
\begin{proof}
The proof is completely analogous to that of Theorem \ref{HTheorem}.
\end{proof}
\noindent In section \ref{Dspace}, we will make considerable use of the multiplication $\mu^{\head}$ on a subspace of the space $\Riem_{{\mathrm{head}}({p_0})}^{+}(S^{n})$ to prove our main result. In the mean time we will switch our focus somewhat and discuss a collection of objects known as operads.

\section{Little Disks, Trees and the $W$-construction for operads}\label{Operadsection}
We now turn our attention to the second of our main results, Theorem \ref{LoopThm}, which states that $\Riem^{+}(S^{n})$ is weakly homotopy equivalent to an $n$-fold loop space when $n=3$ and $n\geq 5$. The problem of recognising when an $H$-space is an iterated loop space is an old problem in Algebraic Topology; see \cite{Stash} for an account of this story. A key step in tackling this problem was the discovery by Boardman and Vogt that an $n$-fold loop space is a $\D_n$-space. That is, it admits an action of the {\em operad of little $n$-dimensional disks}. Before explaining what this means, we should point out that we are interested in a converse to this proposition. That is, given a space which admits such an action, is it an $n$-fold loop space? Under reasonable conditions such a converse holds. This is the subject of a theorem of Boardman, Vogt and May, which we will state shortly as Theorem \ref{BVM} and which helps us to prove our main result, Theorem \ref{LoopThm}.  Before that however, we need to discuss the aforementioned operad of little disks. Much of this section is based on the work of Boardman, Vogt and May in \cite{BV}, \cite{May} and \cite{V}.

The following definition of an operad is due to P. May \cite{May}. The definition itself is rather involved, however it is followed by two very illustrative examples which the reader may wish to study first. These examples, the operad of little disks and the operad of grown trees, will play a central role in our main construction. An {\em operad} $\P$ consists of a collection of compactly generated Hausdorff topological spaces $\P(j)$, $j\in\{0,1,2,\cdots\}$, together with the following data:
\begin{enumerate}
\item{} The space $\P(0)$ is a single point $*$.
\item{} There are continuous maps (known as {\em composition maps}) 
\begin{equation*}
\gamma:\P(k)\times\P(j_1)\times\P(j_2)\times\cdots\times\P(j_k)\longrightarrow \P(j),
\end{equation*}
where $\Sigma j_s=j$, and which satisfy the following associativity condition for all $c\in \P(k), d_s\in\P(j_s)$ and $e_t\in\P(i_t)$:
\begin{equation*}
\gamma(\gamma(c;d_1,\cdots,d_k);e_1, \cdots, e_j)=\gamma(c;f_1,\cdots,f_k),
\end{equation*}
where $f_s=\gamma(d_s;e_{j_1+j_2+\cdots+e_{s-1}+1},\cdots,e_{j_1+\cdots+j_s})$, and $f_s=*$ if $j_s=0$.
\item{} There is an identity elemnt $1\in\P(1)$ so that $\gamma(1;d)=d$ for all $d\in\P(j)$ and $\gamma(c;1^{k})$ for $c\in\P(k)$, $1^{k}=(1,\cdots,1)\in\P(k)$. 
\item{} There is a right operation of the symmetric group $\Sigma_j$ on $\P(j)$ so that the following equivariance formulae are satisfied for all $c\in\P(k), d_s\in\P(j_s), \sigma\in\Sigma_k$ and $\tau_s\in\Sigma_{j_s}$:
\begin{equation*}
\gamma(c\sigma;d_1,\cdots,d_k)=\gamma(c;d_{\sigma^{-1}(1)},\cdots,d_{\sigma^{-1}(k)})\sigma(j_1,\cdots,j_k)
\end{equation*}
and $\gamma (c;d_1\tau_1,\cdots,d_k\tau_k)=\gamma (c;d_1,\cdots,d_k) (\tau_1\oplus\cdots \oplus\tau_k)$, where $\sigma (j_1,\cdots , j_k)$ denotes the permutation of $j$ letters which permutes the $k$ blocks of letters determined by the given partition of $j$ as $\sigma$ permutes $k$ letters, and $\tau_1\oplus\cdots\oplus\tau_k$ denotes the image of $(\tau_1,\cdots,\tau_k)$ under the obvious inclusion of $\Sigma_{j_1}\times \cdots\times\Sigma_{j_k}$ into $\Sigma_{j}$.
\end{enumerate}

\subsection{The operad of little $n$-dimensional disks}
We now consider a very important example.  For $n\geq 1$ we recall that $D^{n}$ denotes the standard closed unit disk in $\mathbb{R}^{n}$. For each point $p$ in the interior of $D^{n}$ and each quantity $\epsilon$ where $0<\epsilon\leq 1-|p|$, let $D(p,\epsilon)$ denote the round disk of radius $\epsilon$ which is centred at $p$. Let $j\geq 0$ be an integer. We denote by $\D(j)_{n}$, the set of ordered $j$-tuples of closed round disks $D(p_i, \epsilon_i)$ where $i=1,\cdots j$ and which satisfy the following conditions:
\begin{enumerate}
\item{} $\oD (p_i,\epsilon_i)\subset \oDn$ for all $i=1,\cdots ,j$,
\item{} $\oD (p_i,\epsilon_i)\cap \oD (p_k,\epsilon_k)=\emptyset$ for all $i,k\in\{1,\cdots,j\}$.
\end{enumerate}
To ease notation we will fix an $n$ and simply write $\D(j)$ instead of $\D(j)_{n}$. Each element of $\D(j)$ is therefore an ordered $j$-tuple of little disks. By viewing each such element as a collection of pairs $(p_i,\epsilon_i)$ we may topologise $\D(j)$ by identifying it with an appropriate subspace of the space $(D^{n}\times I)^{j}$. There is an obvious action on $\D(j)$ of the permutation group $\Sigma_j$, where for any pair $c\in \D(j)$, $\sigma\in\Sigma_j$, the element $c\sigma$ has little disks labelled $\sigma(1),\sigma(2),\cdots\sigma( j)$. We illustrate this for an element of $\D(3)$, where $\sigma=(1\hspace{1mm} 2\hspace{1mm} 3)$, in Fig. \ref{opsymmaction} below.
 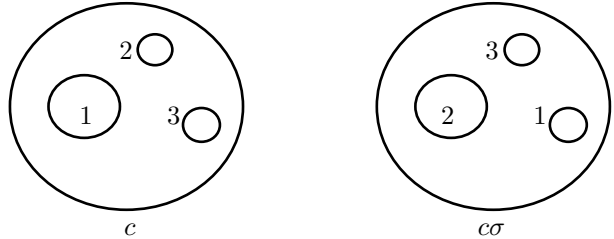
\begin{figure}[!htbp]
\vspace{1.0cm}
\hspace{1.50cm}
\begin{picture}(0,0)
\includegraphics{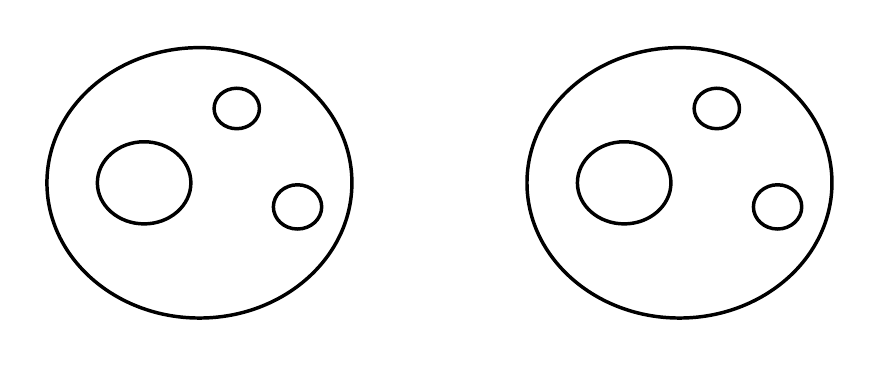}%
\end{picture}
\setlength{\unitlength}{3947sp}%
\begingroup\makeatletter\ifx\SetFigFont\undefined%
\gdef\SetFigFont#1#2#3#4#5{%
  \reset@font\fontsize{#1}{#2pt}%
  \fontfamily{#3}\fontseries{#4}\fontshape{#5}%
  \selectfont}%
\fi\endgroup%
\begin{picture}(5079,1559)(1902,-7227)
\put(2780,-7106){\makebox(0,0)[lb]{\smash{{\SetFigFont{10}{8}{\rmdefault}{\mddefault}{\updefault}{\color[rgb]{0,0,0}$c$}%
}}}}
\put(2500,-6406){\makebox(0,0)[lb]{\smash{{\SetFigFont{10}{8}{\rmdefault}{\mddefault}{\updefault}{\color[rgb]{0,0,0}$1$}%
}}}}
\put(2750,-6000){\makebox(0,0)[lb]{\smash{{\SetFigFont{10}{8}{\rmdefault}{\mddefault}{\updefault}{\color[rgb]{0,0,0}$2$}%
}}}}
\put(3050,-6406){\makebox(0,0)[lb]{\smash{{\SetFigFont{10}{8}{\rmdefault}{\mddefault}{\updefault}{\color[rgb]{0,0,0}$3$}%
}}}}

\put(5000,-7106){\makebox(0,0)[lb]{\smash{{\SetFigFont{10}{8}{\rmdefault}{\mddefault}{\updefault}{\color[rgb]{0,0,0}$c\sigma$}%
}}}}
\put(4770,-6406){\makebox(0,0)[lb]{\smash{{\SetFigFont{10}{8}{\rmdefault}{\mddefault}{\updefault}{\color[rgb]{0,0,0}$2$}%
}}}}
\put(5050,-6000){\makebox(0,0)[lb]{\smash{{\SetFigFont{10}{8}{\rmdefault}{\mddefault}{\updefault}{\color[rgb]{0,0,0}$3$}%
}}}}
\put(5350,-6406){\makebox(0,0)[lb]{\smash{{\SetFigFont{10}{8}{\rmdefault}{\mddefault}{\updefault}{\color[rgb]{0,0,0}$1$}%
}}}}
\end{picture}%
\caption{The action of $\Sigma_3$ on $\D(3)$}
\label{opsymmaction}
\end{figure}    
Notice that for each little disk in an element of $\D(j)$, there is a canonical homeomorphism which identifies it with the larger unit disk $D^{n}$ i.e. shrink $D^{n}$ and translate. This allows us to construct the following ``fitting" map. Consider the product space $\D(k)\times \D(j_1)\times\cdots\times \D(j_k)$. Suppose $\{c, d_{j_1}, \cdots, d_{j_k}\}$ is an element of this space.  The first component $c$ consists of $k$ ordered little disks in $D^{n}$. By appropriately shrinking and translating the standard unit disk we may ``fit" each of the elements $d_{j_r}$ into the corresponding $r$-th little disk of $c$. The resulting object now consists of $j=\Sigma j_s$ little disks. Regarding the labelling, we apply the following rule. For each element $d_{j_k}\in \D(j_k)$, the corresponding $k$-th little disk in $\D(j)$ obtains its labels from the map: 
\begin{equation*}
(1,2,\cdots,j_k)\longmapsto (j_1+\cdots+ j_{k-1}+1, j_1+\cdots+ j_{k-1}+2,\cdots j_1+\cdots+ j_{k-1}+j_k).
\end{equation*}
This is shown for a particular example when $k=2, j_1=3$ and $j_2=2$ in Fig. \ref{diskoperad} below. The result is an element of $\D(j)$ which we denote $c(d_{j_1}, \cdots, d_{j_k})$. We summarise the {\em fitting map}, which we denote $\gamma$, as follows:
\begin{equation*}
\begin{split}
\gamma:\D(k)\times \D(j_1)\times\cdots\times \D(j_k)&\longrightarrow \D(j_1+\cdots+j_k)\\
(c;(d_{j_1},\cdots, d_{j_k}))&\longmapsto c(d_{j_1},\cdots, d_{j_k}).
\end{split}
\end{equation*}

\begin{figure}[!htbp]
\vspace{1.0cm}
\hspace{-3.0cm}
\begin{picture}(0,0)
\includegraphics{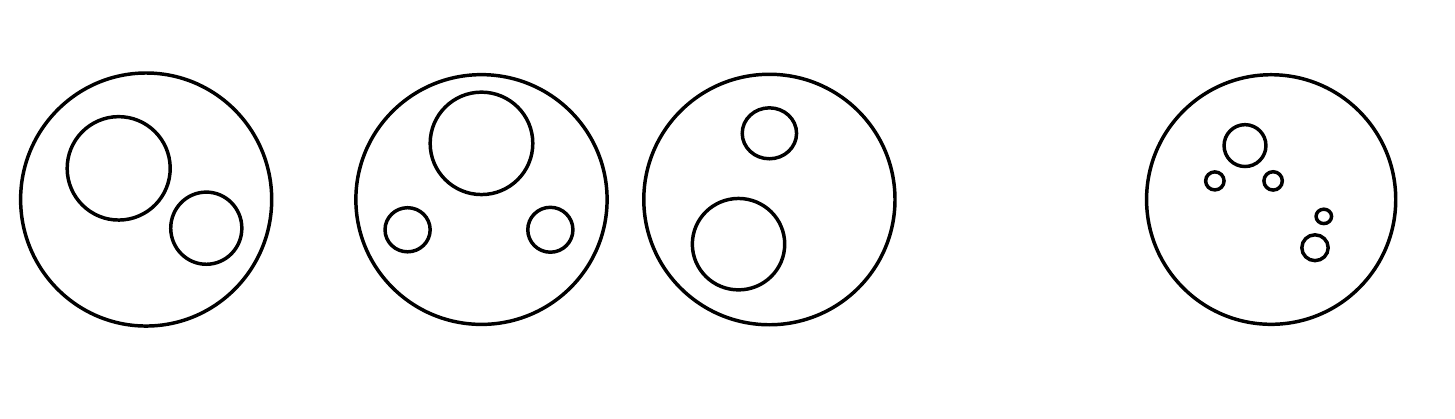}%
\end{picture}
\setlength{\unitlength}{3947sp}%
\begingroup\makeatletter\ifx\SetFigFont\undefined%
\gdef\SetFigFont#1#2#3#4#5{%
  \reset@font\fontsize{#1}{#2pt}%
  \fontfamily{#3}\fontseries{#4}\fontshape{#5}%
  \selectfont}%
\fi\endgroup%
\begin{picture}(5079,1559)(1902,-7227)
\put(2550,-7006){\makebox(0,0)[lb]{\smash{{\SetFigFont{10}{8}{\rmdefault}{\mddefault}{\updefault}{\color[rgb]{0,0,0}$c$}%
}}}}
\put(2350,-6106){\makebox(0,0)[lb]{\smash{{\SetFigFont{10}{8}{\rmdefault}{\mddefault}{\updefault}{\color[rgb]{0,0,0}$1$}%
}}}}
\put(2750,-6356){\makebox(0,0)[lb]{\smash{{\SetFigFont{10}{8}{\rmdefault}{\mddefault}{\updefault}{\color[rgb]{0,0,0}$2$}%
}}}}

\put(4050,-7006){\makebox(0,0)[lb]{\smash{{\SetFigFont{10}{8}{\rmdefault}{\mddefault}{\updefault}{\color[rgb]{0,0,0}$d_3$}%
}}}}
\put(3750,-6406){\makebox(0,0)[lb]{\smash{{\SetFigFont{10}{8}{\rmdefault}{\mddefault}{\updefault}{\color[rgb]{0,0,0}$1$}%
}}}}
\put(4100,-6106){\makebox(0,0)[lb]{\smash{{\SetFigFont{10}{8}{\rmdefault}{\mddefault}{\updefault}{\color[rgb]{0,0,0}$2$}%
}}}}
\put(4450,-6406){\makebox(0,0)[lb]{\smash{{\SetFigFont{10}{8}{\rmdefault}{\mddefault}{\updefault}{\color[rgb]{0,0,0}$3$}%
}}}}

\put(5550,-7006){\makebox(0,0)[lb]{\smash{{\SetFigFont{10}{8}{\rmdefault}{\mddefault}{\updefault}{\color[rgb]{0,0,0}$d_2$}%
}}}}
\put(5350,-6406){\makebox(0,0)[lb]{\smash{{\SetFigFont{10}{8}{\rmdefault}{\mddefault}{\updefault}{\color[rgb]{0,0,0}$1$}%
}}}}
\put(5450,-5906){\makebox(0,0)[lb]{\smash{{\SetFigFont{10}{8}{\rmdefault}{\mddefault}{\updefault}{\color[rgb]{0,0,0}$2$}%
}}}}

\put(7750,-7006){\makebox(0,0)[lb]{\smash{{\SetFigFont{10}{8}{\rmdefault}{\mddefault}{\updefault}{\color[rgb]{0,0,0}$c(d_3, d_2)$}%
}}}}
\put(7550,-6256){\makebox(0,0)[lb]{\smash{{\SetFigFont{10}{8}{\rmdefault}{\mddefault}{\updefault}{\color[rgb]{0,0,0}$1$}%
}}}}
\put(7750,-5960){\makebox(0,0)[lb]{\smash{{\SetFigFont{10}{8}{\rmdefault}{\mddefault}{\updefault}{\color[rgb]{0,0,0}$2$}%
}}}}
\put(7840,-6256){\makebox(0,0)[lb]{\smash{{\SetFigFont{10}{8}{\rmdefault}{\mddefault}{\updefault}{\color[rgb]{0,0,0}$3$}%
}}}}
\put(8230,-6486){\makebox(0,0)[lb]{\smash{{\SetFigFont{10}{8}{\rmdefault}{\mddefault}{\updefault}{\color[rgb]{0,0,0}$4$}%
}}}}
\put(8250,-6326){\makebox(0,0)[lb]{\smash{{\SetFigFont{10}{8}{\rmdefault}{\mddefault}{\updefault}{\color[rgb]{0,0,0}$5$}%
}}}}

\end{picture}%
\caption{The fitting map $\gamma$ in action}
\label{diskoperad}
\end{figure}
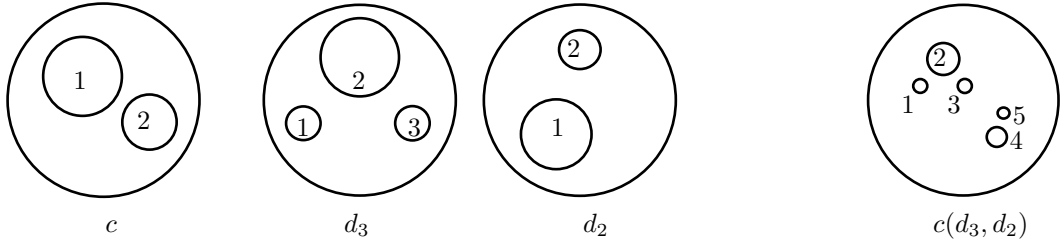   
It is a straightforward (albeit tedious) exercise to show that the fitting maps $\gamma$ satisfy the associativity and permutation equivariance conditions described in properties (1) and (3) of the definition on an operad above. Furthermore, property (2) is satisfied by taking the identity element $1\in\D(1)$ as the disk $D^{n}$ with itself as the lone little disk. Finally, we define the space $\D$ as:
\begin{equation*}
{\D}=\bigcup_{j\geq 0}\D(j),
\end{equation*}
where $\D(0)$ is the single point $*$.
Recall that we have supressed the dimension, $n$, of the underlying disk. For each $n$, the space $\D=\D_{n}$, along with the appropriate collection of fitting maps $\gamma$, is known as the {\em operad of little $n$-dimensional disks}. Before discussing our second example, the operad of grown trees, it is worth considering a variation on the little-disks operad which will be useful for us later on.

\subsection{Little disks, little lenses and little bulbs}
Recall from section \ref{lensection}, that on a round $n$-dimensional sphere of radius $\lambda$, we described a canonical way of identifying the $(\lambda,r)$-lens at the point $p\in S^{n}$, $B_{\lambda}(p,r)$, with the disk $D^{n}$. We will assume here that $r\in(0,\lambda\frac{\pi}{2}]$. Using this identification we may, for any point $p\in S^{n}$ and such a pair $\lambda, r$, reinvent the operad of little $n$-dimensional disks on $D^{n}$ as an operad of little $n$-dimensional lenses on $B_{\lambda}(p,r)$. The centre points of little disks are determined by this identification while the radii are determined by the map $\epsilon\mapsto r\epsilon$, where $\epsilon\in(0,1]$ denotes the radius of a little disk in $D^{n}$. In Fig. \ref{lensoperad}, we depict an example of this where $p$ is the north pole. 
\begin{figure}[!htbp]
\vspace{+1cm}
\hspace{-0.0cm}
\begin{picture}(0,0)
\includegraphics{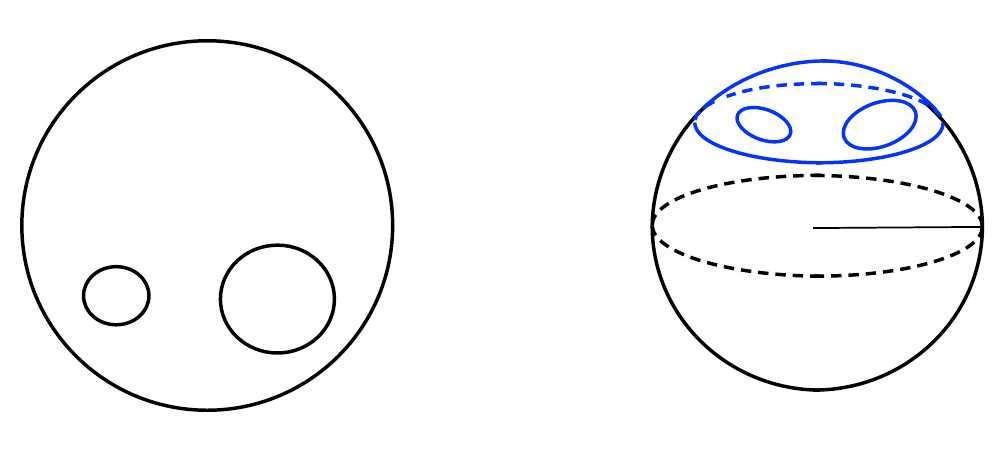}%
\end{picture}
\setlength{\unitlength}{3947sp}%
\begingroup\makeatletter\ifx\SetFigFont\undefined%
\gdef\SetFigFont#1#2#3#4#5{%
  \reset@font\fontsize{#1}{#2pt}%
  \fontfamily{#3}\fontseries{#4}\fontshape{#5}%
  \selectfont}%
\fi\endgroup%
\begin{picture}(5079,1559)(1902,-7227)
\put(5250,-5300){\makebox(0,0)[lb]{\smash{{\SetFigFont{10}{8}{\rmdefault}{\mddefault}{\updefault}{\color[rgb]{0,0,0}$B_{\lambda}(p,r)$}%
}}}}
\put(6100,-6100){\makebox(0,0)[lb]{\smash{{\SetFigFont{10}{8}{\rmdefault}{\mddefault}{\updefault}{\color[rgb]{0,0,0}$\lambda$}%
}}}}
\end{picture}%
\caption{Reinventing the little disks operad as a little lens operad}
\label{lensoperad}
\end{figure}
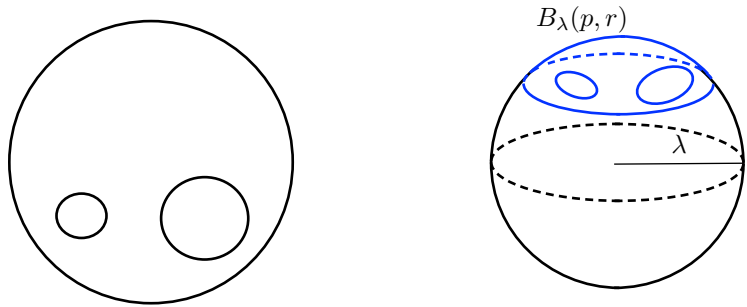   
Thus, instead of thinking of $j$-tuples of little disks in $D^{n}$ we may substitute $j$-tuples of little lenses in $B_{\lambda}(p,r)$. Furthermore, we can obtain a psc-metric representative of each such element using the work done in Lemmas \ref{bulblemma} and \ref{bulblemma2}. We will state this in the form of a lemma below.

\begin{Lemma}\label{psclensoperadlemma}
Let $n\geq 3, p\in S^{n}, \lambda>0$ and $r\in(0,\lambda\frac{\pi}{2}]$. Let $c\in \D(j)$ and let $\{B_{\lambda}(p_i, r_i)\}$ denote the collection of little lens in $B_{\lambda}(p,r)$ arising from the identification above. Then there is a psc-metric $g_c\in \Riem_{\head({\bf p})}^{+}(S^{n})$ and a continuous family of psc-metrics $g_c(t_1,\cdots, t_j)$ where each $t_i\in I$, so that the following conditions hold:
\begin{enumerate}
\item{} $g_{c}(0,\cdots, 0)=\lambda^{2}ds_{n}^{2}$ and $g_c(1,\cdots , 1)=g_c$.
\item{} Outside of $\bigcup_{i=1}^{j}\{B_{\lambda(p_i, r_i)}\}$, the metric $g_c(t_1, \cdots, t_j)$ is the standard round metric of radius $\lambda$ for all $(t_1, \cdots, t_j)\in I^{j}$.
\item{} On each ball $B_{\lambda}(p_i, r_i)$, the metric $g_c(t_1, \cdots, t_i=1, \cdots, t_j)$ is precisely the metric obtained by pushing out a bulb with head radius $1$ using the method of Lemma \ref{bulblemma2}. 
\item{} For each $i$, there are continuous parameters $\lambda_{i}(t)>0, r_i(t)>0$ and $\epsilon_{i}(t)\in(0,r_i]$, where $t\in I$, so that the restriction $g_c(t_1, \cdots, t_j)$ to the ball $B_{\lambda}(p_i, \epsilon_i)$ is the lens metric $g_{\lens}(\lambda_{i}(t_i), r_i(t_i))$. 
\item{} The parameters $r_i, \lambda_i$ and $\epsilon_i$ above satisfy $\lambda_i(0)=\lambda$, $\lambda_i(1)=1$, $r_i(0)=r_i$, $r_i(1)=\frac{\pi}{2}$ and $\epsilon_{i}(0)=r_i$.
\end{enumerate}
\end{Lemma}
\begin{proof}
This is an easy consequence of Lemmas \ref{bulblemma} and \ref{bulblemma2}.
\end{proof}
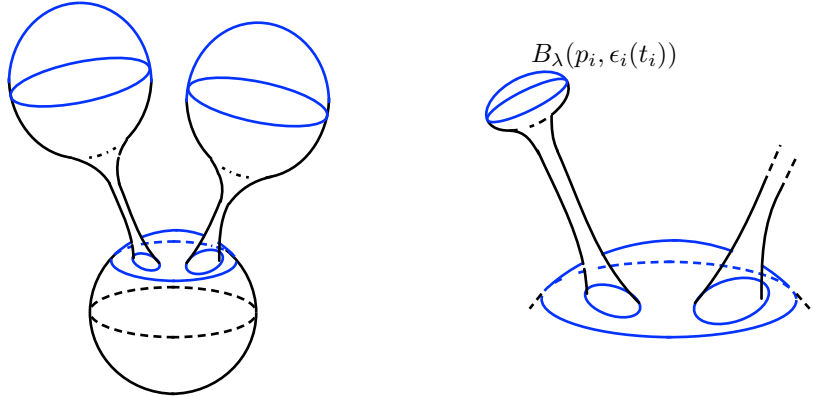
\begin{figure}[!htbp]
\vspace{+2cm}
\hspace{-0.0cm}
\begin{picture}(0,0)
\includegraphics{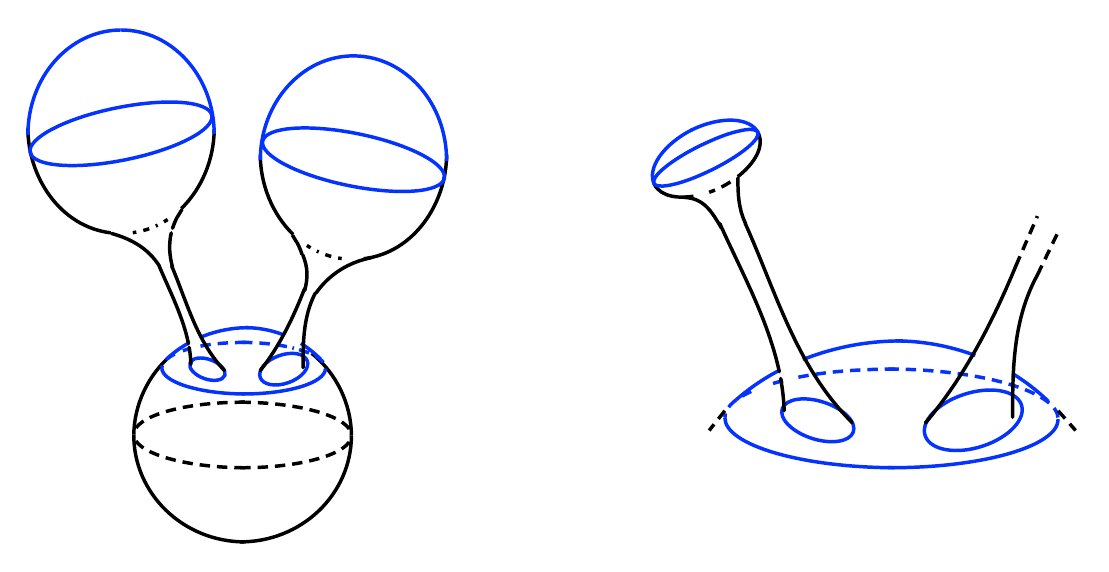}%
\end{picture}
\setlength{\unitlength}{3947sp}%
\begingroup\makeatletter\ifx\SetFigFont\undefined%
\gdef\SetFigFont#1#2#3#4#5{%
  \reset@font\fontsize{#1}{#2pt}%
  \fontfamily{#3}\fontseries{#4}\fontshape{#5}%
  \selectfont}%
\fi\endgroup%
\begin{picture}(5079,1559)(1902,-7227)
\put(5250,-5000){\makebox(0,0)[lb]{\smash{{\SetFigFont{10}{8}{\rmdefault}{\mddefault}{\updefault}{\color[rgb]{0,0,0}$B_{\lambda}(p_i,\epsilon_{i}(t_i))$}%
}}}}
\end{picture}%
\caption{The metric $g_c=g_c(1,\cdots,1)$ is depicted (left) for an element $c\in\D(2)$ along with a special focus on the $i^{th}$ bulb as it undergoes a continuous deformation via the psc-isotopy $g_c(1,\cdots t_i, \cdots,1), t_i\in I$ (right). The neighborhood $B_{\lambda}(p_i, \epsilon_i(t_i))$ on which the metric takes the form $g_{\lens}(\lambda_i(t_i), r_i(t_i))$ is highlighted.}
\label{psclensoperad}
\end{figure}   
\noindent The point of this rather technical lemma concerns operad composition. Roughly speaking, if we used the round hemisphere of radius $1$ at $p_i$ to push out bulbs corresponding to another element $c'\in \D(j)$ for some $j$, we could continuously deform the resulting psc-metric back to the one which we would have obtained by simply composing the elements $c$ and $c'$ at $p_i$. Furthermore, conditions (4) and(5) of Lemma \ref{psclensoperadlemma} mean that at each stage in the deformation we would have, as a restriction on a lens cap around $p_i$, precisely the restriction of the corresponding psc-metric for $c'$ on that lens. This will be of immense benefit later on. We now consider a second example of an operad.

\subsection{The operad of grown trees} A {\em tree} $T$ is a finite contractible planar graph with the exception that edges may have less than two adjacent vertices. Every tree must have at least one edge; the tree consisting of just one edge and no vertices is called the {\em trivial tree} and is depicted in Fig. \ref{Tree}. Edges which have two adjacent vertices are called {\em internal edges} while edges with only one adjacent vertex are called {\em external edges} or {\em legs}. The edges are oriented in the following way. Each vertex $v$ in $T$ has a set of incoming edges, denoted $\In(v)$, and exactly one outgoing edge. We allow for the case when $\In(v)=\emptyset$. In particular, this means that an internal edge is both an outgoing edge for one of its vertices (the starting vertex) and an incoming edge for the other (ending) vertex. Moreover, the set of external edges of $T$ consists of two mutually disjoint subsets: the set of {\em inputs} $\In(T)$ of all incoming edges of $T$ which have no starting vertices and the set consisting of the single outgoing edge or {\em output}, which has no end vertex. We typically depict trees with the edges directed from bottom to top and with inputs ordered from left to right. The orientation of the trivial tree is ambiguous so we simply choose one.

To aid the reader we provide an example. Consider the tree $T$ shown on the right of Fig. \ref{Tree}. This tree has $3$ internal edges and $7$ external edges. The output is the edge at the top adjacent to the vertex $v_1$. The remaining $6$ external edges are inputs. Notice also that, in the case of this tree, $|\In(v_1)|=2$, $|\In(v_2)|=4$, $|\In(v_3)|=3$, $|\In(v_4)|=0$ and $|\In(T)|=6$, where $|S|$ is the cardinality of a set $S$. 
\begin{figure}[!htbp]
\vspace{2.0cm}
\hspace{2.0cm}
\begin{picture}(0,0)
\includegraphics{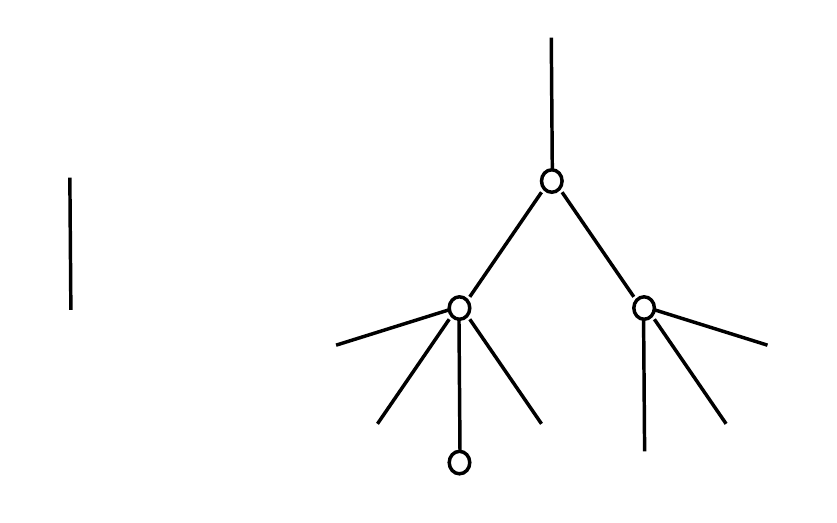}%
\end{picture}
\setlength{\unitlength}{3947sp}%
\begingroup\makeatletter\ifx\SetFigFont\undefined%
\gdef\SetFigFont#1#2#3#4#5{%
  \reset@font\fontsize{#1}{#2pt}%
  \fontfamily{#3}\fontseries{#4}\fontshape{#5}%
  \selectfont}%
\fi\endgroup%
\begin{picture}(5079,1559)(1902,-7227)
\put(4550,-5500){\makebox(0,0)[lb]{\smash{{\SetFigFont{10}{8}{\rmdefault}{\mddefault}{\updefault}{\color[rgb]{0,0,0}$v_1$}%
}}}}
\put(4100,-6200){\makebox(0,0)[lb]{\smash{{\SetFigFont{10}{8}{\rmdefault}{\mddefault}{\updefault}{\color[rgb]{0,0,0}$v_2$}%
}}}}
\put(4720,-6200){\makebox(0,0)[lb]{\smash{{\SetFigFont{10}{8}{\rmdefault}{\mddefault}{\updefault}{\color[rgb]{0,0,0}$v_3$}%
}}}}
\put(3850,-7000){\makebox(0,0)[lb]{\smash{{\SetFigFont{10}{8}{\rmdefault}{\mddefault}{\updefault}{\color[rgb]{0,0,0}$v_4$}%
}}}}
\end{picture}%
\caption{The trivial tree (left) and the tree $T$ (right)}
\label{Tree}
\end{figure}
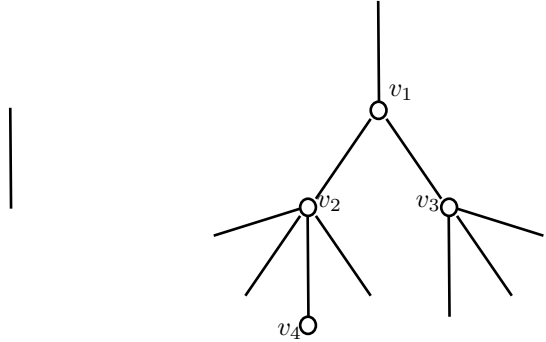 

Next, we let $X=\{X_n:n\in\mathbb{N}\}$ be a collection of topological spaces. We will now define a collection of spaces $\GTr X(j)$ as follows. Let $\GTr X(0)$ denote the single point $*$. For each $j\in\mathbb{N}$, we let $\GTr X(j)$ denote the set of all triples $(T, \alpha, \beta)$ where $T$ is a tree, $\alpha$ is a function which sends each vertex $v$ of $T$ to an element $x\in X_{|\In(v)|}$ and $\beta$ is a bijectiion from the set of inputs $\In(T)$ to $\{1,2,\cdots ,j\}$. This labelling is best thought of as a permutation of the input edges (originally ordered from left to right). Elements of $\GTr X(j)$ are best thought of as trees with $j$ inputs labelled $1,2,\cdots j$ and with vertices labelled by elements of $X$ according to the rule that the label associated to a vertex $v$ is an element of $X_{|\In(v)|}$. Finally, we specify composition maps:
\begin{equation*}
\begin{split}
\gamma:\GTr X(k)\times \GTr X(j_1)\times \GTr X(j_2)\times\cdots\times \GTr X(j_k)&\longrightarrow \GTr X(j_1+j_2+\cdots j_k),\\
(T, \psi_1, \psi_2,\cdots, \psi_k)\longmapsto \phi
\end{split}
\end{equation*}
where the element $\phi$ is obtained in the following way. Each $\psi_i$ has inputs labelled $1,2,\cdots, j_i$. For each $i=1,2,\cdots ,k$, relabel these inputs by the rule:
\begin{equation*}
(1,2,\cdots, j_i)\longmapsto (j_1+j_2+\cdots j_{i-1}+1, j_1+j_2+\cdots j_{i-1}+2, \cdots, j_1+j_2+\cdots j_{i-1}+j_{i}).
\end{equation*}
Then identify the lone outgoing edge of each newly labelled $\psi_i$ with the $i$-th input of $T$. This results in a grown tree $\phi$ with $j_1+j_2+\cdots +j_k$ labelled inputs. Note that the identity element in this case is of course the trivial tree. In Fig. \ref{GrownTree} we provide an example of this composition with the vertex labels suppressed. In this case $k=3, j_1=2, j_2=2$ and $j_3=1$.
\begin{figure}[!htbp]
\vspace{2.0cm}
\hspace{-3.0cm}
\begin{picture}(0,0)
\includegraphics{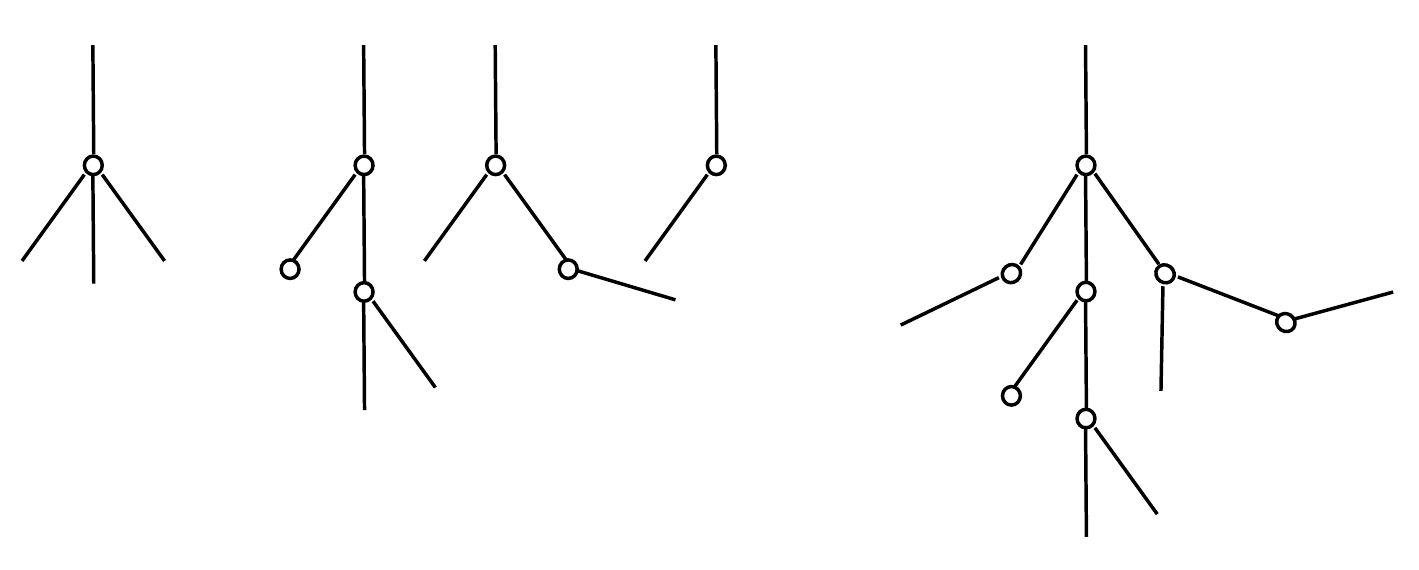}%
\end{picture}
\setlength{\unitlength}{3947sp}%
\begingroup\makeatletter\ifx\SetFigFont\undefined%
\gdef\SetFigFont#1#2#3#4#5{%
  \reset@font\fontsize{#1}{#2pt}%
  \fontfamily{#3}\fontseries{#4}\fontshape{#5}%
  \selectfont}%
\fi\endgroup%
\begin{picture}(5079,1559)(1902,-7227)
\put(2000,-5500){\makebox(0,0)[lb]{\smash{{\SetFigFont{10}{8}{\rmdefault}{\mddefault}{\updefault}{\color[rgb]{0,0,0}$3$}%
}}}}
\put(2300,-5700){\makebox(0,0)[lb]{\smash{{\SetFigFont{10}{8}{\rmdefault}{\mddefault}{\updefault}{\color[rgb]{0,0,0}$1$}%
}}}}
\put(2500,-5500){\makebox(0,0)[lb]{\smash{{\SetFigFont{10}{8}{\rmdefault}{\mddefault}{\updefault}{\color[rgb]{0,0,0}$2$}%
}}}}

\put(3600,-6300){\makebox(0,0)[lb]{\smash{{\SetFigFont{10}{8}{\rmdefault}{\mddefault}{\updefault}{\color[rgb]{0,0,0}$2$}%
}}}}
\put(3800,-6150){\makebox(0,0)[lb]{\smash{{\SetFigFont{10}{8}{\rmdefault}{\mddefault}{\updefault}{\color[rgb]{0,0,0}$1$}%
}}}}

\put(4700,-6000){\makebox(0,0)[lb]{\smash{{\SetFigFont{10}{8}{\rmdefault}{\mddefault}{\updefault}{\color[rgb]{0,0,0}$2$}%
}}}}
\put(3900,-5550){\makebox(0,0)[lb]{\smash{{\SetFigFont{10}{8}{\rmdefault}{\mddefault}{\updefault}{\color[rgb]{0,0,0}$1$}%
}}}}
\put(4970,-5550){\makebox(0,0)[lb]{\smash{{\SetFigFont{10}{8}{\rmdefault}{\mddefault}{\updefault}{\color[rgb]{0,0,0}$1$}%
}}}}

\put(6400,-6100){\makebox(0,0)[lb]{\smash{{\SetFigFont{10}{8}{\rmdefault}{\mddefault}{\updefault}{\color[rgb]{0,0,0}$5$}%
}}}}
\put(7450,-6100){\makebox(0,0)[lb]{\smash{{\SetFigFont{10}{8}{\rmdefault}{\mddefault}{\updefault}{\color[rgb]{0,0,0}$3$}%
}}}}
\put(8250,-6100){\makebox(0,0)[lb]{\smash{{\SetFigFont{10}{8}{\rmdefault}{\mddefault}{\updefault}{\color[rgb]{0,0,0}$4$}%
}}}}
\put(7070,-6900){\makebox(0,0)[lb]{\smash{{\SetFigFont{10}{8}{\rmdefault}{\mddefault}{\updefault}{\color[rgb]{0,0,0}$2$}%
}}}}
\put(7300,-6800){\makebox(0,0)[lb]{\smash{{\SetFigFont{10}{8}{\rmdefault}{\mddefault}{\updefault}{\color[rgb]{0,0,0}$1$}%
}}}}
\end{picture}%
\caption{From left to right, the elements $T, \psi_1, \psi_2, \psi_3$ and $\phi$}
\label{GrownTree}
\end{figure}
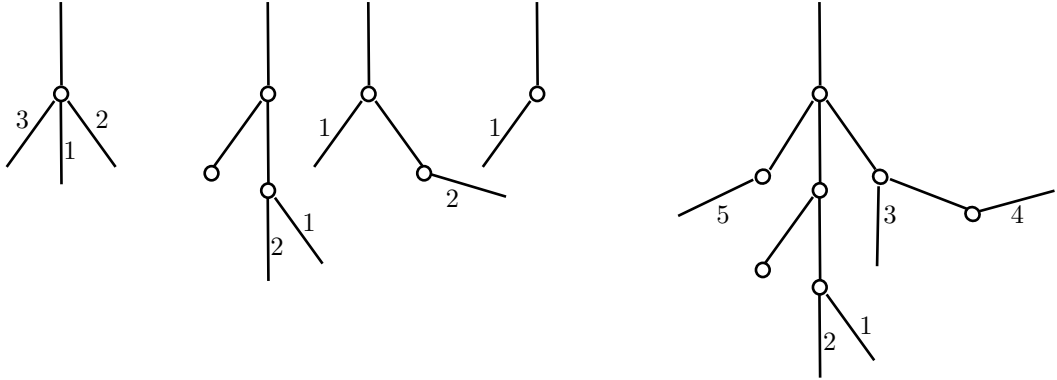 
Finally, we obtain the {\em operad of grown trees} $\GTr X$ as the union of the spaces $\GTr X(j)$ for $j\geq 0$, along with the above composition maps. As with the previous example, it is straightforward to verify that the various operad axioms are satisfied. 

\subsection{The operad of trees}\label{optree}
We now make an important modification to the previous example. Recall that for a collection of topological spaces $X=\{ X_n, n\in\mathbb{N} \}$, the space $\GTr X(j)$ consists of triples $(T, \alpha, \beta)$ where $T$ is a tree with $j$ inputs, $\alpha$ is a map which labels each vertex $v$ of $T$ with an element of $X_{|\In(v)|}$ and $\beta$ is a map which labels the inputs of $T$ with the numbers $1,2,\cdots j$. We define the set ${\Tr} X(j)$ to be the set of all quadruples $(T, \alpha, \beta, \kappa)$, where $T$, $\alpha$ and $\beta$ are as before and $\kappa$ is a function which assigns to each edge $e$ of $T$ a number $\kappa (e)$ satisfying:
\begin{enumerate}
\item{}$0\leq \kappa (e) \leq 1$.
\item{}$\kappa (e)=1$ when $e$ is an external edge (i.e. input of output) of $T$.
\end{enumerate}
The number $\kappa(e)$ is called the {\em length} of the edge $e$. Each set $\Tr X(j)$ is then given the obvious function space topology as before. The composition maps defined in the case of spaces $\GTr X(j)$ are easily generalised to work here. In the case of new inputs obtained by composing trees, the edges are assigned a length of $1$. The union of the resulting collection of spaces, along with the composition maps, gives rise to the {\em operad of trees $\Tr X$}. Note that the operad of grown trees $\GTr X$ can be identified with the suboperad of $\Tr X$ consisting of trees with only edges of length $1$. 

\subsection{The bar construction}\label{barcon}
We now consider the case that the collection of spaces $X=\{X_m:m\in\mathbb{N}\}$ introduced above is an operad in its own right, complete with composition maps $\gamma$. Of course we need to add in $X_0$ as the single point space $*$. The example to keep in mind is where $X$ is the collection of spaces $\D (m)$ of $m$-tuples of little disks with base point the disk of radius $1$ in $\D (1)$.  We first consider three relations one may impose on the operad $\Tr X$ above. Note that we will frequently abuse notation and refer to the element $(T, \alpha, \beta, \kappa)$ in $\Tr X(j)$ as the tree $T$.
\begin{enumerate}
\item[(a.)] Suppose $T$ is a tree, with an internal vertex labelled with the identity element $1$ from the space $X_{1}$. Furthermore, suppose this vertex's unique incoming edge has length $t_1$ and its outgoing edge has length $t_2$. Let $T'$ be the tree which is obtained from $T$ by replacing this vertex and its adjacent edges by a single edge of length $t_1*t_2=t_1+t_2-t_1t_2$. This leads to a relation $T \sim T '$ as described in Fig. \ref{rel1}.
\begin{figure}[!htbp]
\vspace{0.0cm}
\hspace{4.0cm}
\begin{picture}(0,0)
\includegraphics{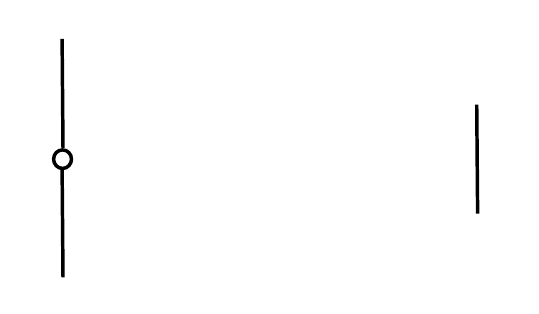}%
\end{picture}
\setlength{\unitlength}{3947sp}%
\begingroup\makeatletter\ifx\SetFigFont\undefined%
\gdef\SetFigFont#1#2#3#4#5{%
  \reset@font\fontsize{#1}{#2pt}%
  \fontfamily{#3}\fontseries{#4}\fontshape{#5}%
  \selectfont}%
\fi\endgroup%
\begin{picture}(5079,1559)(1902,-7227)
\put(1950,-6800){\makebox(0,0)[lb]{\smash{{\SetFigFont{10}{8}{\rmdefault}{\mddefault}{\updefault}{\color[rgb]{0,0,0}$t_1$}%
}}}}
\put(1950,-6200){\makebox(0,0)[lb]{\smash{{\SetFigFont{10}{8}{\rmdefault}{\mddefault}{\updefault}{\color[rgb]{0,0,0}$t_2$}%
}}}}
\put(1960,-6500){\makebox(0,0)[lb]{\smash{{\SetFigFont{10}{8}{\rmdefault}{\mddefault}{\updefault}{\color[rgb]{0,0,0}$1$}%
}}}}
\put(3100,-6500){\makebox(0,0)[lb]{\smash{{\SetFigFont{10}{8}{\rmdefault}{\mddefault}{\updefault}{\color[rgb]{0,0,0}$\sim$}%
}}}}
\put(4200,-6500){\makebox(0,0)[lb]{\smash{{\SetFigFont{10}{8}{\rmdefault}{\mddefault}{\updefault}{\color[rgb]{0,0,0}$t_1*t_2$}%
}}}}
\end{picture}%
\caption{The values $t_1, t_2$ and $t_1*t_2$ are the edge lengths}
\label{rel1}
\end{figure}
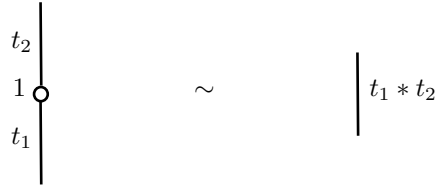 

\item[(b.)] Let $v$ be a vertex of a tree $T$ in $\Tr X(j)$ and let $T_v$ be the subtree consisting of $v$, its unique outgoing edge and all directed paths which end in $v$. Suppose $\In(v)=k$, that $v$ is labelled by the element $x\in X_{k}$ and that $x=y.\sigma$ for some element $\sigma\in \Sigma_k$ (recall that $X$ is an operad and so there is an action of the symmetric group). Then the following relation on subtrees induces a relation on trees. 
\begin{figure}[!htbp]
\vspace{1.0cm}
\hspace{4.0cm}
\begin{picture}(0,0)
\includegraphics{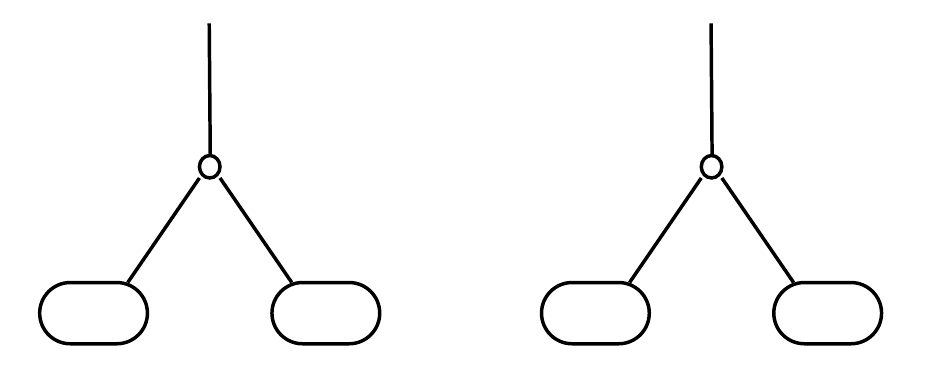}%
\end{picture}
\setlength{\unitlength}{3947sp}%
\begingroup\makeatletter\ifx\SetFigFont\undefined%
\gdef\SetFigFont#1#2#3#4#5{%
  \reset@font\fontsize{#1}{#2pt}%
  \fontfamily{#3}\fontseries{#4}\fontshape{#5}%
  \selectfont}%
\fi\endgroup%
\begin{picture}(5079,1559)(1902,-7227)
\put(2500,-6200){\makebox(0,0)[lb]{\smash{{\SetFigFont{10}{8}{\rmdefault}{\mddefault}{\updefault}{\color[rgb]{0,0,0}$x.\sigma$}%
}}}}
\put(5070,-6200){\makebox(0,0)[lb]{\smash{{\SetFigFont{10}{8}{\rmdefault}{\mddefault}{\updefault}{\color[rgb]{0,0,0}$x$}%
}}}}
\put(4000,-6200){\makebox(0,0)[lb]{\smash{{\SetFigFont{10}{8}{\rmdefault}{\mddefault}{\updefault}{\color[rgb]{0,0,0}$\sim$}%
}}}}
\put(2190,-6900){\makebox(0,0)[lb]{\smash{{\SetFigFont{10}{8}{\rmdefault}{\mddefault}{\updefault}{\color[rgb]{0,0,0}$S_1$}%
}}}}
\put(2790,-6900){\makebox(0,0)[lb]{\smash{{\SetFigFont{10}{8}{\rmdefault}{\mddefault}{\updefault}{\color[rgb]{0,0,0}$...$}%
}}}}
\put(3300,-6900){\makebox(0,0)[lb]{\smash{{\SetFigFont{10}{8}{\rmdefault}{\mddefault}{\updefault}{\color[rgb]{0,0,0}$S_k$}%
}}}}
\put(4500,-6900){\makebox(0,0)[lb]{\smash{{\SetFigFont{10}{8}{\rmdefault}{\mddefault}{\updefault}{\color[rgb]{0,0,0}$S_{\sigma^{-1}1}$}%
}}}}
\put(5200,-6900){\makebox(0,0)[lb]{\smash{{\SetFigFont{10}{8}{\rmdefault}{\mddefault}{\updefault}{\color[rgb]{0,0,0}$...$}%
}}}}
\put(5600,-6900){\makebox(0,0)[lb]{\smash{{\SetFigFont{10}{8}{\rmdefault}{\mddefault}{\updefault}{\color[rgb]{0,0,0}$S_{\sigma^{-1}k}$}%
}}}}
\end{picture}%
\caption{The $S_i$ terms denote subtrees}
\label{rel2}
\end{figure}
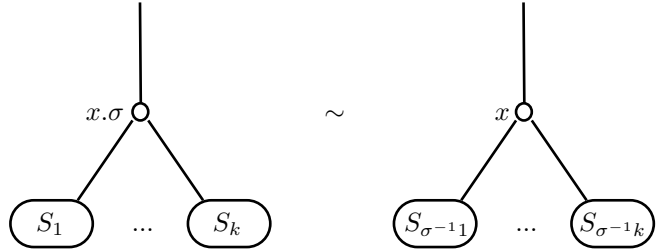 
\item[(c.)] If $T$ is a tree with an edge of length $0$ (this must be an internal edge), then we may shrink this edge away and compose the vertices by means of the operad composition. More precisely, suppose $T$ has edge $e$ with starting vertex $v_2$ and ending vertex $v_1$. Thus $e$ is one of potentially many incoming edges for $v_1$ but the unique outgoing edge of $v_2$. Let us say that $|\In(v_1)|=k$. Recall also that $v_1$ is labelled by an element $x \in X_k$ and $v_2$ is labelled by some element $y\in X_l$. The value of $l$ is unimportant. Recall that the inputs of $v_1$ are ordered from left to right and so we assume that $e=e_i$ is the $i$-th edge according to this ordering where $i\in\{1,2,\cdots,k\}$. We replace the edge $e_i$ and its adjacent vertices $v_1$ and $v_2$ with a single vertex labelled by the operad element $x\circ_i y$ which is defined as follows:
\begin{equation*}
x\circ_i y=\gamma(x;1,\cdots,1,y,1,\cdots 1),
\end{equation*}
where the element $y$ appears in the $i$-th position. Note that the resulting element $x\circ_i y$ lies in the space $X_{k+l-1}$. The edge ordering is repaired in the obvious way. The input edges $e_1, \cdots e_{i-1}$ of the removed vertex $v_1$ are unaffected. The incoming edges $e_1',\cdots,e_l'$ of the removed vertex $v_2$ become the incoming edges $e_{i}, \cdots e_{i+l-1}$ for the new vertex. Finally the edges $e_{i+1}, \cdots ,e_{k}$ become the incoming edges $e_{i+l}, \cdots, e_{k+l-1}$ for the new vertex. To aid the reader we provide an example in Fig. \ref{rel3}. Note that in this example $0$ represents the edge length, while the numbers $1,2,3$ and $4$ show the edge ordering.

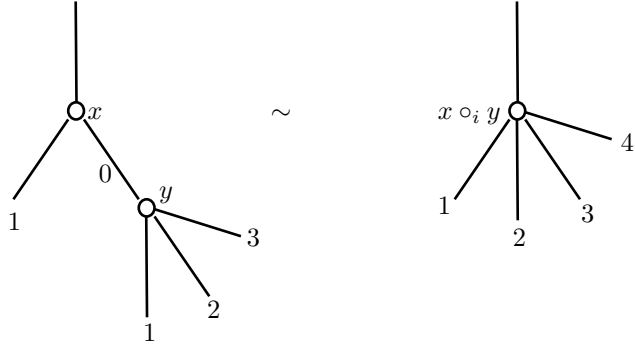
\begin{figure}[!htbp]
\vspace{2.0cm}
\hspace{4.0cm}
\begin{picture}(0,0)
\includegraphics{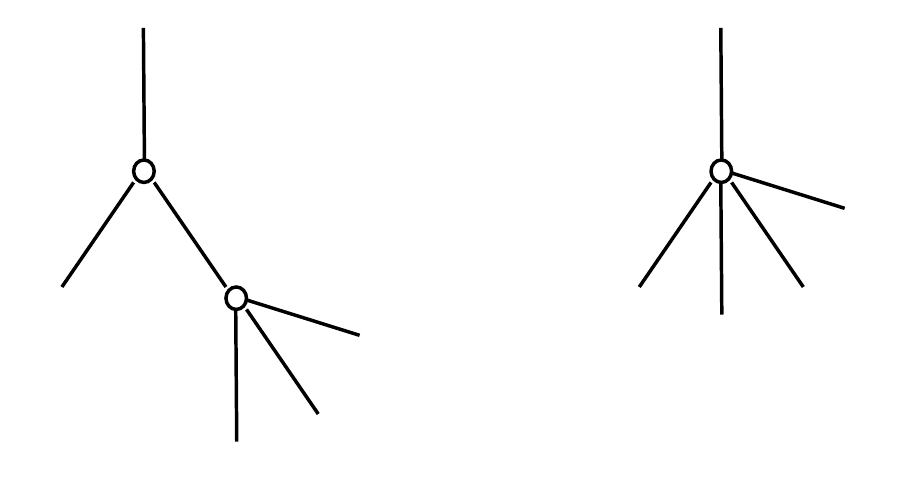}%
\end{picture}
\setlength{\unitlength}{3947sp}%
\begingroup\makeatletter\ifx\SetFigFont\undefined%
\gdef\SetFigFont#1#2#3#4#5{%
  \reset@font\fontsize{#1}{#2pt}%
  \fontfamily{#3}\fontseries{#4}\fontshape{#5}%
  \selectfont}%
\fi\endgroup%
\begin{picture}(5079,1559)(1902,-7227)
\put(2670,-6200){\makebox(0,0)[lb]{\smash{{\SetFigFont{10}{8}{\rmdefault}{\mddefault}{\updefault}{\color[rgb]{0,0,0}$0$}%
}}}}
\put(2600,-5800){\makebox(0,0)[lb]{\smash{{\SetFigFont{10}{8}{\rmdefault}{\mddefault}{\updefault}{\color[rgb]{0,0,0}$x$}%
}}}}
\put(3050,-6300){\makebox(0,0)[lb]{\smash{{\SetFigFont{10}{8}{\rmdefault}{\mddefault}{\updefault}{\color[rgb]{0,0,0}$y$}%
}}}}
\put(2100,-6500){\makebox(0,0)[lb]{\smash{{\SetFigFont{10}{8}{\rmdefault}{\mddefault}{\updefault}{\color[rgb]{0,0,0}$1$}%
}}}}
\put(3750,-5800){\makebox(0,0)[lb]{\smash{{\SetFigFont{10}{8}{\rmdefault}{\mddefault}{\updefault}{\color[rgb]{0,0,0}$\sim$}%
}}}}
\put(4800,-5800){\makebox(0,0)[lb]{\smash{{\SetFigFont{10}{8}{\rmdefault}{\mddefault}{\updefault}{\color[rgb]{0,0,0}$x\circ_{i}y$}%
}}}}
\put(2950,-7200){\makebox(0,0)[lb]{\smash{{\SetFigFont{10}{8}{\rmdefault}{\mddefault}{\updefault}{\color[rgb]{0,0,0}$1$}%
}}}}
\put(3600,-6600){\makebox(0,0)[lb]{\smash{{\SetFigFont{10}{8}{\rmdefault}{\mddefault}{\updefault}{\color[rgb]{0,0,0}$3$}%
}}}}
\put(3350,-7050){\makebox(0,0)[lb]{\smash{{\SetFigFont{10}{8}{\rmdefault}{\mddefault}{\updefault}{\color[rgb]{0,0,0}$2$}%
}}}}
\put(4800,-6400){\makebox(0,0)[lb]{\smash{{\SetFigFont{10}{8}{\rmdefault}{\mddefault}{\updefault}{\color[rgb]{0,0,0}$1$}%
}}}}
\put(5270,-6600){\makebox(0,0)[lb]{\smash{{\SetFigFont{10}{8}{\rmdefault}{\mddefault}{\updefault}{\color[rgb]{0,0,0}$2$}%
}}}}
\put(5700,-6450){\makebox(0,0)[lb]{\smash{{\SetFigFont{10}{8}{\rmdefault}{\mddefault}{\updefault}{\color[rgb]{0,0,0}$3$}%
}}}}
\put(5950,-6000){\makebox(0,0)[lb]{\smash{{\SetFigFont{10}{8}{\rmdefault}{\mddefault}{\updefault}{\color[rgb]{0,0,0}$4$}%
}}}}
\end{picture}%
\caption{Shrinking an edge of length zero away by composing its vertices}
\label{rel3}
\end{figure} 
\end{enumerate}
We now state a theorem which follows directly from Theorem 2.20 of \cite{Stash}.
\begin{Theorem}
For each operad $\P$, there is an operad $W\P$ defined as follows:
\begin{equation*}
W\P=\Tr \P / \text{relations \{(a.), (b.), (c.)\}}.
\end{equation*}
\end{Theorem} 
\noindent The process of constructing $W\P$ from $\P$ is known as the {\em bar construction} or {\em {$W$}-construction} for operads and is due to Boardman and Vogt.
We will now state some important results concerning the relationship between $\P$ and $W\P$.

\subsection{Operad actions}\label{operadaction}
Let $Z$ is a topological space and let $\P$ be an operad. We describe $Z$ as a {\em $\P$-space} if for each integer $k\geq 0$ there are actions
\begin{equation*}
\begin{split}
\theta:\P(k)\times Z^{k}\longrightarrow Z,
\end{split}
\end{equation*}
so that the following conditions hold:
\begin{enumerate}
\item{}$\theta(c.{\sigma}, (z_1, z_2, \cdots z_k))=\theta(c, (z_{\sigma^{-1}(1)},\cdots, z_{\sigma^{-1}(k)}))$ for all $\sigma\in\Sigma_k$, $c\in \P(k)$ and $(z_1,\cdots ,z_k)\in Z^{k}$.
\item{} The following diagram commutes:

\begin{tikzpicture}[description/.style={fill=white,inner sep=2pt}] 
\matrix (m) [matrix of math nodes, row sep=3em, 
column sep=2.5em, text height=1.5ex, text depth=0.25ex] 
{  \P(k)\times \P(j_1)\times \P(j_k)\times Z^{j}& & \P(j)\times Z^{j} & & Z \\ 
\P({k})\times \P({j_1})\times Z^{j_1}\times \cdots \P({j_k})\times Z^{j_k} &  & \P({k})\times Z^{k}  & & Z \\ }; 
\path[->,font=\scriptsize] 
(m-1-3) edge node[auto] {$ \theta$} (m-1-5) 
(m-1-5) edge node[auto] {$Id $} (m-2-5) 
(m-2-1) edge node[auto] {$ Id \times \theta^{k} $} (m-2-3) 
(m-1-1) edge node[auto] {$ shuffle $} (m-2-1)
(m-1-1) edge node[auto]  {$\gamma\times Id$}(m-1-3) 
(m-2-3) edge node[auto]  {$\theta$}(m-2-5); 
\end{tikzpicture} 
\end{enumerate}
We now state a theoem due to Boardman and Vogt which concerns the bar construction from the previous section. 
\begin{Theorem}\label{BVthm} {\rm (Theorem 4.37 \cite{BV})}
A topological space $Z$ is a $\P$-space, for some operad $\P$, if and only it is a $W\P$-space.
\end{Theorem}

A little later we will show that, for $n\geq 3$, the space of metrics of positive scalar curvature on the sphere $S^{n}$ is homotopy equivalent to a $W\D_{n}$-space where $\D_{n}$ is the operad of little $n$-dimensional disks. The above theorem allows us to conclude that this space is also homotopy equivalent to a $\D_{n}$-space. In section \ref{Dspace} below, we will demonstrate that, for $n\geq3$, the space $\Riem^{+}(S^{n})$ is homotopy equivalent to a $W\D_n$ and consequently a $\D_n$ space. The importance of this stems from the following theorem due to Boardman, Vogt and May. This is a case of Theorem 13.1 from \cite{May}.
\begin{Theorem} \label{BVM}{\rm (Boardman, Vogt and May) \cite{May}}
If a ${\D}_n$-space $Z$ is group-like (i.e. $\pi_0(Z)$ is a group under the induced multplication), then it is weakly homotopy equivalent to an $n$-fold loop space.
\end{Theorem}
\noindent Thus, to show that $\Riem^{+}(S^{n})$ is an $n$-fold loop space when $n\geq 3$, it is enough to show that $\pi_0(\Riem^{+}(S^{n}))$ is a group under the appropriate multiplication. We will return to this problem in the final section, section \ref{grouplike}, where we will demonstrate that $\Riem^{+}(S^{n})$ is indeed group-like in the case when $n=3$ or $n\geq 5$. The case of $n=4$ is an open problem.

\section{Showing that $\Riem^{+}(S^{n})$ is homotopy equivalent to a ${\D}_{n}$-space}\label{Dspace}
We now return to the sphere $S^{n}$ which, as we discussed earlier, is modelled on the standard unit sphere in $\mathbb{R}^{n+1}$. Once again, we assume that $n\geq 3$. We denote by $p_0$, the north pole $(0,0,\cdots,0,1)\in S^{n}\subset\mathbb{R}^{n+1}$. Recall that immediately preceding \ref{fixheadvary}, we defined the space $\Riem^{+}_{{\rm{head}}(p_0, 1)}(S^{n})$ consisting of a psc-metrics which take the form of a bulb-head with head radius $1$ (but arbitrary head angle $r\in(0,\frac{\pi}{2}]$) on some neighborhood $D_{p_0}$. For our purposes, we choose $D_{p_0}$ to a be a geodesic ball $B_{1}(p_0,\frac{\pi}{2}+\epsilon)$ for some small $\epsilon\in(0,\frac{\pi}{2})$. The value of $\epsilon$ is not important. For each metric $g\in \Riem^{+}_{{\rm{head}}(p_0, 1)}(S^{n})$, the restriction of $g$ to the closed northern hemisphere ${D_{+}}$ is now precisely the round hemisphere of radius $1$. To simplify the notation, henceforth we write:
\begin{equation*}
\Riem^{+}_{{\rm{D_{+}}}(1)}(S^{n})=\Riem^{+}_{{\rm{head}}(p_0, 1)}(S^{n}).
\end{equation*}
In Lemma \ref{bulbdefret}, we showed that, when $n=3$, this space is homotopy equivalent to the space of all psc-metrics on $S^{n}$, $\Riem^{+}(S^{n})$. Thus, in order to demonstrate that $\Riem^{+}(S^{n})$ has the homotopy type of a $W{\D}_{n}$-space (and consequently a ${\D}_{n}$-space), it will be sufficient to show this for the space $\Riem^{+}_{{\rm{D_{+}}}(1)}(S^{n})$.

\subsection{The action of $W\D_n$ on $\Riem^{+}_{{\rm{D_{+}}}(1)}(S^{n})$}\label{action}
We begin by defining a map from ${W\D_n}$ to $\Riem^{+}_{{\rm{D_{+}}}(1)}(S^{n})$. Essentially, psc-metrics in the image of this map will be analogues of the elements of $W\D_n$, which we will use to define the action. In order to define this map, we begin by specifying some rules for associating elements of $\Riem^{+}_{{\rm{D_{+}}}(1)}(S^{n})$ to certain building blocks of $W\D_n$. To ease notation we will once more suppres the $n$, writing $\D_n(j)$ as simply $\D(j)$. 

\begin{enumerate}
\item{}{\bf The trivial tree.} We assign the trivial tree to the standard round metric $ds_{n}^{2}$ in $\Riem^{+}_{{\rm{D_{+}}}(1)}(S^{n})$. 

\item{} {\bf A tree with a single vertex.} 
Suppose we have a tree consisting of a single vertex, labelled by the element $c\in\D(j)$. All edges must have length $1$. We associate to this tree precisely the element $g_c=g_c(1)$ obtained by Lemma \ref{psclensoperadlemma} with respect to the southern hemisphere $D_{-}$.

\item{}{\bf A tree with all edges of length $1$.} We start by associating the root vertex to a psc-metric exactly as in the previous case. This results in a psc-metric with bulb-heads of radius $1$ for each of this vertices input edges. On each of these heads we repeat the previous step. We continue on in this way for all other vertices; see Fig. \ref{bulbtree} for an example.

\begin{figure}[!htbp]
\vspace{3.0cm}
\hspace{2.0cm}
\begin{picture}(0,0)
\includegraphics{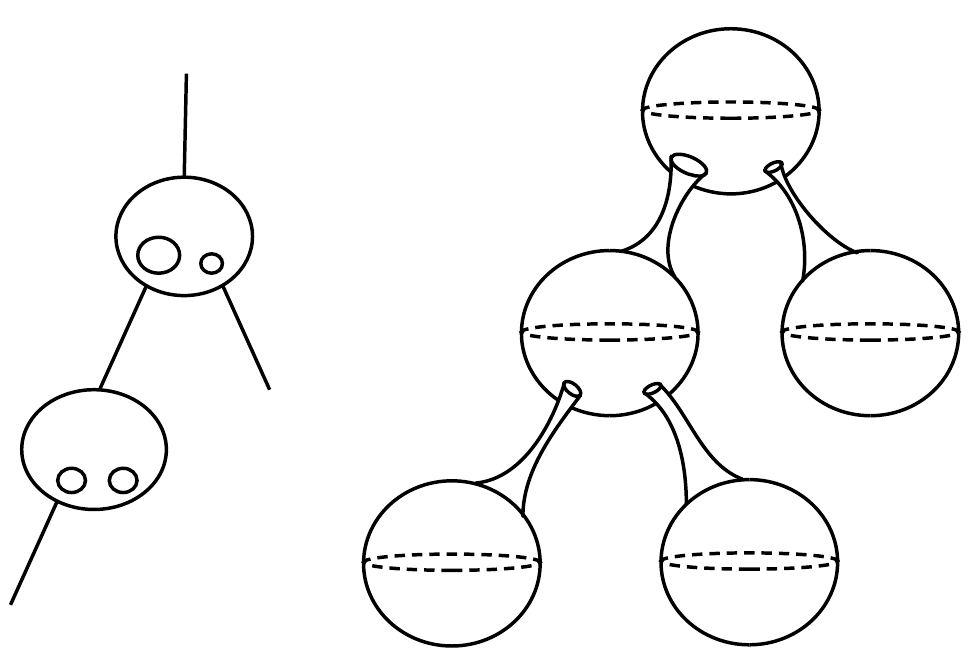}%
\end{picture}
\setlength{\unitlength}{3947sp}%
\begingroup\makeatletter\ifx\SetFigFont\undefined%
\gdef\SetFigFont#1#2#3#4#5{%
  \reset@font\fontsize{#1}{#2pt}%
  \fontfamily{#3}\fontseries{#4}\fontshape{#5}%
  \selectfont}%
\fi\endgroup%
\begin{picture}(5079,1559)(1902,-7227)
\end{picture}%
\caption{An element of $W\D$ with two vertices and all edges of length $1$ (left) and the corresponding psc-metric (right)}
\label{bulbtree}
\end{figure} 

\item{}{\bf General trees.} We must now consider what happens to the psc-metric above if one alters the internal edge lengths. As it stands each edge has length $1$ and corresponds to a bulb which has been pushed out by Lemma \ref{psclensoperadlemma}. Consider for a moment the $i^{th}$-input edge (currently of length $1$) of a vertex $v$ with label $c\in \D(j)$. Recall that  the corresponding $i^{th}$ bulb was attained at the $t_i=1$ stage of a psc-isotopy $g_{c}(1,\cdots, t_i, \cdots ,1), t_i\in I$. If we now replace the edge length of $1$ with some other length $t_i$, we need to perform a corresponding replacement of the metric 
$g_{c}(1,\cdots, t_i=1, \cdots ,1)$.  One might assume that the metric $g_{c}(1,\cdots, t_i, \cdots ,1)$ is the obvious replacement and, usually, this is precisely what we do. Unfortunately, to properly satisfy relation (a) of the bar construction in \ref{barcon}, there is a case where we must make a tiny adjustment to this association. To deal with this problem we specify a weighting function $\omega:I\rightarrow I$; see below. Then, instead of replacing $g_{c}(1,\cdots, t_i=1, \cdots ,1)$ with $g_{c}(1,\cdots, t_i, \cdots ,1)$, we replace it with $g_{c}(1,\cdots, \omega(t_i), \cdots ,1)$. Of course, this replacement may have the effect of reducing the hemisphere of radius $1$ on which operad composition takes place to some general $(\lambda,r)$-lens. This is not a problem, given that we have a canonical way of reproducing operad elements on this lens and modifying the construction accordingly, via Lemma \ref{psclensoperadlemma}. The weighting function $\omega$ satisfies the following properties.

\noindent {\bf The weighting function $\omega$ in the regular case.} For edges of length $t\in I$, whose non-empty adjacent vertices are labelled by elements of the little disk operad whose little disks all have radius $\leq\frac{1}{2}$, we set $\omega(t)=t$. Note that an edge with an adjacent vertex is an external edge.

\noindent {\bf The weighting function $\omega$ in the special case.} We consider paths of the following type on a tree $T\in W\D$. All vertices in the path are labelled by an element of the little disks operad with a little disk of radius $\geq \frac{3}{4}$, with the exception of the end vertices. Moreover the end vertices may be empty. By including the possibility of an empty vertex, we allow for paths which include external edges. Suppose the edges of this path have lengths $s_1, s_2, \cdots, s_k$, in order of the path direction. Such a situation is illustrated below in Fig. \ref{pathrelation}, where we draw the path from left to right.
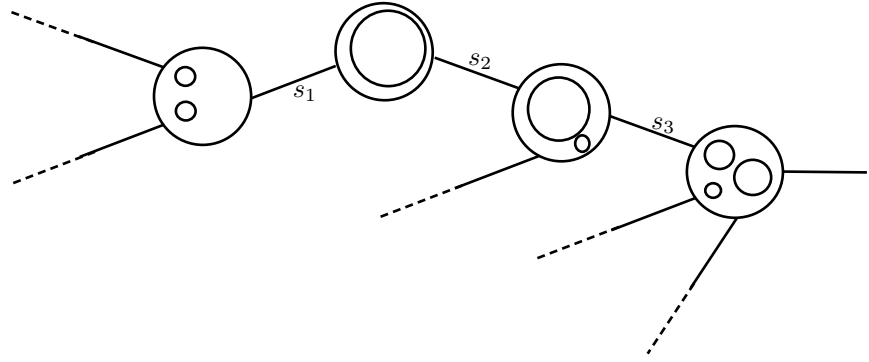
\begin{figure}[!htbp]
\vspace{2.0cm}
\hspace{1.0cm}
\begin{picture}(0,0)
\includegraphics{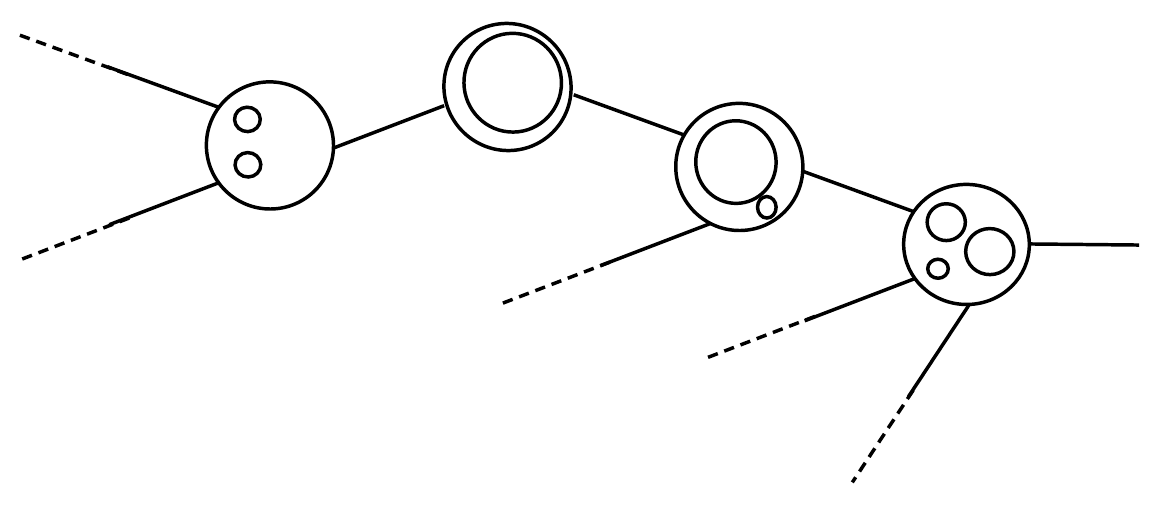}%
\end{picture}
\setlength{\unitlength}{3947sp}%
\begingroup\makeatletter\ifx\SetFigFont\undefined%
\gdef\SetFigFont#1#2#3#4#5{%
  \reset@font\fontsize{#1}{#2pt}%
  \fontfamily{#3}\fontseries{#4}\fontshape{#5}%
  \selectfont}%
\fi\endgroup%
\begin{picture}(5079,1559)(1902,-7227)

\put(3700,-5500){\makebox(0,0)[lb]{\smash{{\SetFigFont{10}{8}{\rmdefault}{\mddefault}{\updefault}{\color[rgb]{0,0,0}$s_1$}%
}}}}
\put(4800,-5300){\makebox(0,0)[lb]{\smash{{\SetFigFont{10}{8}{\rmdefault}{\mddefault}{\updefault}{\color[rgb]{0,0,0}$s_2$}%
}}}}
\put(5950,-5700){\makebox(0,0)[lb]{\smash{{\SetFigFont{10}{8}{\rmdefault}{\mddefault}{\updefault}{\color[rgb]{0,0,0}$s_3$}%
}}}}
\end{picture}%
\caption{A path in $T$ whose internal vertices are of the type described in the special case}
\label{pathrelation}
\end{figure} 
We now define $\omega$ on the edge lengths along this path by the following recursive formula:
\begin{equation}\label{omega}
\omega(s_1)=s_1, \hspace{1cm} \omega(s_i+1)=s_{i+1}+\omega(s_i)-s_{i+1}\omega(s_i).
\end{equation}
\end{enumerate}
We denote by $P$, the map
\begin{equation}\label{proxy}
P:W\D_{n}\longrightarrow \Riem^{+}_{{\rm{D_{+}}}(1)}(S^{n}),
\end{equation}
which sends an element $T\in W\D_{n}$ to the psc-metric $g_{T}$ defined by the construction above.
\begin{Lemma}
For $n\geq 3$, the above process gives rise to a well-defined map $P:W\D_{n}\rightarrow \Riem^{+}_{{\rm{D_{+}}}(1)}(S^{n})$ 
\end{Lemma}
\begin{proof} This involves checking that the relations (a.), (b.) and (c.) of the bar construction in section \ref{barcon} are satisfied. Relation (c.) is satisfied as a result of the construction in Lemma \ref{psclensoperadlemma} which guarantees that shrinking an edge length to zero corresponds to rewinding the psc-isotopy which pushed out a bulb of head radius $1$ back to the lens from which it grew. This is precisely the composition we require. It should be obvious that nothing in this construction interferes combinatorially with the tree $T$ and so relation (b.) is easily satisfied.  Finally, relation (a.) is satisfied as a result of the weight function $\omega$ on the edge lengths of $T$. 
\end{proof}

We are now in a position to define the action of $W\D_n$ on $\Riem^{+}_{{\rm{D_{+}}}(1)}(S^{n})$. Recall that an element $T\in W\D_{k}$ is a psc-metric on $S^{n}$ with $k$ bulb-heads of radius $1$ ordered $1, \cdots, k$. We now define an action $\theta$ as follows:
\begin{equation}\label{WDaction}
\begin{split}
\theta_{\mathrm{psc}}:\D(k)\times \Riem^{+}_{{\rm{D_{+}}}(1)}(S^{n})^{k} &\longrightarrow \Riem^{+}_{{\rm{D_{+}}}(1)}(S^{n})\\
(T;(g_1, g_2, \cdots g_k))&\longmapsto J_{k0}(J_{(k-1)(0)}\cdots J_{20}( J_{10}(P(T), g_1), g_2)\cdots g_{k-1}), g_k),
\end{split}
\end{equation}
where $J_{ij}=J_{ij}^{\head(1,\frac{\pi}{2})}$, the map defined in \ref{bulbjoin}. Simply put, we cut off the round radius $1$ hemispheres from the $k$ bulb heads on $P(T)$ and on the north pole for each $g_i$, where $i\in\{1,\cdots, k\}$. We then glue in the obvious way according to label. To aid the reader, we depict an example in Fig. \ref{opaction} below.
\begin{figure}[!htbp]
\vspace{2.0cm}
\hspace{-5.0cm}
\begin{picture}(0,0)
\includegraphics{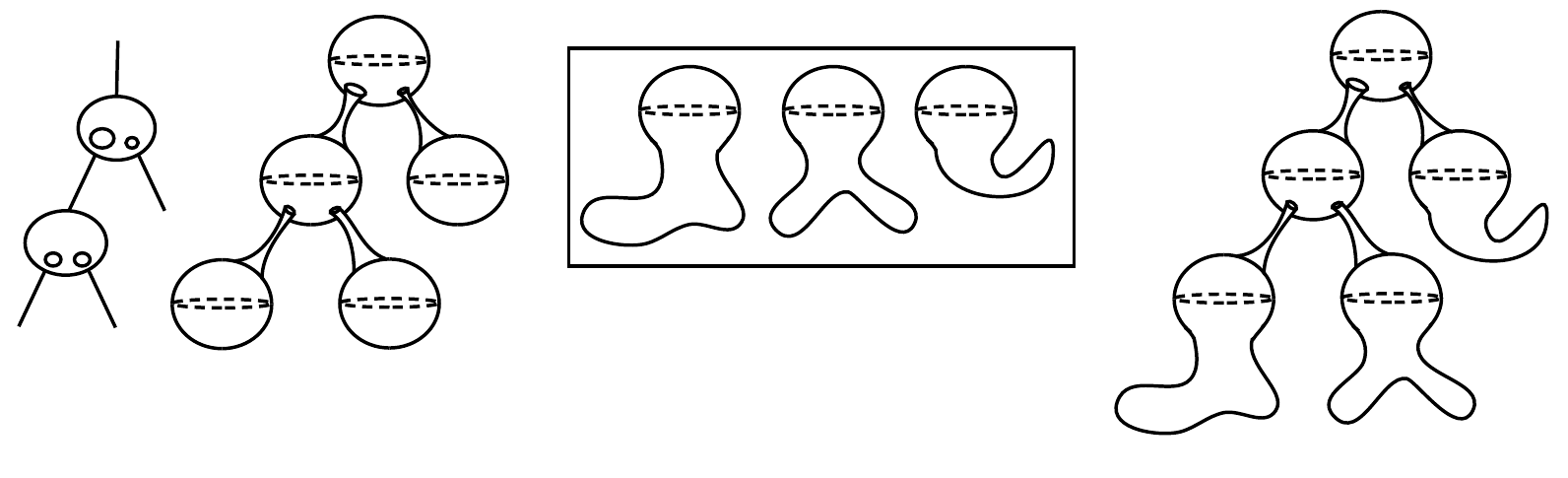}%
\end{picture}
\setlength{\unitlength}{3947sp}%
\begingroup\makeatletter\ifx\SetFigFont\undefined%
\gdef\SetFigFont#1#2#3#4#5{%
  \reset@font\fontsize{#1}{#2pt}%
  \fontfamily{#3}\fontseries{#4}\fontshape{#5}%
  \selectfont}%
\fi\endgroup%
\begin{picture}(5079,1559)(1902,-7227)
\put(2100,-5200){\makebox(0,0)[lb]{\smash{{\SetFigFont{10}{8}{\rmdefault}{\mddefault}{\updefault}{\color[rgb]{0,0,0}$T$}%
}}}}
\put(3000,-5200){\makebox(0,0)[lb]{\smash{{\SetFigFont{10}{8}{\rmdefault}{\mddefault}{\updefault}{\color[rgb]{0,0,0}$P(T)$}%
}}}}

\put(1800,-6700){\makebox(0,0)[lb]{\smash{{\SetFigFont{10}{8}{\rmdefault}{\mddefault}{\updefault}{\color[rgb]{0,0,0}$1$}%
}}}}
\put(2400,-6700){\makebox(0,0)[lb]{\smash{{\SetFigFont{10}{8}{\rmdefault}{\mddefault}{\updefault}{\color[rgb]{0,0,0}$2$}%
}}}}
\put(2630,-6100){\makebox(0,0)[lb]{\smash{{\SetFigFont{10}{8}{\rmdefault}{\mddefault}{\updefault}{\color[rgb]{0,0,0}$3$}%
}}}}
\put(4750,-5500){\makebox(0,0)[lb]{\smash{{\SetFigFont{10}{8}{\rmdefault}{\mddefault}{\updefault}{\color[rgb]{0,0,0}$g_1$}%
}}}}
\put(5530,-5600){\makebox(0,0)[lb]{\smash{{\SetFigFont{10}{8}{\rmdefault}{\mddefault}{\updefault}{\color[rgb]{0,0,0}$g_2$}%
}}}}
\put(6200,-5600){\makebox(0,0)[lb]{\smash{{\SetFigFont{10}{8}{\rmdefault}{\mddefault}{\updefault}{\color[rgb]{0,0,0}$g_3$}%
}}}}
\put(7200,-5000){\makebox(0,0)[lb]{\smash{{\SetFigFont{10}{8}{\rmdefault}{\mddefault}{\updefault}{\color[rgb]{0,0,0}$\theta_{\mathrm{psc}}(T;g_1,g_2,g_3)$}%
}}}}

\end{picture}%
\caption{An example of the operad action when $k=3$}
\label{opaction}
\end{figure}
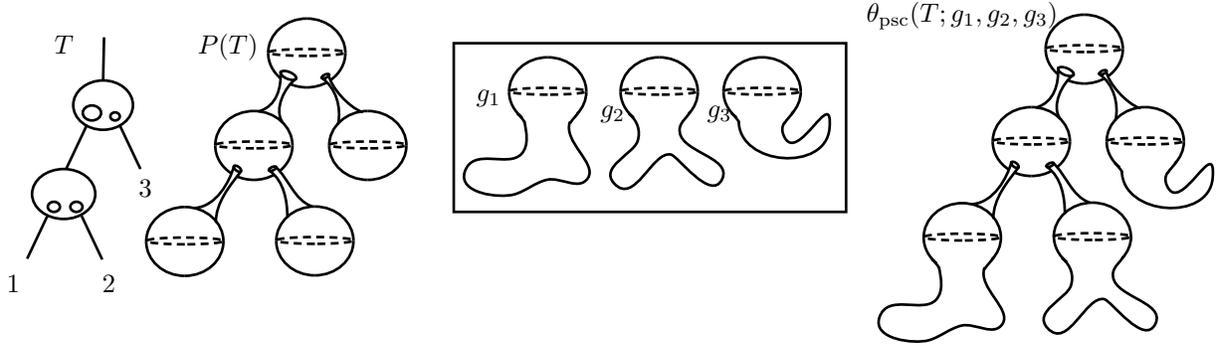
\noindent We now state a lemma concerning this action.
\begin{Lemma}\label{actionlemma}
When $n\geq 3$, the action $\theta$ defined in \ref{WDaction} gives $\Riem^{+}_{{\rm{D_{+}}}(1)}(S^{n})$ the structure of a $W\D_{n}$-space.
\end{Lemma}
\begin{proof}
We need to verify that the map $\theta_{\mathrm{psc}}$ satsifies conditions (1.) and (2.) in the definition of an operad action at the beginning of section \ref{operadaction}. Showing that condition (1.) is satisfied is an easy combinatorial exercise. The second condition, which concerns composition, is a little more subtle. The main pitfall is as follows. Suppose we compose trees $T_1$ and $T_2$ to obtain $T_3$. We need to be sure that metric $P(T_3)$ is precisely the metric obtained by the corresponding attachment of the metrics $P(T_1)$ and $P(T_2)$. In the case when all edges are of length $1$, this is obvious by the construction. However, in the case of more general trees we have to consider the effects of the weighting function $\omega$ on the lengths of edges. Recall that for certain edges, ones which are part of a special case described above, we have to ensure that the function $\omega$ respects tree composition. The saving grace here is the way in which we compose trees. Recall from section \ref{optree} that this composition involves the identification of the outgoing external edge of one tree with an incoming external edge of another. The new edge length is always $1$. Now suppose that the edge length directly above this newly formed edge has edge length $t$ and that the edge below has length $s$. We therefore have a sequence of $3$ edges with lengths, listed in order from bottom to top: $s,1,t$. When we apply $\omega$ we obtain the following new values:
\begin{equation*}
\begin{split}
t&\longmapsto \omega(t)\\
1&\longmapsto \omega(1)=1+\omega(s)-1.\omega(s)=1\\
s&\longmapsto \omega(s)
\end{split}
\end{equation*}
Thus, the edges of length $1$ act as ``resets" in our recursive formula for $\omega$. In particular, this means that the weight information above the point of composition is unaffected by the newly added subtree.
\end{proof}
\begin{Corollary}\label{actioncor}
When $n\geq3$, the space $\Riem^{+}(S^{n})$ has the homotopy type of a $\D_{n}$-space
\end{Corollary}
\begin{proof}
This is an immediate consequence of Lemma \ref{actionlemma} above, Theorem \ref{BVthm} and Lemma \ref{bulbdefret}.
\end{proof}
\noindent We have overcome one significant obstacle to proving our main result. In the next section, we will deal with the other.

\section{Showing that $\pi_0$ is a group}\label{grouplike}
In this section we wish to show that $\pi_{0}(\Riem^{+}(S^{n}))$ is a group under the operation induced by the Gromov-Lawson connected sum construction. The set $\pi_0(\Riem^{+}(S^{n}))$ is of course the set of path components of $\Riem^{+}(S^{n})$.  Earlier in the paper we noted that two metrics in $\Riem^{+}(S^{n})$ are said to be {\em psc-isotopic} if they lie in the same path component. The notion of {\em psc-isotopy} is therefore an equivalence relation on the space $\Riem^{+}(S^{n})$. A related notion, which we will make use of shortly is the notion of {\em psc-concordance}. In the case of metrics $g_0, g_1\in \Riem^{+}(S^{n})$, we say that $g_0$ and $g_1$ are {\em psc-concordant} if there is a psc-metric $\bar{g}$ on $S^{n}\times I$ which near $S^{n}\times \{0\}$ takes the form of a product $g_0+dt^{2}$ and near $S^{n}\times \{1\}$ takes the form $g_1+dt^{2}$. Again, psc-concordance is an equivalence relation on the set of psc-metrics $\Riem^{+}(S^{n})$.
It is a well known fact that metrics which are psc-isotopic are psc-concordant; see Lemma 1.3 from \cite{Walsh1} for example. Recent work by Botvinnik in \cite{Botvinnik} has shown that, under reasonable hypotheses, the converse is true. In particular, the following theorem is a case of Theorem B. from \cite{Botvinnik}.

\vspace{0.2cm}
\noindent {\bf Theorem.} (Botvinnik) \cite{Botvinnik}
{\em Let $g_0, g_1\in \Riem^{+}(S^{n})$. When $n\geq 5$, $g_0$ is psc-isotopic to $g_1$ if and only if $g_1$ is psc-concordant to $g_1$.}
\vspace{0.2cm} 

The above theorem will play an important role in the proof of our main result. It is worth noting that the hypothesis that $n$ be at least five cannot be removed as the above result fails to be true when $n=4$; see \cite{Ru}. In the case when $n=3$, it is demonstrated by Marques in \cite{Marques} that the space $\Riem^{+}(S^{n})$ is path connected, and so Botvinnik's theorem holds here for trivial reasons.
We now return to the problem of equipping $\pi_0(\Riem^{+}(S^{n}))$ with a group structure. 
\begin{Lemma}\label{pi0group}
For $n=3$ or $n\geq 5$, the set $\pi_0(\Riem^{+}(S^{n}))$ is a group under the operation induced by connected sum of metrics.
\end{Lemma}  
\begin{proof}
Corollary 1.1 of \cite{Marques} states that the space $\Riem^{+}(S^{n})$ is path-connected when $n=3$. We therefore concentrate on the case when $n\geq 5$. The group operation is of course induced by the Gromov-Lawson connected sum construction on psc-metrics. In the case of two psc-merics $g_0$ and $g_1$ on $S^{n}$, we may form a new psc-metric  $g_0\# g_1$ by removing disks from $(S^{n}, g_0)$ and $(S^{n}, g_1)$ and connecting the resulting disks with a cylinder $S^{n-1}\times I$ equipped with an appropriate connecting psc-metric, a la Gromov and Lawson in \cite{GL}. It is an easy corollary of Lemma \ref{bulblemma} that this induces a well defined binary operation on $\pi_0(\Riem^{+}(S^{n}))$.
 
Verifying that the various group axioms hold is mostly straightforward. In particular, it is clear that the class containing the standard round metric, $[ds_{n}^{2}]$, is the identity. The only difficult lies in verifying that each element has an inverse. To see this we briefly return to the notion of psc-concordance. Gajer in \cite{Gajer} shows that the set of concordance classes of $\Riem^{+}(S^{n})$, which we denote $\pi_0^{c}(\Riem^{+}(S^{n}))$, forms a group also under the operation induced by connected sum. We won't reprove it here but it is worth briefly recounting Gajer's method for showing that each concordance class has an inverse, as we will make good use of it. Given a psc-metric $g$ on $S^{n}$ which represents a particular concordance class, equip $S^{n}\times I$ with the standard product $g+dt^{2}$. Let $p\in S^{n}$ be any point. Consider the arc $\{p\}\times I$ in $S^{n}\times I$. Using Lemma \ref{bulblemma} in a slicewise fashion, one can easily adjust the metric $g+dt^{2}$ in a neighbourhood of this arc to obtain a psc-metric $g'+dt^{2}$ so that near $\{p\}\times I$, $g'+dt^{2}=g_{tor}^{n}(\delta)+dt^{2}$ for some $\delta>0$. Recall that $g_{tor}^{n}(\delta)$ is the standard torpedo metric of radius $\delta$ on the disk. Next, we use the Gromov-Lawson method to push out a torpedo cap away from this neighbourhood and preserve positive scalar curvature. By first removing the cap part, then removing the previously constructed ``cylinder of caps" and finally smoothing out the inevitable corners, we are left with a manifold which is topologically a cylinder $S^{n}\times I$ but with a very different metric; see Fig. \ref{concinv}. At one end we have a standard round metric of radius $\delta$. At the other end we have the psc-metric $g\# g^{-1}$ obtained by taking a connected sum of $g$ and $g^{-1}$. Here $g^{-1}$ is isometric to $g$ but via an orientation reversing isometry. In Theorem 2.2 of \cite{Walsh1}, we show in great detail how to adjust a psc-metric on a manifold with boundary in precisely this situation in order to obtain a psc-metric which has a product structure near the boundary. On performing such an adjustment we obtain a psc-concordance between $\delta^{2}ds_{n}^{2}$ and $g\# g^{-1}$ and thus between $ds_{n}^{2}$ and $g\# g^{-1}$. Thus the classes containing $g$ and $g^{-1}$ are inverses in the group $\pi_0^{c}(\Riem^{+}(S^{n}))$. 
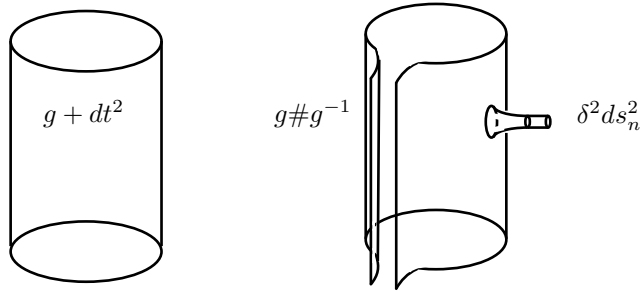
\begin{figure}[!htbp]
\vspace{1.0cm}
\hspace{2.5cm}
\begin{picture}(0,0)
\includegraphics{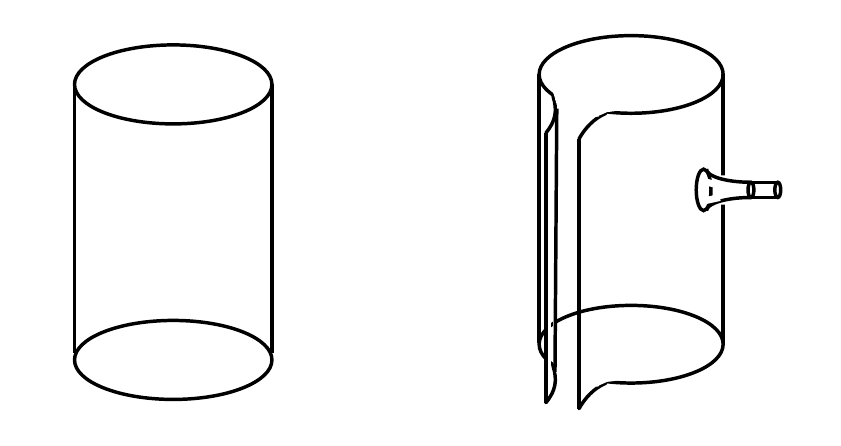}%
\end{picture}
\setlength{\unitlength}{3947sp}%
\begingroup\makeatletter\ifx\SetFigFont\undefined%
\gdef\SetFigFont#1#2#3#4#5{%
  \reset@font\fontsize{#1}{#2pt}%
  \fontfamily{#3}\fontseries{#4}\fontshape{#5}%
  \selectfont}%
\fi\endgroup%
\begin{picture}(5079,1559)(1902,-7227)
\put(2400,-6000){\makebox(0,0)[lb]{\smash{{\SetFigFont{10}{8}{\rmdefault}{\mddefault}{\updefault}{\color[rgb]{0,0,0}$g+dt^{2}$}%
}}}}
\put(3850,-6000){\makebox(0,0)[lb]{\smash{{\SetFigFont{10}{8}{\rmdefault}{\mddefault}{\updefault}{\color[rgb]{0,0,0}$g\# g^{-1}$}%
}}}}
\put(5750,-6006){\makebox(0,0)[lb]{\smash{{\SetFigFont{10}{8}{\rmdefault}{\mddefault}{\updefault}{\color[rgb]{0,0,0}$\delta^{2}ds_{n}^{2}$}%
}}}}
\end{picture}%
\caption{The cylinder $g+dt^{2}$ (left) and the metric which gives rise, after adjustment, to the concordance between $g\# g^{-1}$ and $\delta^{2}ds_{n}^2$ (right)}
\label{concinv}
\end{figure}

To show that $g$ and $g^{-1}$ represent inverse elements in $\pi_{0}(\Riem^{+}(S^{n}))$, we need only show that as well as being psc-concordant, the round metric and the connected sum of $g$ and $g^{-1}$ are also psc-isotopic. That these metrics are indeed psc-isotopic when $n\geq 5$ follows of course from the aforementioned theorem of Botvinnik, Theorem B. of \cite{Botvinnik}.
\end{proof}

\begin{Corollary}\label{pi0}
 For $n=3$ or $n\geq 5$, the set $\pi_0(\Riem_{D_{+}(1)}^{+}(S^{n}))$ is a group under the operation induced by the homotopy product $\mu^{\head}$ defined in \ref{muprod2}.
\end{Corollary}

\begin{proof} The inclusion $\Riem_{D_{+}(1)}^{+}(S^{n})\subset \Riem^{+}(S^{n})$ gives a bijection between $\pi_0(\Riem_{D_{+}(1)}^{+}(S^{n}))$ and  $\pi_0(\Riem^{+}(S^{n}))$. The corollary then follows from the fact that the intermediary metric $g_3$, used in determining the product $\mu^{\head}$, is isotopic to the standard round metric. This means metrics resulting from this product are easily deformed by psc-isotopy to a regular Gromov-Lawson style connected sum, and so the two operations behave in the same way with regard to path components. 
\end{proof}

\subsection{The Loop Space Theorem}\label{Loop4}
By combining the results of the previous sections we can now prove the following theorem, our second main result.

\begin{Theorem}\label{LoopThm} 
For $n=3$ or $n\geq 5$, the space of positive scalar curvature metrics on the $n$-dimensional sphere, $\Riem^{+}(S^{n})$, is weakly homotopy equivalent to an $n$-fold loop space.
\end{Theorem}
\begin{proof} 
We know from Lemma \ref{homequivspaces} and Lemma \ref{bulbdefret} that when $n\geq3$, the spaces $\Riem^{+}(S^{n})$ and $\Riem_{D_{+}(1)}^{+}(S^{n})$ are homotopy equivalent. It is therefore enough to prove the theorem for $\Riem_{D_{+}(1)}^{+}(S^{n})$. Following Theorem \ref{BVM}, we need only show that $\Riem_{D_{+}(1)}^{+}(S^{n})$ is a $\mathbb{D}^{n}$-space and that $\pi_0(\Riem_{D_{+}}^{+}(S^{n}))$ is a group under the operation induced by the $H$-space product $\mu^{\head}$.  The first of these is done when $n\geq 3$ in Lemma \ref{actionlemma}, while the second is done when $n=3$ or when $n\geq 5$ in Corollary \ref{pi0}.
\end{proof}
\begin{Corollary}\label{pathcompcor} 
For $n=3$ or $n\geq 5$, all path components of the space $\Riem^{+}(S^{n})$ are weakly homotopy equivalent.
\end{Corollary}

\bibliographystyle{amsplain}

\end{document}